\documentclass[10-pt,a4paper,leqno]{article}
\DeclareUnicodeCharacter{0302}{\^}
\usepackage[utf8]{inputenc}
\usepackage{xcolor}
\usepackage[ left=3.00cm, right=3.00cm]{geometry}
\usepackage{amsmath}
\usepackage{amsthm}
\usepackage{enumitem}
\usepackage{booktabs}
\usepackage{tabularx}
\usepackage{graphicx}
\usepackage{amsfonts}
\usepackage{amssymb}
\usepackage{graphicx}
\usepackage{placeins}
\usepackage{mathtools}  
\usepackage{bm}    
\usepackage{algorithm}
\usepackage{algpseudocode}   
\usepackage{float}           
\usepackage{titlesec}

\usepackage[style=numeric, maxnames=99, backend=biber]{biblatex}
\usepackage{hyperref}
\usepackage{tikz}
\usetikzlibrary{arrows.meta,calc}

\titleformat{\subsubsection}[runin]
  {\normalfont\bfseries\color{black}}
  {\thesubsubsection.}
  {0.5em}
  {}
  [.\;]

\hypersetup{
    colorlinks=true,
    linkcolor=blue,
    urlcolor=blue,
    citecolor=blue
}
\newtheorem{mydef}{Definition}
\newtheorem{thm}{Theorem}

\newtheorem{rmk}{Remark}

\newtheorem{lemma}{Lemma}
\newtheorem{corollary}{Corollary}
\newtheorem{assumption}{Assumption}
\newtheorem{prop}{Proposition}
\begin{document}
\title{Dynamical Optimal Transport with $\mathfrak{so}(d)$-Invariance: From Theory to Computation}

\providecommand{\Tr}{\textup{Tr}}
  \providecommand{\comment}[2][Unknown author]{%
    \leavevmode
    \noindent\vskip.3cm%
    \fbox{%
      \parbox{0.95\linewidth}{%
        \footnotesize\color{blue}{#1: {#2}}}
      }%
      \vskip.3cm%
      \noindent
    }%

\providecommand{\dx}{\,\mathrm{d}x}
\providecommand{\dy}{\,\mathrm{d}y}
\providecommand{\dt}{\,\mathrm{d}t}
\providecommand{\ds}{\,\mathrm{d}s}
\providecommand{\dz}{\,\mathrm{d}z}
\providecommand{\dmu}{\,\mathrm{d}\mu}
\providecommand{\drho}{\,\mathrm{d}\rho}
\newcommand{\Skew}{\operatorname{skew}}
\providecommand{\Tr}{\mathrm{Tr}}
\newcommand{\tria}{\mathcal{T}}

\providecommand{\argmin}{\mathop{\mathrm{arg\,min}}}
\providecommand{\argmax}{\mathop{\mathrm{arg\,max}}}

\providecommand{\Tr}{\textup{Tr}}

\author{Kevine Meugang Toukam\footnote{Universität Leipzig, Fakultät für Mathematik und Informatik, Augustusplatz 10, 04109 Leipzig,
Germany and Max Planck Institute for Mathematics in the Sciences, 04103 Leipzig, Germany},\  Max von Renesse\footnote{Universität Leipzig, Fakultät für Mathematik und Informatik, Augustusplatz 10, 04109 Leipzig, Germany},\ Johannes Storn\footnote{Universität Leipzig, Fakultät für Mathematik und Informatik, Augustusplatz 10, 04109 Leipzig, Germany}}
    
\date{\today}
\maketitle

    \begin{abstract}
We introduce a modified Benamou--Brenier (MBB) formulation of optimal transport
that incorporates Euclidean invariance at the dynamical level.
We establish existence of minimizers for the resulting variational problem
and prove its equivalence to a static formulation defining the
Procrustes--Wasserstein distance.
In the Gaussian setting, we show that this distance admits a closed-form expression,
reducing to the Euclidean distance between the vectors of square roots of the ordered eigenvalues of the covariance matrices. 
On the computational side, we formulate a primal--dual
scheme for the discretized problem. We prove a local conditional subsequential
convergence result through an abstract analysis of a class of
parameter-dependent saddle-point problems and illustrate the method's performance
numerically.
\end{abstract}

\section{Introduction}
Optimal Transport (OT) provides a principled framework for comparing probability measures
\cite{villani2009optimal,ambrosio2005gradient}. The classical Benamou--Brenier formulation 
\cite{benamou2000computational} reinterprets the quadratic Wasserstein distance as a 
minimum kinetic energy problem: transporting mass between $\mu_0$ and $\mu_1$ is equivalent 
to finding a fluid flow of least action. 
One limitation remains. Classical OT and its dynamical equivalent are 
not invariant under applying an isometry to only one of the measures. Transporting a distribution to a rotated 
or translated copy of itself still incurs positive cost 
\cite{kloeckner2010geometric,ambrosio2005gradient}, which is undesirable in 
symmetry-aware tasks such as shape matching or data alignment \cite{courty2016optimal,arjovsky2017wasserstein}.

A line of research has therefore introduced invariance into OT through static formulations.
Procrustes--Wasserstein distances \cite{grave2019unsupervised,alvarez2019towards} jointly 
optimize over transport plans and orthogonal alignments, with applications in word embedding 
alignment, transfer learning, and shape analysis. More recent work has studied empirical Procrustes--Wasserstein distances and barycenters, combining Sinkhorn solvers \cite{adamo2025depth,cuturi2013sinkhorn,peyre2019computational} 
with iterative orthogonal projections. These methods have been successful in practice. Nevertheless, several aspects remain open: (i) closed-form formulas are lacking, even for simple Gaussian inputs; 
(ii) the dynamical intuition of Benamou--Brenier is lost when invariance is imposed only at the 
static level; (iii) the algorithms typically rely on heuristic alternating minimization and 
lack a PDE-based interpretation.

Connections between optimal transport and control theory have led to dynamical
formulations with controlled state equations
\cite{chen2016optimal,agrachev2013control,agrachev2009optimal}.
In particular, \cite{elamvazhuthi2023dynamical} introduced optimal transport
for nonlinear control-affine systems.
In this setting, the Benamou--Brenier action is generalized by incorporating
controlled dynamics of the form $\partial_t{x}_t = f(x_t,u_t)$, linking
Hamilton--Jacobi equations, Pontryagin-type optimality conditions,
and Wasserstein geometry. From a modeling viewpoint, controlled optimal transport can be approached
in two distinct ways: either by introducing control variables directly
in the kinetic energy (the perspective adopted in the present work) or by
prescribing controlled dynamics in the state evolution.

While these frameworks significantly extend the expressive power of dynamical OT,
they do not explicitly encode invariance under the full Euclidean group $E(d)$.
This naturally raises the question:
Can Euclidean invariance be built directly into the Benamou--Brenier
formulation, rather than imposed externally?
\subsubsection*{Our contributions}
In this work, we take a dynamic perspective on invariant optimal transport.
Rather than enforcing invariance at the level of static couplings, we study the quotient geometry of the Wasserstein space under Euclidean
isometries, derive a rigorous dynamic--static realization and provide a formal tangent-space interpretation through infinitesimal rigid motions Our contributions are as follows:
\begin{enumerate}
  \item \textbf{Quotient geometry of Wasserstein space.}
  We identify the tangent structure of $\mathcal P_2(\mathbb R^d)$ modulo the Euclidean group $E(d)$ and introduce a modified Benamou--Brenier action obtained by orthogonally projecting velocity fields onto the complement of infinitesimal rigid motions.  This yields a well-defined seminorm on tangent vectors that annihilates infinitesimal rotations and translations and provides the corresponding infinitesimal quotient seminorm; its equivalence with the Eulerian formulation is established under the measurable realization condition stated below. 
  \item \textbf{Closed-form Gaussian solution\footnote{A preliminary derivation of the Gaussian result appeared in the preprint \cite{toukam2025procrustes}. In the present work, we incorporate this into a broader dynamical framework and develop a computational approach.}.} In the Gaussian case, we prove that the PW distance reduces to the Euclidean distance between the vectors formed by the square roots of the ordered eigenvalues of the covariance matrices, as announced in \cite{toukam2025procrustes}.
  \item \textbf{Abstract convergence framework.} We develop an abstract conditional subsequential convergence result for an alternating scheme applied to a class of parameter-dependent finite-dimensional saddle-point problems.
  \item \textbf{Numerical scheme and experiments.} We apply this framework to the discretized modified Benamou--Brenier problem using a generalized primal-dual hybrid gradient (PDHG) \cite{chambolle2016ergodic,jacobs2019solving} method and present numerical experiments illustrating the resulting scheme.
\end{enumerate}
\section{A Modified Benamou--Brenier Problem}\label{section3}
In this section, we introduce and motivate our dynamical formulation for optimal transport. In particular, let $\mathcal{P}_2(\mathbb{R}^d)$ denote the space of Borel probability measures on $\mathbb{R}^d$ with finite second moments, endowed with the 2-Wasserstein distance $W_2$. Then we define a dynamical transport problem for given input measures $\mu_0, \mu_1 \in \mathcal{P}_2(\mathbb{R}^d)$ that accounts for global invariance under Euclidean rigid motions. The dynamical formulation is developed for orientation-preserving rigid motions in $SE(d)=SO(d)\ltimes\mathbb R^d$; the orientation-reversing branch is treated separately.
\subsection{Definition of the Modified Benamou–Brenier Problem}
In the following, we propose a modification of the Benamou--Brenier formulation that incorporates global geometric invariance into the objective function in the sense that we do not penalize global rigid motions; only the intrinsic deformation beyond such motions contributes to the cost.
\subsubsection{Euclidean Group and its Infinitesimal Action}Klöckner~\cite[Theorem~1.2]{kloeckner2010geometric} proved that, for $d\geq2$, \[ \operatorname{Isom}\bigl(\mathcal P_2(\mathbb R^d),W_2\bigr) \cong E(d)\ltimes O(d). \] Here the first factor is induced by Euclidean isometries of the underlying space, whereas the second factor fixes all Dirac measures. Since the present formulation quotients out only ambient rigid motions, we restrict attention to the first factor, and dynamically to its identity component
$SE(d)=SO(d)\ltimes\mathbb R^d$. Since our formulation is concerned with rigid motions, we restrict our attention to this subgroup $E(d)$. 
This motivates our focus on $E(d)$ as the natural symmetry group underlying the following MBB formulation. 
We first recall the groups and infinitesimal actions used below.

The group of isometries of the Euclidean space $\mathbb{R}^d$ is the 
Euclidean group, which can be written in terms of rotations and reflections defined by the orthogonal group $O(d)$ and translations $\mathbb{R}^d$ as the semidirect product
\[
E(d) \coloneqq O(d) \ltimes \mathbb{R}^d.
\]
The action of an element $(\theta, A) \in E(d)$ on $x \in \mathbb{R}^d$ is defined by
\[
(\theta, A) x \coloneqq \theta x + A.
\]
The Lie algebra of $E(d)$ at the identity $(\mathrm{id}, 0)$ is skew-symmetric matrices $\mathfrak{so}(d) \coloneqq \lbrace B \in \mathbb{R}^{d\times d}\colon B + B^\top = 0\rbrace$ and translations $\mathbb{R}^d$; that is,
\[
\mathfrak{e}(d) = \mathfrak{so}(d) \ltimes \mathbb{R}^d.
\]
Let $\mathrm{Diff}(\mathbb{R}^d)$ denote the group of orientation-preserving
$C^\infty$-diffeomorphisms of $\mathbb{R}^d$ with smooth inverses, endowed with
composition as the group operation. Its formal Lie algebra at the identity
$\mathrm{id}$ is the space of smooth vector fields
\[
T_{\mathrm{id}}\mathrm{Diff}(\mathbb{R}^d)
\coloneqq  \mathfrak X(\mathbb{R}^d)
\coloneqq  C^\infty(\mathbb{R}^d;\mathbb{R}^d),
\]
with the Lie bracket given by the commutator
$[v,w]=D w\,v - D v\,w$.

At the infinitesimal level, we decompose an Eulerian velocity field into
a deformational component and a rigid component. More precisely, for
$E\in\mathfrak{so}(d)$, $C\in\mathbb R^d$, and a vector field $v$, the
total velocity is
\begin{align}\label{eq:DecompVelo}
    x\mapsto v(x)+Ex+C.
\end{align}
The purely rigid infinitesimal velocities are therefore the affine fields
$x\mapsto Ex+C$.
Its Lie algebra at the identity $(I,0,\mathrm{id})$ for $\mathcal{G}= E(d) \ltimes \mathrm{Diff}(\mathbb{R}^d)$ is
\[
T_{(I,0,\mathrm{id})}\mathcal G
\simeq \mathfrak e(d)\ltimes\mathfrak X(\mathbb{R}^d)
= (\mathfrak{so}(d)\ltimes\mathbb{R}^d)\ltimes C^\infty(\mathbb{R}^d;\mathbb{R}^d).
\]
\paragraph{Infinitesimal Eulerian action.} The infinitesimal action of an element
$(E,C,v)\in \mathfrak e(d)\ltimes\mathfrak X(\mathbb{R}^d)$, where  $E\in\mathfrak{so}(d)$,
$C\in\mathbb{R}^d$,  in Eulerian coordinates induces the velocity field defined in \eqref{eq:DecompVelo}.
In particular, the subspace of ``purely rigid'' infinitesimal actions  corresponds to the fields of the form
$x\mapsto E x + C$ with $(E,C)\in\mathfrak e(d)$ and $v=0$.


\subsubsection{Decomposition of Motion and Action Functional}
Let \( \mu \in C([0,1];\mathcal{P}_2(\mathbb{R}^d))\) describe the motion of particles with total velocity field $\tilde{v}(x) = V(x) + Ex + C$ decomposing according to \eqref{eq:DecompVelo} into $E \in L^2((0,1);\mathfrak{so}(d)) $, $C \in L^2((0,1);\mathbb{R}^d)$, and an instantaneous velocity $V \in L^2_\mu((0,1)\times \mathbb{R}^d;\mathbb{R}^d)$.
Since mass is conserved, the total velocity satisfies the continuity equation
\begin{equation}\label{eq:continuity}
    \partial_t \mu + \nabla \cdot (\mu (V + E x + C )) = 0\quad\text{in }(0,1)\times \mathbb{R}^d.
\end{equation}
We understand this equation in a distributional sense. More precisely, 
\begin{equation*}
\int_0^1 \int_{\mathbb{R}^d} \left( \partial_t \varphi + \nabla \varphi \cdot (V + E x + C ) \right) \mathrm{d}\mu_t \dt = 0\quad\text{for all }\varphi\in C_c^\infty((0,1)\times \mathbb{R}^d).
\end{equation*} 
Since we assign zero cost for rigid motion, the action functional representing the kinetic energy equals
\begin{equation}\label{kineticenergymbb}
\mathcal{J}(\mu, V) \coloneqq  \int_0^1 \int_{\mathbb{R}^d}| V |^2 \dmu \dt.
\end{equation}
This energy functional penalizes only the non-rigid, or deformational, component $V$ from the decomposition \( \tilde{v}(x) = V + E x + C \) derived from the Lie algebra $E(d)\ltimes\mathrm{Diff}(\mathbb{R}^d)$ for masses $\mu$ that are transported by the total velocity $\tilde{v}$, see \eqref{eq:continuity}. 
%
\begin{rmk}[Classical Benamou--Brenier problem]
In contrast to our formulation, the BB problem minimizes the energy
\begin{align*}
    \mathcal{J}_\textup{BB}(\mu, V, E, C) \coloneqq  \int_0^1 \int_{\mathbb{R}^d}| V + Ex + C |^2 \dmu \dt,
\end{align*}
where $(\mu,V,E,C)$ are linked via the constraint in \eqref{eq:continuity}.
\end{rmk}
We use the energy defined in \eqref{kineticenergymbb} to introduce a measure between two probability measures $\mu_0,\mu_1\in \mathcal{P}_2(\mathbb{R}^d)$. In particular, we minimize the energy $\mathcal{J}$ over the set
\begin{align}\label{eq:DefAorig}
\begin{aligned}
      & \mathcal A = \mathcal{A}(\mu_0,\mu_1)\coloneqq \lbrace (\mu,V,E,C) \in C([0,1];\mathcal P_2(\mathbb{R}^d)) \times L^2_\mu((0,1)\times \mathbb{R}^d;\mathbb{R}^d)\\
    & \quad\quad \times  L^2((0,1);\mathfrak{so}(d))\times  L^2((0,1);\mathbb{R}^d) \colon \mu_{t=0}=\mu_0, \mu_{t=1}=\mu_1, \text{ and }\eqref{eq:continuity} \rbrace.
\end{aligned}
\end{align}
\begin{mydef}
    The {Modified Benamou--Brenier (MBB)} problem $\bar{d}\colon \mathcal{P}_2(\mathbb{R}^d) \times \mathcal{P}_2(\mathbb{R}^d) \to [0,\infty)$ between $\mu_0$ and $\mu_1$ is defined as the square root of the minimal value
\begin{equation}
\label{continuousMBB}
\bar{d}^2(\mu_0, \mu_1) \coloneqq \inf_{(\mu, V, E, C) \in \mathcal{A}(\mu_0,\mu_1)} \mathcal{J}(\mu, V).
\end{equation}
\end{mydef}

\subsection{Geometric structure of the MBB problem }
\label{subsec:mbb-geom-structure}

In this subsection we describe, at the level of Eulerian
tangent vectors, how the MBB kinetic energy ``mods out'' infinitesimal rigid
motions of the Euclidean group.

\subsubsection{Rigid velocity fields and orthogonal projection}
\begin{mydef}
\label{def:rigid-fields-mu} We denote by $\mathcal R_{\mathrm{rig}}^2$ the class of Borel vector fields $r\colon(0,1)\times\mathbb R^d\to\mathbb R^d$ for which there exist $E\in L^2((0,1);\mathfrak{so}(d)), \text{ and } C\in L^2((0,1);\mathbb R^d)$ such that $ r_t(x)=E_tx+C_t$  for almost every $t\in(0,1)$ and every $x\in\mathbb R^d$. \end{mydef} Let $\mu\in C([0,1];\mathcal P_2(\mathbb R^d))$ respectively $\mu\in \mathcal P_2(\mathbb R^d)$. Since $\sup_{t\in[0,1]} \int_{\mathbb R^d}|x|^2\,\mathrm d\mu_t(x)<\infty$, every $r\in\mathcal R_{\mathrm{rig}}^2$ belongs to $L^2_\mu((0,1)\times\mathbb R^d;\mathbb R^d)$ respectively $L^2_\mu(\mathbb R^d)$. And we define the space of rigid fields as
\[
\mathcal R_\mu
\coloneqq \Big\{\, r\in L^2_\mu(\mathbb{R}^d) \colon \exists\, (E,C)\in\mathfrak{so}(d)\times\mathbb{R}^d
\text{ with } r(x)=Ex+C \ \mu\text{-a.e.}\Big\}.\]

\begin{lemma}
\label{lem:Rmu-closed}
For every $\mu\in \mathcal P_2(\mathbb{R}^d)$, the space $\mathcal R_\mu$ is a
finite-dimensional closed linear subspace of $L^2_\mu(\mathbb{R}^d)$.
\end{lemma}

\begin{proof}
Since $\mu\in \mathcal P_2(\mathbb{R}^d)$, we have $\int_{\mathbb{R}^d} |x|^2\,\mathrm{d}\mu<\infty$.
Hence for any $E\in\mathfrak{so}(d)$ and $C\in\mathbb{R}^d$,
\[
\int_{\mathbb{R}^d} |Ex+C|^2\,\mathrm{d}\mu(x)
\le 2|E|^2\int_{\mathbb{R}^d} |x|^2\,\mathrm{d}\mu(x) + 2|C|^2 <\infty.
\]
Consequently, the mapping $r\colon x\mapsto Ex+C$ is in $ L^2_\mu(\mathbb{R}^d)$. The linear map
$\Phi_\mu\colon \mathfrak{so}(d)\times\mathbb{R}^d\to L^2_\mu(\mathbb{R}^d)$ with
$\Phi_\mu(E,C)(x)=Ex+C$ has image $\mathcal R_\mu$.
Since the domain is finite-dimensional, the image $\mathcal R_\mu$ is finite-dimensional
and therefore closed in $L^2_\mu(\mathbb{R}^d)$.
\end{proof}

\begin{corollary}
\label{cor:orth-decomp}
For each $\mu\in \mathcal P_2(\mathbb{R}^d)$ there is an orthogonal decomposition
$
L^2_\mu(\mathbb{R}^d) = \mathcal R_\mu \oplus \mathcal R_\mu^\perp
$
in the sense that 
\begin{align*}
  \int_{\mathbb{R}^d} r \cdot v  \, \mathrm{d}\mu = 0\quad\text{for all }r \in\mathcal R_\mu, v \in\mathcal R_\mu^\perp.     
\end{align*}
\end{corollary} The corresponding orthogonal projections are
$P_\mu\colon L^2_\mu\to \mathcal R_\mu^\perp$ and $Q_\mu\coloneqq \mathrm{Id}-P_\mu\colon L^2_\mu\to\mathcal R_\mu$. They are equivalently characterized via the following minimization problem.
\begin{lemma}
\label{lem:projection-min}
Let $\mu\in \mathcal P_2(\mathbb{R}^d)$ and $V\in L^2_\mu(\mathbb{R}^d)$. We have
$
\int_{\mathbb{R}^d} |P_\mu V|^2\,\mathrm{d}\mu
=\min_{r\in\mathcal R_\mu}\int_{\mathbb{R}^d}|V-r|^2\,\mathrm{d}\mu,
$
and the infimum is attained at $r^\star = Q_\mu V$.
\end{lemma}

\begin{proof}
We decompose $V=P_\mu V + Q_\mu V$ with $P_\mu V\in\mathcal R_\mu^\perp$ and
$Q_\mu V\in\mathcal R_\mu$. For any $r\in\mathcal R_\mu$ we have
$Q_\mu V-r\in\mathcal R_\mu$, hence orthogonality gives
\[
\int_{\mathbb{R}^d}|V-r|^2\,\mathrm{d}\mu
=\int_{\mathbb{R}^d}|P_\mu V|^2\,\mathrm{d}\mu+\int_{\mathbb{R}^d}|Q_\mu V-r|^2\, \mathrm{d}\mu
\ge \int_{\mathbb{R}^d}|P_\mu V|^2\,\mathrm{d}\mu.
\]
Equality holds if and only if $r=Q_\mu V$.
\end{proof}
We recall the Eulerian tangent space $T_\mu\mathcal P_2(\mathbb{R}^d)$ and
the Benamou--Brenier norm, cf.~\cite{otto2001geometry, ambrosio2005gradient}.

\begin{mydef}[Tangent space and BB norm]
\label{def:tangent-bb}
For $\mu\in\mathcal P_2(\mathbb{R}^d)$ we define the dual norm
\[
\|\xi\|_{-1,\mu}
\coloneqq \sup\Big\{\langle\xi,\varphi\rangle \colon \varphi\in C_c^\infty(\mathbb{R}^d),\
\int_{\mathbb{R}^d}|\nabla\varphi|^2\,\mathrm{d}\mu\le 1\Big\}.
\]
The tangent space reads 
$ \operatorname{Tan}_\mu\mathcal P_2(\mathbb R^d)
\coloneqq
\overline{\{\nabla\varphi:\varphi\in C_c^\infty(\mathbb R^d)\}}
^{L^2(\mu;\mathbb R^d)} $ or simply identify by $
T_\mu\mathcal P_2(\mathbb{R}^d)
\coloneqq \overline{\big\{-\nabla\!\cdot(\mu V)\colon V\in C_c^\infty(\mathbb{R}^d;\mathbb{R}^d)\big\}}
^{\ \|\cdot\|_{-1,\mu}}
$
and we set for all $\xi\in T_\mu\mathcal P_2(\mathbb{R}^d)$ the norm
\[
\|\xi\|_{\mathrm{BB},\mu}^2
\coloneqq \inf\Big\{\int_{\mathbb{R}^d}|V|^2\,\mathrm{d}\mu\colon V\in L^2_\mu(\mathbb{R}^d),\ -\nabla\!\cdot(\mu V)=\xi\Big\}.
\]
\end{mydef}

\subsubsection{MBB seminorm induced by rigid directions and geometric dynamic action}
We note that if $(\mu_t)_{t\in[0,1]}\subset\mathcal P_2(\mathbb{R}^d)$ satisfies
$ \partial_t\mu_t + \nabla\!\cdot(\mu_t \widetilde V_t)=0$ for some $\widetilde V\in L^2_{\mu_t}(\mathbb{R}^d),$
then for almost every $t\in(0,1)$ one has
\( \partial_t\mu_t \in T_{\mu_t}\mathcal P_2(\mathbb{R}^d), \)
with $\partial_t\mu_t = -\nabla\!\cdot(\mu_t \widetilde V_t)$ in the sense of distributions \cite{ambrosio2005gradient}.

\begin{mydef}[MBB seminorm at $\mu$]
\label{def:mbb-seminorm}
Let $\mu\in\mathcal P_2(\mathbb{R}^d)$ and $\xi\in T_\mu\mathcal P_2(\mathbb{R}^d)$.
We define
\[
\|\xi\|_{\mathrm{MBB},\mu}^2
\coloneqq \inf\Big\{\int_{\mathbb{R}^d}|V-r|^2\,\mathrm{d}\mu\colon V\in L^2_\mu(\mathbb{R}^d),\ r\in\mathcal R_\mu,\
-\nabla\!\cdot(\mu V)=\xi\text{ in }\mathcal D'(\mathbb{R}^d)\Big\}.
\]
\end{mydef}

\begin{lemma}[Projection form]\label{lem:mbb-projection-form}
For every $\mu\in\mathcal P_2(\mathbb R^d)$ and every $\xi\in T_\mu\mathcal P_2(\mathbb R^d)$, one has 
\begin{equation}
\label{eq:mbb-projection-form}
\|\xi\|_{\mathrm{MBB},\mu}^2
=\inf\Big\{\int_{\mathbb{R}^d}|P_\mu V|^2\,\mathrm{d}\mu\colon V\in L^2_\mu(\mathbb{R}^d)\text{ and }-\nabla\!\cdot(\mu V)=\xi\text{ in }\mathcal D'(\mathbb{R}^d)\Big\}.
\end{equation}
In particular, we have $\|\xi\|_{\mathrm{MBB},\mu}\le \|\xi\|_{\mathrm{BB},\mu}$.
\end{lemma}

\begin{proof}
Fix $V$ with $-\nabla\!\cdot(\mu V)=\xi$.
By Lemma~\ref{lem:projection-min},
$ \inf_{r\in\mathcal R_\mu}\int|V-r|^2\,\mathrm{d}\mu = \int_{\mathbb{R}^d}|P_\mu V|^2\,\mathrm{d}\mu. $
Since the constraint involves only $V$, the optimization with respect to $r$ is independent and can therefore be performed first; it yields
\eqref{eq:mbb-projection-form}. The inequality follows by 
\begin{align*}
\int_{\mathbb{R}^d}|P_\mu V|^2\,\mathrm{d}\mu& \leq \int_{\mathbb{R}^d}|V|^2\,\mathrm{d}\mu.\qedhere    
\end{align*}
\end{proof}
\begin{rmk}[Seminorm and quotient structure] The Modified Benamou--Brenier norm defines a seminorm on $T_\mu\mathcal P_2$ with kernel given by infinitesimal Euclidean isometries. It induces a genuine norm only on the quotient tangent space obtained by factoring out rigid motions. \end{rmk}

\begin{mydef}[Geometric dynamic MBB value]
\label{def:dbar-geom}
Let $\mu_0,\mu_1\in\mathcal P_2(\mathbb{R}^d)$. We say that $(\mu,\widetilde V)$ is admissible if
$\mu\in C([0,1];\mathcal P_2(\mathbb{R}^d))$, $\mu_{t=0}=\mu_0$, $\mu_{t=1}=\mu_1$,
and $\widetilde V\in L^2_{\mu_t}((0,1)\times \mathbb{R}^d;\mathbb{R}^d)$ satisfies
\[
\partial_t\mu+\nabla\!\cdot(\mu \widetilde V)=0
\quad\text{in }\mathcal D'((0,1)\times\mathbb{R}^d).
\]
For such $(\mu,\widetilde V)$, we set
$
\mathcal J_{\mathrm{geom}}(\mu,\widetilde V)
\coloneqq \, \int_0^1 \|\partial_t\mu_t\|_{\mathrm{MBB},\mu_t}^2\dt 
$
and define
\[
\bar d_{\mathrm{geom}}^2(\mu_0,\mu_1)
\coloneqq \inf\big\{\mathcal J_{\mathrm{geom}}(\mu,\widetilde V)\colon (\mu,\widetilde V)\ \text{admissible}\big\}.
\]

\end{mydef}

\begin{prop}[Geometric--Eulerian comparison and formal converse]
\label{prop:geom-eulerian-equivalence}
For all $\mu_0,\mu_1\in\mathcal P_2(\mathbb R^d)$,
\[ \bar d_{\mathrm{geom}}^2(\mu_0,\mu_1) \le    \bar d^2(\mu_0,\mu_1). \]
Conversely, equality holds provided that every admissible curve
$\mu\in C([0,1];\mathcal P_2(\mathbb R^d))$ with finite geometric action
admits, for every $\varepsilon>0$, jointly measurable fields
$W:(0,1)\times\mathbb R^d\to\mathbb R^d$ and
\[ r_t(x)=E_tx+C_t, \quad E\in L^2((0,1);\mathfrak{so}(d)), \quad C\in L^2((0,1);\mathbb R^d), \] such that  $-\nabla\cdot(\mu_tW_t)=\partial_t\mu_t$ for almost every $t$ and
\[ \int_0^1\int_{\mathbb R^d}|W_t-r_t|^2\,\mathrm d\mu_t\,\mathrm dt \le \int_0^1 \|\partial_t\mu_t\|_{\mathrm{MBB},\mu_t}^2\,\mathrm dt +\varepsilon.
\]
\end{prop}
%
\begin{proof}
\emph{Step 1 ($\bar d^2\ge \bar d_{\mathrm{geom}}^2$).}
Let $(\mu,V,E,C)\in\mathcal A(\mu_0,\mu_1)$ and define the total velocity
$
\widetilde V \coloneqq V+Ex+C.
$
Thus $(\mu,\widetilde V)$ is according to \eqref{eq:DefAorig} admissible in the sense of Definition~\ref{def:dbar-geom}.
For almost every $t$, the rigid field $r_t(x)\coloneqq E_tx+C_t$ belongs to $\mathcal R_{\mu_t}$.
Hence, Definition~\ref{def:mbb-seminorm} implies
\[
\int_{\mathbb{R}^d} |V_t|^2\,\mathrm{d}\mu_t
=\int_{\mathbb{R}^d} |\widetilde V_t-r_t|^2\,\mathrm{d}\mu_t
\ge \,\|\partial_t\mu_t\|_{\mathrm{MBB},\mu_t}^2.
\]
Integrating over $t\in (0,1)$ yields
$
\mathcal J(\mu,V)\ge \mathcal J_{\mathrm{geom}}(\mu,\widetilde V).
$
Thus, taking the infimum over  all $(\mu,V,E,C)\in \mathcal A(\mu_0,\mu_1)$ gives 
\begin{align*}
\bar d^2(\mu_0,\mu_1)\ge \bar d_{\mathrm{geom}}^2(\mu_0,\mu_1).    
\end{align*}

\emph{Step 2 ($\bar d^2\le \bar d_{\mathrm{geom}}^2$).} Let $\mu\in C([0,1];\mathcal P_2(\mathbb{R}^d))$ satisfy \( \int_0^1 \|\partial_t\mu_t\|_{\mathrm{MBB},\mu_t}^2\dt < \infty. \) Assume the measurable realization property stated in the proposition.
Let $\mu$ be an admissible curve with finite geometric action and fix $\varepsilon>0$. Choose jointly measurable fields $W$ and $r_t(x)=E_tx+C_t$ as in the assumption, and define $V_t\coloneqq W_t-r_t$. Then $\partial_t\mu_t +\nabla\cdot\bigl(\mu_t(V_t+E_tx+C_t)\bigr)=0$, and $(\mu,V,E,C)$ is admissible for the Eulerian MBB problem. Moreover,
\[ \mathcal J(\mu,V) = \int_0^1\int_{\mathbb R^d}|W_t-r_t|^2\,\mathrm d\mu_t\,\mathrm dt
\le \int_0^1 \|\partial_t\mu_t\|_{\mathrm{MBB},\mu_t}^2\,\mathrm dt+\varepsilon.
\]
Taking the infimum over $\mu$ and then letting $\varepsilon\downarrow0$
gives
\[ \bar d^2(\mu_0,\mu_1) \le \bar d_{\mathrm{geom}}^2(\mu_0,\mu_1).
\]
Combining Step 1 and 2 concludes the proof.
\end{proof}
\begin{rmk}
The measurable realization property is not needed in the subsequent dynamic--static equivalence, which is proved independently in Theorem~\ref{thm:dyn-static-full}.
\end{rmk}

\subsection{A Wasserstein-like formulation of the MBB problem and existence of minimizers}\label{proofdynamictostaticMBB}
Throughout this subsection, we use the following notation: For $a\in\mathbb{R}^d$, we denote by $\tau_a(x)\coloneqq x+a$ the translation map. For $(Q,b)\in E(d)\coloneqq O(d)\ltimes\mathbb{R}^d$, we define  $ (Q,b)x\coloneqq Qx+b$ and $(Q,b)_\#\mu \coloneqq  (Q\cdot+b)_\#\mu . $ Moreover, for $\mu\in\mathcal{P}_2(\mathbb{R}^d)$, we denote its mean value by $m(\mu)\coloneqq \int_{\mathbb{R}^d}x\,\mathrm{d}\mu(x)$ and its associated centered distribution $ \bar\mu\coloneqq(\tau_{-m(\mu)})_\#\mu.$ Here, $(T)_\#\mu$ denotes the pushforward of $\mu$ by $T$, defined by $(T)_\#\mu(A)\coloneqq \mu(T^{-1}(A))$ for all Borel sets $A\subset\mathbb{R}^d$.\\
We start with the following auxiliary results.
\begin{lemma}[Rigid flow generated by $(E,C)$]
\label{lem:rigid-flow}
Let $E\in L^2((0,1);\mathfrak{so}(d))$, $C\in L^2((0,1);\mathbb{R}^d)$, and $(Q_0,b_0)\in E(d)$.
Then there exists a unique $(Q,b)\in W^{1,2}((0,1);O(d)\times\mathbb{R}^d)$ such that $Q_t\in O(d)$ for almost every  $t\in (0,1)$ and
\begin{equation}\label{eq:ODE-Qb}
\partial_t Q_t=E_tQ_t,\qquad Q_{t=0}=Q_0,\qquad
\partial_t b_t=E_tb_t+C_t,\qquad b_{t=0}=b_0.
\end{equation}
Moreover, the function $g_t(x)\coloneqq Q_tx+b_t$ is the unique solution to
\[
\partial_t g_t(x)=E_t g_t(x)+C_t\quad\text{and}\quad Dg_t(x)=Q_t.
\]
Its inverse equals $g_t^{-1}(y)=Q_t^\top(y-b_t)$.
\end{lemma}
\begin{proof}
The existence and uniqueness of $(Q,b)\in W^{1,1}((0,1);\mathbb{R}^{d\times d}\times\mathbb{R}^d)$ with \eqref{eq:ODE-Qb} follows from
classical linear ODE theory with integrable coefficients, see for example \cite[Ch.~3, \S3]{coddington1955theory}. Since $E,C\in L^2(0,1)\subset L^1(0,1)$ and $(Q,b)$ is bounded on $(0,1)$ by Grönwall’s inequality, it follows that $(\partial_t Q,\partial_t b)\in L^2(0,1)$ and hence $(Q,b)\in W^{1,2}((0,1))$. To verify the orthogonality of $Q_t$, we define
$M_t\coloneqq Q_t^\top Q_t$. Then $M\in W^{1,1}((0,1);\mathbb{R}^{d\times d})$ satisfies for almost every $t \in (0,1)$ due to the identity $E_t^\top=-E_t$
\[
\partial_t M_t = \partial_t Q_t^\top Q_t + Q_t^\top \partial_t Q_t
= Q_t^\top E_t^\top Q_t + Q_t^\top E_t Q_t
= Q_t^\top(E_t^\top+E_t)Q_t =0.
\]
Hence $M_t\equiv M_0=Q_0^\top Q_0=I$, so that $Q_t\in O(d)$ for almost every $t\in (0,1)$.
The identities for $g_t$ follow directly. Uniqueness is inherited from uniqueness of $(Q,b)$.
\end{proof}
\begin{rmk}
    The evolution equation $\partial_t Q_t=E_tQ_t$ with $E\in L^2((0,1);\mathfrak{so}(d))$ corresponds to the right-trivialized representation of the tangent dynamics on the orthogonal group $O(d)$, since $E_t=\partial_tQ_t Q_t^\top $. Throughout this work, we adopt this convention.
\end{rmk}
\begin{lemma}
\label{lem:pullback-cont}
Let $\mu\in C([0,1];\mathcal P_2(\mathbb{R}^d))$ and
$V\in  L^2_{\mu_t}((0,1)\times \mathbb{R}^d;\mathbb{R}^d)$.
Let $E\in L^1((0,1);\mathfrak{so}(d))$ and $C\in L^1((0,1);\mathbb{R}^d)$ and define
$
r_t(x)\coloneqq E_t x + C_t .
$
Assume that
\begin{equation}\label{eq:cont-with-drift}
\partial_t\mu_t+\nabla\!\cdot\big(\mu_t(V_t+r_t)\big)=0 \quad\text{in }\mathcal D'((0,1)\times\mathbb{R}^d).
\end{equation}
Let $g_t(x)=Q_t x+b_t$ be the unique rigid flow solving
\begin{equation}\label{eq:rigid-flow-ODE}
\partial_t g_t(x)=r_t(g_t(x))=E_t g_t(x)+C_t,\quad g_0(x)=Q_0x +b_0,
\end{equation}
i.e.\ $\partial_t Q_t=E_tQ_t$ and $\partial_t b_t=E_t b_t+C_t$ with $Q_0 \in O(d)$, $b_0\in \mathbb{R}^d$.
Define the pulled-back curve
$
\tilde\mu_t \coloneqq (g_t^{-1})_\#\mu_t,
$
and the pulled-back velocity field
$
\tilde V_t(y) \coloneqq Dg_t^{-1}(g_t(y))\,V_t(g_t(y))
      = Q_t^\top\, V_t(g_t(y)).
$
Then one has
\begin{equation}\label{eq:cont-pulled}
\partial_t\tilde\mu_t+\nabla\!\cdot(\tilde\mu_t\,\tilde V_t)=0 \quad\text{in }\mathcal D'((0,1)\times\mathbb{R}^d).
\end{equation}
\end{lemma}
\begin{proof}
Fix $\varphi\in C_c^\infty((0,1)\times\mathbb{R}^d)$ and set $\psi(t,x)\coloneqq \varphi(t,g_t^{-1}(x)).$

\textit{Step 1 (Time derivative of $g_t^{-1}$).} Differentiation $g_t^{-1}(g_t(y))=y$ with respect to $t$ yields
$
0=\partial_t\big(g_t^{-1}(g_t(y))\big)
=(\partial_t g_t^{-1})(g_t(y)) + Dg_t^{-1}(g_t(y))\,\partial_t g_t(y).
$
Thus, we obtain for $x=g_t(y)$ the identity
\begin{equation}\label{eq:dt-inverse}
(\partial_t g_t^{-1})(x) = -Dg_t^{-1}(x)\,\partial_t g_t(g_t^{-1}(x)).
\end{equation}

\textit{Step 2 (Chain rule for $\psi(t,x)=\varphi(t,g_t^{-1}(x))$).} Although $\psi$ need not be smooth in time, it belongs to
\[ W^{1,1}\bigl((0,1);C_c^1(\mathbb R^d)\bigr)
\] and has spatial support contained in a common compact set. Indeed, $t\mapsto g_t$ and $t\mapsto g_t^{-1}$ are absolutely continuous, while $g_t$ is affine in space. The weak formulation of the continuity equation extends to this class of test functions by time mollification. All terms are integrable because $V\in L_\mu^2$, $E,C\in L^1$, and $\mu_t\in\mathcal P_2(\mathbb R^d)$. Hence $\psi$ is an admissible test function. Let $y=g_t^{-1}(x)$. Then \eqref{eq:dt-inverse} yields almost everywhere the identity
\begin{align}\label{eq:dt-psi}
\begin{aligned}
\partial_t\psi(t,x)
&=\partial_t\varphi(t,y)+\nabla\varphi(t,y)\cdot (\partial_t g_t^{-1})(x)\\
&=\partial_t\varphi(t,y)-\nabla\varphi(t,y)\cdot Dg_t^{-1}(x)\,\partial_t g_t(y)\\
&=\partial_t\varphi(t,y)-\nabla\varphi(t,y)\cdot Q_t^\top\,\partial_t g_t(y).
\end{aligned}
\end{align}
Moreover, we have
\begin{equation}\label{eq:grad-psi}
\nabla_x\psi(t,x)=Dg_t^{-1}(x)^\top \nabla\varphi(t,y) = (Q_t^\top)^\top \nabla\varphi(t,y) = Q_t\,\nabla\varphi(t,y).
\end{equation}
Using \eqref{eq:cont-with-drift} in weak form,
$0=\int_0^1\!\!\int_{\mathbb{R}^d}\Big(\partial_t\psi(t,x)
+\nabla\psi(t,x)\cdot\big(V_t(x)+r_t(x)\big)\Big)\,\mathrm{d}\mu_t(x)\dt.$
Plugging \eqref{eq:dt-psi}--\eqref{eq:grad-psi} and writing $y=g_t^{-1}(x)$ yields
\begin{align*}
0 &=\int_0^1\!\!\int_{\mathbb{R}^d}\Big(\partial_t\varphi(t,y)-\nabla\varphi(t,y)\cdot Q_t^\top\partial_t g_t(y) +Q_t\nabla\varphi(t,y)\cdot V_t(x) +Q_t\nabla\varphi(t,y)\cdot r_t(x)\Big)\,\mathrm{d}\mu_t(x)\dt.
\end{align*}
Now $x=g_t(y)$ and by \eqref{eq:rigid-flow-ODE},
$
r_t(x)=r_t(g_t(y))=\partial_t g_t(y).
$
Therefore
\[
Q_t\nabla\varphi(t,y)\cdot r_t(x)=Q_t\nabla\varphi(t,y)\cdot\partial_t g_t(y)
=\nabla\varphi(t,y)\cdot Q_t^\top\partial_t g_t(y)
\]
which cancels exactly the term $-\nabla\varphi(t,y)\cdot Q_t^\top\partial_t g_t(y)$.
Hence, we have
$$
0=\int_0^1\!\!\int_{\mathbb{R}^d}\Big(
\partial_t\varphi(t,y)
+\nabla\varphi(t,y)\cdot\big(Q_t^\top V_t(g_t(y))\big)
\Big)\,\mathrm{d}\mu_t(g_t(y)) \dt.
$$

\textit{Step 3} Since $\tilde\mu_t=(g_t^{-1})_\#\mu_t$, the change of variables for the pushforward measure gives
$\int_{\mathbb{R}^d} f(y)\,d\tilde\mu_t(y)=\int_{\mathbb{R}^d} f(g_t^{-1}(x))\,\mathrm{d}\mu_t(x)
=\int_{\mathbb{R}^d} f(y)\,\mathrm{d}\mu_t(g_t(y))$ for every bounded measurable function $f$. Applying this identity to the last equation of Step 2 together with $\tilde V_t(y)\coloneqq Q_t^\top V_t(g_t(y)),$ we obtain 
$
0=\int_0^1\!\!\int_{\mathbb{R}^d}\Big(
\partial_t\varphi(t,y)
+\nabla\varphi(t,y)\cdot \tilde V_t(y)
\Big)\,d\tilde\mu_t(y)\dt,$  
which is the weak formulation of \eqref{eq:cont-pulled}.
\end{proof}

\begin{lemma}
\label{lem:energy-inv}
Under the assumptions of Lemma~\ref{lem:pullback-cont} we have the identity
$$
\int_0^1\int_{\mathbb{R}^d}|\tilde V_t|^2\,d\tilde\mu_t\dt
=\int_0^1\int_{\mathbb{R}^d}|V_t|^2\,\mathrm{d}\mu_t\dt.
$$
\end{lemma}

\begin{proof}
For almost every $t$,
$
\int_{\mathbb{R}^d}|\tilde V_t(y)|^2\,d\tilde\mu_t(y)
=\int_{\mathbb{R}^d}\big|Q_t^\top V_t(g_t(y))\big|^2\,d\tilde\mu_t(y).
$
Since $\tilde\mu_t=(g_t^{-1})_\#\mu_t$, we have
$
\int_{\mathbb{R}^d}\big|Q_t^\top V_t(g_t(y))\big|^2\,d\tilde\mu_t(y)
=\int_{\mathbb{R}^d}\big|Q_t^\top V_t(x)\big|^2\,\mathrm{d}\mu_t(x).
$\\
Because $Q_t\in O(d)$, $|Q_t^\top u|=|u|$, hence
$
\int_{\mathbb{R}^d}\big|Q_t^\top V_t(x)\big|^2\,\mathrm{d}\mu_t(x)=\int_{\mathbb{R}^d}|V_t(x)|^2\,\mathrm{d}\mu_t(x).
$
Integrate in $t$ to get the second identity.
\end{proof}

\begin{mydef}[Static Procrustes--Wasserstein value over $E(d)$]
\label{def:PW-Ed}
For $\mu_0,\mu_1\in\mathcal P_2(\mathbb{R}^d)$ define
\[
d_{\mathrm{PW}}^2(\mu_0,\mu_1)
\coloneqq \inf_{(Q,b)\in E(d)} W_2^2\big(\mu_0,(Q,b)_\#\mu_1\big).
\]
\end{mydef}
\begin{thm}[Dynamic--static equivalence]
\label{thm:dyn-static-full}
We denote the special Euclidean group by $SE(d)= SO(d)\ltimes \mathbb{R}^d $. For all $\mu_0,\mu_1\in\mathcal{P}_2(\mathbb{R}^d)$,
\[
\bar d^2(\mu_0,\mu_1)=\inf_{(Q,b)\in SE(d)} W_2^2\big(\mu_0,(Q,b)_\#\mu_1\big)= d_S^2(\mu_0,\mu_1) .
\]
Moreover, fixing any reflection $J\in O(d)$ with $\det (J)=-1$, the full Procrustes--Wasserstein value satisfies \[ d_{\mathrm{PW}}^2(\mu_0,\mu_1) = \min \big( \bar d^2(\mu_0,\mu_1), \, \bar d^2(\mu_0,J_\#\mu_1) \big).\]
\end{thm}

\begin{proof}
\noindent{I. Proof of $\bar d^2(\mu_0,\mu_1)\ge d_{S}^2(\mu_0,\mu_1)$.}
Let $(\mu,V,E,C)$ be admissible, that is, $(\mu,V,E,C)\in\mathcal A(\mu_0,\mu_1)$.
Define $r_t(x)\coloneqq E_tx+C_t$.
Let $g_t(x)=Q_tx+b_t$ be the rigid flow with $g_0=\mathrm{Id}$ generated by $(E,C)$ via Lemma~\ref{lem:rigid-flow}, so $\partial_t g_t(x)=r_t(g_t(x))$ and $g_0=\mathrm{Id}$.
Set
\[
\tilde\mu_t\coloneqq (g_t^{-1})_\#\mu_t,
\quad
\tilde V_t(y)\coloneqq Q_t^\top V_t(g_t(y)).
\]
By Lemma~\ref{lem:pullback-cont},
\begin{equation}\label{eq:BB-cont-from-MBB}
\partial_t\tilde\mu_t+\nabla\!\cdot(\tilde\mu_t\tilde V_t)=0
\quad\text{in }\mathcal{D}'((0,1)\times\mathbb{R}^d),
\end{equation}
and the endpoints are
$ \tilde\mu_0=(g_0^{-1})_\#\mu_0=\mu_0,
\quad
\tilde\mu_1=(g_1^{-1})_\#\mu_1. $
Let $h\coloneqq (g_1^{-1})\in SE(d)$. By Lemma~\ref{lem:energy-inv},
$ \mathcal{J}[\mu,V,E,C]
=\int_0^1\int_{\mathbb{R}^d}|\tilde V_t(y)|^2\,d\tilde\mu_t(y)\dt. $
By the Benamou--Brenier dynamical characterization of $W_2^2$ on $\mathcal P_2(\mathbb{R}^d)$,
applied to the admissible pair $(\tilde\mu,\tilde V)$ in \eqref{eq:BB-cont-from-MBB},
\[
W_2^2(\mu_0,h_\#\mu_1)
\le \int_0^1\int_{\mathbb{R}^d}|\tilde V_t|^2\,d\tilde\mu_t\dt
=\mathcal{J}[\mu,V,E,C].
\]
Hence, for every admissible $(\mu,V,E,C)$,
\[
d_{S}^2(\mu_0,\mu_1)
=\inf_{(Q,b)\in SE(d)}W_2^2(\mu_0,(Q,b)_\#\mu_1)
\le W_2^2(\mu_0,h_\#\mu_1)
\le \mathcal{J}[\mu,V,E,C].
\]
Taking the infimum over all admissible quadruples gives
$ d_{S}^2(\mu_0,\mu_1)\le \bar d^2(\mu_0,\mu_1). $

\medskip
\noindent{II. Proof of $\bar d^2(\mu_0,\mu_1)\le d_{S}^2(\mu_0,\mu_1)$.}
Fix $(Q,b)\in SE(d)$. Let $(\hat\mu,\hat W)$ be a Benamou--Brenier minimizer between
$\mu_0$ and $(Q,b)_\#\mu_1$, i.e.
\begin{equation}\label{eq:BB-min}
\partial_t\hat\mu_t+\nabla\!\cdot(\hat\mu_t\hat W_t)=0,\quad
\hat\mu_0=\mu_0,\quad \hat\mu_1=(Q,b)_\#\mu_1,
\end{equation}
and
\[
\int_0^1\int_{\mathbb{R}^d}|\hat W_t(y)|^2\,d\hat\mu_t(y)\dt
=W_2^2\big(\mu_0,(Q,b)_\#\mu_1\big).
\]
Fix $(Q_t,b_t)\in W^{1,2}(0,1;SO(d)\times\mathbb{R}^d)$ with $(Q_0,b_0)=(I,0)$ and $(Q_1,b_1)=(Q,b)^{-1}$, and set $g_t(x)\coloneqq Q_tx+b_t$. 
Define
\( \mu_t\coloneqq (g_t)_\#\hat\mu_t, \quad V_t(x)\coloneqq Q_t\,\hat W_t(g_t^{-1}(x)).\)
Define 
\( E_t\coloneqq \partial_t Q_tQ_t^\top\in\mathfrak{so}(d), \quad C_t\coloneqq \partial_t b_t-E_tb_t\in\mathbb{R}^d, \quad r_t(x)\coloneqq E_tx+C_t. \) Then $\partial_t Q_t=E_tQ_t$ and $\partial_t b_t=E_tb_t+C_t$, hence $\partial_t g_t(x)=E_tg_t(x)+C_t=r_t(g_t(x))$.
Reversing the change of variables argument of Lemma~\ref{lem:pullback-cont} to $(\mu,V,E,C)$ and using that $(g_t^{-1})_\#\mu_t=\hat\mu_t$,
we obtain that $(\mu,V,E,C)$ satisfies
\[
\partial_t\mu_t+\nabla\!\cdot\big(\mu_t(V_t+r_t)\big)=0
\quad\text{in }\mathcal D'((0,1)\times\mathbb{R}^d),
\quad \mu_{t=0}=\mu_0,\, \mu_{t=1} =(Q,b)^{-1}_\#(Q,b)_\#\mu_1 =\mu_1,
\]
so $(\mu,V,E,C) \in  \mathcal{A}(\mu_0,\mu_1)$. 
By Lemma~\ref{lem:energy-inv} ,
\[
\int_0^1\int_{\mathbb{R}^d}|V_t(x)|^2\,\mathrm{d}\mu_t(x)\dt =
\int_0^1\int_{\mathbb{R}^d}|\hat W_t(y)|^2\,d\hat\mu_t(y)\dt =
W_2^2\big(\mu_0,(Q,b)_\#\mu_1\big).
\]
Therefore $ \bar d^2(\mu_0,\mu_1) \le \mathcal{J}[\mu,V,E,C]
= W_2^2\big(\mu_0,(Q,b)_\#\mu_1\big).
$
Taking the infimum over $(Q,b)\in SE(d)$ yields
$ \bar d^2(\mu_0,\mu_1)\le d_{S}^2(\mu_0,\mu_1). $
\medskip
From I and II, we have $ \bar d^2(\mu_0,\mu_1)=d_{S}^2(\mu_0,\mu_1). $
Finally, fix a reflection \(J\in O(d)\) with \(\det J=-1\). Since $O(d)=SO(d)\,\dot\cup\,SO(d)J$, we obtain 
\begin{align*}
d_{\mathrm{PW}}^2(\mu_0,\mu_1) &=
\min\Big\{ \inf_{\substack{R\in SO(d), b\in\mathbb R^d}} W_2^2\bigl(\mu_0,(R,b)_\#\mu_1\bigr), \inf_{\substack{R\in SO(d), b\in\mathbb R^d}} W_2^2\bigl(\mu_0,(R,b)_\#J_\#\mu_1\bigr)
\Big\}\\
&= \min\left\{ \bar d^2(\mu_0,\mu_1), \bar d^2\bigl(\mu_0,J_\#\mu_1\bigr) \right\}.
\end{align*} 
This concludes the proof.
\end{proof}

\begin{lemma}\label{prop:centered-full}
   Let $\mu_0$, $\mu_1\in P_2(\mathbb{R}^d)$. The Modified Benamou--Brenier problem \eqref{continuousMBB} admits the static Wasserstein formulation 
    \begin{equation} \label{staticMBB}
    \bar{d}^2(\mu_0,\mu_1)= \inf_{Q\in SO(d)}W_2^2\big(\bar{\mu_0},Q_\#\bar{\mu_1}\big). 
    \end{equation}
   For each fixed $Q\in SO(d)$, the unique minimizer with respect to
$b\in\mathbb R^d$ is
\[
b_Q^\star\coloneqq m(\mu_0)-Qm(\mu_1).
\]
This will be called the orientation-preserving Procrustes Wasserstein pseudometric.
\end{lemma}
    \begin{rmk}
        While the above expression resembles a term in the splitting formula given in \cite[Theorem 1.7]{burger2025covariance}, it is important to note that the meaning of the symbols differs, and therefore the expressions are not equivalent.
    \end{rmk} 
\begin{proof}[Proof of Lemma~\ref{prop:centered-full}]
By Theorem \ref{thm:dyn-static-full}, we have the following equality $$\bar d^{\,2}(\mu_0,\mu_1) = \inf_{Q\in SO(d)}
\inf_{b\in\mathbb R^d} \inf_{\pi\in\Pi(\mu_0,\mu_1)}
\int_{\mathbb R^d\times\mathbb R^d}
|x-Qy-b|^2\,d\pi(x,y).$$
Since $\int_{\mathbb R^d\times\mathbb R^d}(x-Qy)\,d\pi(x,y) =m(\mu_0)-Qm(\mu_1)$, we have, for every \(\pi\in\Pi(\mu_0,\mu_1)\),
\begin{align*}
\int_{\mathbb R^d\times\mathbb R^d}|x-Qy-b|^2\,d\pi(x,y)
&=
\int_{\mathbb R^d\times\mathbb R^d}
\bigl|(x-m(\mu_0))-Q(y-m(\mu_1))\bigr|^2\,d\pi(x,y)
+\bigl|m(\mu_0)-Qm(\mu_1)-b\bigr|^2.
\end{align*}
Hence the unique minimizer with respect to \(b\) is $b_Q^\ast=m(\mu_0)-Qm(\mu_1)$. Moreover, centering induces a one-to-one correspondence between \(\Pi(\mu_0,\mu_1)\) and \(\Pi(\bar\mu_0,\bar\mu_1)\). Therefore,
\[ \bar d^{\,2}(\mu_0,\mu_1) = \inf_{Q\in SO(d)} \inf_{\pi\in\Pi(\bar\mu_0,\bar\mu_1)} \int_{\mathbb R^d\times\mathbb R^d}|x-Qy|^2\,d\pi(x,y)= \inf_{Q\in SO(d)}
W_2^2\bigl(\bar\mu_0,Q_{\#}\bar\mu_1\bigr).
\]
This concludes the proof.
\end{proof}

\begin{thm}[Existence of static and dynamic minimizers]
\label{thm:existence-full}
Let $\mu_0,\mu_1\in\mathcal{P}_2(\mathbb{R}^d)$. We have the following:
\begin{enumerate}[noitemsep]
\item There exists $(Q^\star,b^\star)\in SE(d)$ such that
$
d_{S}^2(\mu_0,\mu_1)=W_2^2\big(\mu_0,(Q^\star,b^\star)_\#\mu_1\big).
$
\item There exists an admissible $(\mu,V,E,C)$ attaining the orientation-preserving dynamic value $\bar d^2(\mu_0,\mu_1)$:
$
\bar d^2(\mu_0,\mu_1)=\mathcal{J}(\mu,V,E,C).
$
\end{enumerate}
\end{thm}

\begin{proof}
\textit{(1) Existence of a static minimizer.}
By the translation reduction in the proof of  Lemma~\ref{prop:centered-full}, it suffices to minimize over $Q\in SO(d)$
$
\Phi(Q)\coloneqq W_2^2\big(\bar{\mu_0},Q_\#\bar{\mu_1}\big).
$
Since $SO(d)$ is compact, it is enough to show that $\Phi$ is continuous.

Let $Q_n\to Q$ in $SO(d)$, and set $\nu_n\coloneqq Q_{n\#}\bar{\mu_1}$ and $\nu\coloneqq Q_\#\bar{\mu_1}$.
For any bounded continuous $f$,
$
\int f\,d\nu_n=\int f(Q_nx)\,d\bar{\mu_1}(x)\to \int f(Qx)\,d\bar{\mu_1}(x)=\int f\,d\nu,
$
so $\nu_n\rightharpoonup \nu$ narrowly.
Moreover,
$
\int_{\mathbb{R}^d}|y|^2\,d\nu_n(y)=\int_{\mathbb{R}^d}|Q_nx|^2\,d\bar{\mu_1}(x)=\int_{\mathbb{R}^d}|x|^2\,d\bar{\mu_1}(x),
$
so $\{\nu_n\}$ has constant second moment and $\nu_n\to\nu$ in $\mathcal P_2(\mathbb{R}^d)$.
By stability of $W_2$ under $\mathcal P_2$ convergence,
\[
W_2(\bar{\mu_0},\nu_n)\to W_2(\bar{\mu_0},\nu),
\]
hence $\Phi(Q_n)\to \Phi(Q)$, i.e.\ $\Phi$ is continuous. Therefore $\Phi$
attains its minimum at some $Q^\star\in SO(d)$.
Define $b^\star\coloneqq m(\mu_0)-Q^\star m(\mu_1)$. Then we obtain by Definition~\ref{def:PW-Ed} that
$$
d_{S}^2(\mu_0,\mu_1)=W_2^2\big(\mu_0,(Q^\star,b^\star)_\#\mu_1\big).
$$

\textit{(2) Existence of a dynamic minimizer.} We now construct a minimizer for the dynamical formulation starting from the optimizer $(Q^\star,b^\star)$.
Let $(Q^\star,b^\star)\in SE(d)$ be as above and let $(\hat\mu,\hat W)$ be a Benamou--Brenier
minimizer between $\mu_0$ and $(Q^\star,b^\star)_\#\mu_1$, so \eqref{eq:BB-min} holds and
\[
\int_0^1\int_{\mathbb{R}^d}|\hat W_t(y)|^2\,d\hat\mu_t(y)\dt
=W_2^2\big(\mu_0,(Q^\star,b^\star)_\#\mu_1\big)
=d_{S}^2(\mu_0,\mu_1).
\]
Choose $(Q_t,b_t)\in W^{1,2}(0,1;SO(d)\times\mathbb{R}^d)$ with $(Q_0,b_0) = (I,0)$ and
$(Q_1,b_1)=(Q^\star,b^\star)^{-1}$, and set $g_t(x)\coloneqq Q_tx+b_t$.
Define
\[
\mu_t\coloneqq (g_t)_\#\hat\mu_t,\quad
V_t(x)\coloneqq Q_t \hat W_t(g_t^{-1}(x)),
\quad
E_t\coloneqq \partial_t Q_tQ_t^\top,\quad C_t\coloneqq \partial_t b_t-E_tb_t.
\]
Then $E\in L^2(0,1;\mathfrak{so}(d))$ and $C\in L^2(0,1;\mathbb{R}^d)$ and, as in the proof of
Theorem~\ref{thm:dyn-static-full}(II), the quadruple $(\mu,V,E,C)$ satisfies
\eqref{eq:continuity} with endpoints $\mu_0,\mu_1$, hence is admissible.
By Lemma~\ref{lem:energy-inv},
\[
\mathcal{J}[\mu,V,E,C]
=\int_0^1\int_{\mathbb{R}^d}|V_t|^2\,\mathrm{d}\mu_t\dt
=\int_0^1\int_{\mathbb{R}^d}|\hat W_t|^2\,d\hat\mu_t\dt
=d_{S}^2(\mu_0,\mu_1).
\]
Theorem~\ref{thm:dyn-static-full} yields
$
d_{S}^2(\mu_0,\mu_1)=\bar d^2(\mu_0,\mu_1),
$
and consequently 
\[
\bar d^2(\mu_0,\mu_1)\le \mathcal{J}[\mu,V,E,C]
=d_{S}^2(\mu_0,\mu_1)=\bar d^2(\mu_0,\mu_1).
\]
Hence equality holds, and $(\mu,V,E,C)$ is a minimizer.
\end{proof}

\section{Study of the static problem}

By Section~\ref{proofdynamictostaticMBB}, solving
\eqref{continuousMBB} amounts to finding an optimal orientation-preserving rigid motion in $SE(d)$. After eliminating translations by centering, this reduces to an optimization over $SO(d)$. The corresponding formulation over
the full orthogonal group $O(d)$ yields the full Procrustes--Wasserstein value. In other words, before transporting mass, one must first
choose the best representative of $\mu_0$ within its equivalence
class under isometries.
\begin{rmk}[On translations versus orthogonal transformations]
Although the dynamic formulation of the MBB problem is invariant under the full Euclidean group $E(d)$, the translation component can be eliminated: the optimal translation $b_Q^\star=m(\mu_0)-Qm(\mu_1)$ aligns the means of the two measures, and thus the problem reduces to centered measures (cf. Section~\ref{proofdynamictostaticMBB}). 
As a result, the static Procrustes--Wasserstein formulation only requires optimization over the orthogonal group $O(d)$, not the full Euclidean group. 
Hence, from this point onward we restrict to $O(d)$.
\end{rmk}

Define an equivalence relation \(\sim\) on \(\mathcal{P}_{2,0}(\mathbb{R}^d)\), with $ \mathcal{P}_{2,0}(\mathbb{R}^d)= \{\mu \in \mathcal{P}_{2}(\mathbb{R}^d): m(\mu)=0 \}$ by setting \(\mu \sim \nu\) if and only if there exists an orthogonal transformation \(O \in O(d)\) such that \(\nu = O_\# \mu\), where \(O_\# \mu\) denotes the pushforward measure. Let \(\pi: \mathcal{P}_{2,0}(\mathbb{R}^d) \to \mathcal{P}_{2,0}(\mathbb{R}^d) / \sim\) be the canonical projection.

Define the pseudometric \(\bar{d}(\mu, \nu) = \inf_{O \in O(d)} W_2(\mu, O_\# \nu)\). The quotient space \(\mathcal{P}_{2,0}(\mathbb{R}^d) / \sim\) is equipped with the induced metric \({d}'\), defined by
\begin{eqnarray}
 d'([\mu],[\nu]) = \bar{d}(\mu, \nu)  \label{quotientmetric} = \inf_{O \in O(d)} W_2(\mu, O_\# \nu),
\end{eqnarray}
where \([\mu]\) denotes the equivalence class of \(\mu\).
The space \( \mathcal{P}_{2,0}(\mathbb{R}^d)  \) endowed with the \( \bar{d} \) defined in \eqref{staticMBB} is a pseudo-metric space.

\subsection{Topological properties}
The following theorem is a direct consequence of properties of the Wasserstein space.
\begin{thm}
    The space $\left( \mathcal{P}_{2,0}(\mathbb{R}^d)/ O(d), \bar{d} \right)$ is a complete and path-connected metric space.
\end{thm}
\begin{proof}
Since $ \mathcal{P}_{2,0}(\mathbb{R}^d)/ O(d)$ is a closed subset of $ \mathcal{P}_{2}(\mathbb{R}^d)/ O(d)$, it is enough to do the proof for $ \mathcal{P}_{2}(\mathbb{R}^d)/ O(d)$.
The space $(\mathcal P_2(\mathbb R^d),W_2)$ is complete and geodesic 
\cite[Thm.~7.1.5]{ambrosio2005gradient}, \cite[Ch.~7]{villani2009optimal}. 
The quotient of a complete metric space by a compact group acting by isometries 
is again complete \cite[Ch.~3]{burago2001course}. Moreover, the projection of Wasserstein 
geodesics yields continuous paths between any two orbits. Hence 
$(\mathcal P_2/O(d),d')$ is complete and path-connected.
\end{proof}
We next characterize minimizing curves in the Procrustes Wasserstein setting. 
The following result refines existing ideas about Wasserstein distances and quotient spaces by focusing on minimizing curves \cite{kloeckner2010geometric, burago2001course, do1992riemannian, gallot2004differential}.
\begin{corollary}\label{Bureprocrustegeodesic}
Let $G\in\lbrace SO(d),O(d)\rbrace$. The following properties hold.
\begin{enumerate}
\item For every rectifiable curve
$c\colon[0,1]\to\mathcal P_{2,0}(\mathbb R^d)$,
\[
L(\pi_G\circ c)\leq L(c).
\]

\item Let $Q^\star\in G$ satisfy
\[
d_G([\mu],[\nu])=W_2(\mu,Q^\star_\#\nu),
\]
and let $c$ be a $W_2$-geodesic from $\mu$ to $Q^\star_\#\nu$. Then
\[
L(\pi_G\circ c)
=
L(c)
=
W_2(\mu,Q^\star_\#\nu)
=
d_G([\mu],[\nu]).
\]
Consequently, $\pi_G\circ c$ is a minimizing geodesic in
$\mathcal P_{2,0}(\mathbb R^d)/G$. In particular $(\mathcal P_{2,0}(\mathbb R^d)/G, d_G)$ is a geodesic metric space.
\end{enumerate}
\end{corollary}
\begin{proof}
Since $\pi_G$ is $1$-Lipschitz, the length of every rectifiable curve
$c$ satisfies $L(\pi_G\circ c)\le L(c)$.
Now let $Q^\star$ be optimal and let $c$ be a constant-speed
$W_2$-geodesic from $\mu$ to $Q^\star_\#\nu$. Then
\[
L(c)=W_2(\mu,Q^\star_\#\nu)=d_G([\mu],[\nu]).
\]
On the other hand, every curve joining $[\mu]$ and $[\nu]$ has length at
least $d_G([\mu],[\nu])$. Hence
\[
d_G([\mu],[\nu])
\le L(\pi_G\circ c)
\le L(c)
=d_G([\mu],[\nu]).
\]
Thus equality holds throughout, and $\pi_G\circ c$ is a minimizing
geodesic.
\end{proof}
\subsection{Closed Form for Gaussian Distributions: A Bures-Wasserstein like Formulation}
Given $\mu_i = \mathcal{N}(m_i, \Sigma_i)$, $i \in \{0, 1\}$, two Gaussian distributions on $\mathbb{R}^d$, the 2-Wasserstein distance $W_2$ between $\mu_0$ and $\mu_1$ has a closed-form expression, often referred to as the Bures--Wasserstein formula \cite[section 1]{dowson1982frechet}, which can be written as:
\begin{equation}\label{burgmetric}
W_2^2(\mu_0, \mu_1) = \|m_0 - m_1\|^2 + \mathrm{tr}\left(\Sigma_0 + \Sigma_1 - 2\left(\Sigma_1^{1/2}\Sigma_0\Sigma_1^{1/2}\right)^{1/2}\right),
\end{equation}
where, for every symmetric positive semidefinite matrix $M$, the matrix $M^{1/2}$ denotes the unique positive semidefinite square root and $\|m_0 - m_1\|$ stands for the Euclidean norm between $m_0$ and $m_1$.
\\
Given $\mu_i = \mathcal{N}(m_i, \Sigma_i)$, $i \in \{0, 1\}$, two Gaussian distributions on $\mathbb{R}^d$, \eqref{burgmetric} provides the 2-Wasserstein distance $W_2$ between $\mu_0$ and $\mu_1$. As it only requires the expressions of the two covariance matrices, we only need one of $\bar{\mu_0}\circ(\theta)$ and $ \bar{\mu_1}$. That is, respectively $\bar{\Sigma}_0= \theta\Sigma_0 \theta^T$ and $\bar{\Sigma_1}= \Sigma_1$. 
Hence, the problem reduces with $F(\theta) \coloneqq \mathrm{tr}\big( (\Sigma_1^{1/2} \theta\Sigma_0 \theta^T \Sigma_1^{1/2})^{1/2}\big)$ to 
\begin{align}\label{modifiedgaussian}
\begin{aligned}
\bar{d}^2(\mu_0,\mu_1)&= \inf_{\theta} \left\{  \mathrm{tr}\left(\Sigma_1 + \theta\Sigma_0 \theta^T - 2(\Sigma_1^{1/2} \theta\Sigma_0 \theta^T \Sigma_1^{1/2})^{1/2}\right) \right\} \\
&= \mathrm{tr}(\Sigma_1 ) + \mathrm{tr}(\Sigma_0 ) -2  \sup_{ \theta}\{ F(\theta)\}.    
\end{aligned}
\end{align}
The following theorem provides an explicit formula for the Procrustes Wasserstein distance between two Gaussian distributions.
\begin{thm}\label{gaussianprowass}
    Given $\mu_i = \mathcal{N}(m_i, \Sigma_i)$, $i \in \{0, 1\}$, two Gaussian distributions in $\mathbb{R}^d$, we represent their eigendecompositions by $\Sigma_i = P_i A_i {P_i}^T$ for $i \in \{0, 1\}$, with $A_i$ being a diagonal matrix with diagonal vector $a_i$, the Procrustes Wasserstein distance between $\mu_0$ and $\mu_1 $ is equal to the Euclidean distance between the vectors of square roots of the ordered eigenvalues of the two covariance matrices. Namely, given vector $a_i$ the vector \(a_i=(a_{i,1},\ldots,a_{i,d})\) consists of the ordered eigenvalues of \(\Sigma_i\), with \(a_{i,1}\leq\cdots\leq a_{i,d}\), we have 
    \begin{eqnarray*}
        \bar{d}^2(\mu_0,\mu_1)= \|\sqrt{a_0} - \sqrt{a_1}\|^2.
    \end{eqnarray*}
    The optimal orthogonal transformation is given by $ P_1P_0^T $ and the optimal Monge map is the one between $\Sigma_1 $ and $  P_1P_0^T \Sigma_0 P_0 P_1^T  $ provided $\Sigma_0$ is positive definite.
 \end{thm}
  A preliminary version of this result appeared in \cite{toukam2025procrustes}.

\begin{proof}
Consider  $\mu_i = \mathcal{N}(m_i, \Sigma_i)$, $i \in \{0, 1\}$, two Gaussian distributions on $\mathbb{R}^d$. We represent their eigendecomposition by $\Sigma_i = P_i A_i {P_i}^T$ for $i \in \{0, 1\}$, with $A_i$ being a diagonal matrix with diagonal vector $a_i$. The vector \(a_i=(a_{i,1},\ldots,a_{i,d})\) consists of the ordered eigenvalues of \(\Sigma_i\) and $P_i$ is an orthogonal matrix.
    \begin{lemma}\label{setofeigenvalue}
        Let $ M$, $N$ be positive semidefinite matrices. If $\{\lambda(M)\} = \{\lambda(N)\}$, then $\mathrm{tr}(M) = \mathrm{tr}(N)$ and $\mathrm{tr}(M^{\frac{1}{2}}) = \mathrm{tr}(N^{\frac{1}{2}})$.
    \end{lemma}
    If we replace the eigendecomposition of $\Sigma_i$, for all $i \in \{0,1\} $, we obtain
    \begin{align*}
        \Sigma_1^{1/2} \theta\Sigma_0 \theta^T \Sigma_1^{1/2} &= P_1 A_1^{1/2}P_1^T \theta P_0 A_0 P_0^T \theta^T P_1 A_1^{1/2}  P_1^T\\
         &= P_1 A_1^{1/2} {\Theta} A_0 {{\Theta}^T} A_1^{1/2} P_1^T \text{ with } \Theta = P_1^T \theta P_0 \text{ and } T_\Theta = A_1^{1/2}  \Theta A_0^{1/2}\\
          &= P_1 T_\Theta T_\Theta ^T P_1^T.
    \end{align*}
    This is equivalent to $\{\lambda(\Sigma_1^{1/2} \theta\Sigma_0 \theta^T \Sigma_1^{1/2})\} = \{\lambda(T_\Theta T_\Theta ^T)\}$ and thus an application of Lemma \ref{setofeigenvalue} and the fact that the eigenvalue does not depend on the choice of basis in which the matrix is represented yield    
    $$ \mathrm{tr} \left( (\Sigma_1^{1/2} \theta\Sigma_0 \theta^T \Sigma_1^{1/2})^{1/2} \right) = \mathrm{tr} \left( (T_\Theta T_\Theta ^T)^{1/2} \right). $$
    Let $\sigma(M)$ and $\lambda(M)$ denote the singular and eigenvalue of $M$.
    The definition of $F(\cdot)$ yields
    \begin{align*}
        F(\theta )&=  \mathrm{tr}\left( \left(\Sigma_1^{1/2} \theta\Sigma_0 \theta^T \Sigma_1^{1/2}\right)^{1/2}\right)= \sum_i^d \lambda_i ( \left(\Sigma_1^{1/2} \theta\Sigma_0 \theta^T \Sigma_1^{1/2}\right)^{1/2})\\
        &= \sum_i^d \lambda_i ( \left(T_\Theta T_\Theta ^T\right)^{1/2})=  \sum_i^d  \left(\lambda_i (T_\Theta T_\Theta ^T) \right)^{1/2} = \sum_i^d  \left( \sigma_i(T_\Theta T_\Theta ^T) \right)^{1/2}.
    \end{align*}
    Since $P_0$ and $P_1$ are given explicitly, and $\Theta= P_1^T \theta P_0$,  in order to get the optimal $\theta$, one should try to obtain the optimal $\Theta$, as the multiplication map is a bijection onto the set of orthogonal matrix.
    \begin{thm}[Theorem 3.3.14, \cite{horn1994topics}] \label{orderedsingularvalue}
        Let $M$, $N \in \mathbb{R} ^{d\times d}$ and denote the ordered singular values of $M$, $N$ and $MN$  by $0\leq \sigma_1(M) \leq \cdot \cdot \cdot \leq \sigma_d(M)$,  $0\leq \sigma_1(N) \leq \cdot \cdot \cdot \leq \sigma_d(N)$ and  $0\leq \sigma_1(MN) \leq \cdot \cdot \cdot \leq \sigma_d(MN)$. We have,
$\sum_{i=1}^d \sigma_i(MN)\leq \sum_{i=1}^d \sigma_i(M)\sigma_i(N)$
    \end{thm}
    From Theorem \ref{orderedsingularvalue} and the fact that all singular values of an orthogonal matrix are equal to 1,  we obtain the following inequalities;
\begin{align*}
     F(\theta )&=  \sum_{i=1}^d \left( \sigma_i(T_\Theta T_\Theta ^T) \right)^{1/2}\leq \sum_{i=1}^d  \left( \sigma_i(T_\Theta) \sigma_i ( T_\Theta ^T) \right)^{1/2}= \sum_{i=1}^d  \sigma_i(A_1^{1/2}  \Theta A_0^{1/2})\\
     &\leq \sum_{i=1}^d  \sigma_i(A_1^{1/2}) \sigma_i ( \Theta A_0^{1/2})=\langle {a_0}^{1/2}, {a_1}^{1/2} \rangle.
\end{align*} 
For $\Theta = \mathrm{Id}$ equivalently $\theta =  P_1P_0^T $, we have $ \sup_{ \theta}\{ F(\theta)\}= \langle {a_0}^{1/2}, {a_1}^{1/2} \rangle= F( P_1P_0^T  ). $
Replacing the expression of the supremum in \eqref{modifiedgaussian}, we obtain
\begin{align*}
    \bar{d}^2(\mu_0,\mu_1)
&= \mathrm{tr}(\Sigma_0) + \mathrm{tr}(\Sigma_1 ) -2  \sup_{ \theta}\{ F(\theta)\}= \mathrm{tr}(P_0 A_0 {P_0}^T) + \mathrm{tr}(P_1 A_1 {P_1}^T ) -2  \langle {a_0}^{1/2}, {a_1}^{1/2}  \rangle\\
&=\sum_{i=1}^d a_{i,0}  + \sum_{i=1}^d a_{i,1}-2 \sum_{i=1}^d {a_{0,i}}^{1/2}{a_{1,i}}^{1/2}=\sum_{i=1}^d \left( \sqrt{a_{i,0}} -\sqrt{a_{i,1}}\right)^2\\
&= \|\sqrt{a_0} - \sqrt{a_1}\|^2.\qedhere
\end{align*}
\end{proof}

\section{A Computational Framework for the MBB Formulation}
Section~\ref{sec:abstractconv} analyzes problems of this type in an abstract framework, while Section~\ref{sec:conv} applies this framework to the discretized MBB problem and establishes a conditional convergence result. To simplify our presentation, we first show an equivalent formulation that allows us to solve the problem with $C = 0$. Let 
\begin{align*}
    X \coloneqq \lbrace (\mu,V,E) \colon \mu \in  C([0,1];\mathcal{P}_2(\mathbb{R}^d)),  V \in L^2_\mu((0,1)\times \mathbb{R}^d;\mathbb{R}^d), E \in L^2([0,1];\mathfrak{so}(d))\rbrace .
\end{align*}
\begin{rmk}
    Theorem \ref{thm:dyn-static-full} yields for any reflection \(J\in O(d)\setminus SO(d)\) the equivalent characterization    
    \[ d_{\mathrm{PW}}^2(\mu_0,\mu_1) = \min\left\{ \bar d^2(\mu_0,\mu_1), \bar d^2(\mu_0,J_\#\mu_1)
\right\}. \]
  Hence, throughout the remainder of this paper, we restrict our attention to the orientation-preserving formulation. The reflected case is obtained by applying the same analysis and numerical scheme to the reflected endpoint data.
\end{rmk}
\begin{prop}[Reduction to centered distributions and gauge $C\equiv0$]
\label{prop:center-C0}
Let $\mu_0,\mu_1\in\mathcal P_2(\mathbb{R}^d)$ and $\bar\mu_i\coloneqq \tau_{-m(\mu_i)\#}\mu_i$.
Then we have
\begin{equation}\label{eq:center-C0}
\bar d^2(\mu_0,\mu_1)=\inf_{\substack{(\mu,V,E) \in X\\
\partial_t\mu+\nabla\!\cdot(\mu(V+E x))=0,\ \mu(0)=\bar\mu_0,\ \mu(1)=\bar\mu_1}}
\int_0^1\int_{\mathbb{R}^d}|V_t|^2\,\mathrm{d}\mu_t\dt .
\end{equation}
\end{prop}
\begin{proof}
By Lemma~\ref{prop:centered-full} and Theorem~\ref{thm:dyn-static-full}, we obtain
\[
\bar d^2(\mu_0,\mu_1)=\inf_{Q\in SO(d)}W_2^2(\bar\mu_0,Q_\#\bar\mu_1).
\]
For centered endpoints, the optimal translation in the static formulation is zero.
The dynamical realization with $C\equiv0$ follows from the construction in
Theorem~\ref{thm:dyn-static-full}\, together with
Lemmas~\ref{lem:pullback-cont} and~\ref{lem:energy-inv}.
\end{proof}

\subsection{Change of Variables, Momentum Formulation, and Lagrangian}
For the computational formulation, we henceforth assume that the endpoint measures are absolutely continuous,
\[
\mu_i=\rho_i\,\mathrm dx,
\quad i\in\{0,1\},
\]
and restrict the admissible curves to measures of the form
$\mu_t=\rho_t\,\mathrm dx$.
Due to Proposition~\ref{prop:center-C0} we focus on the case $C = 0$. Thus, we redefine the set $\mathcal{A} = \mathcal{A}(\mu_0,\mu_1)$ in \eqref{eq:DefAorig} with given centered measures $\mu_0,\mu_1\in \mathcal{P}_2(\mathbb{R}^d)$ as
\begin{align*}
    \mathcal{A} \coloneqq \lbrace (\mu,V,E)\in X \colon \mu(0) = \mu_0,\ \mu(1) = \mu_1, \text{ and } (\mu,V,E,0)\text{ satisfy }\eqref{eq:continuity}\rbrace.
\end{align*}
The dynamic problem \eqref{continuousMBB} reads with momentum $m \coloneqq \rho V$ equivalently (for convenience we work with the rescaled action since the multiplication of the objective by a positive constant does not change its minimizers)
\begin{align}\label{eq:MomentumPrimal}
    \mathcal{J}(\mu,V) = \frac12 \int_0^1 \int_{\mathbb{R}^d} \frac{|m|^2}{\rho}\dx\dt.
\end{align}
The constraint in \eqref{eq:continuity} becomes  with drift term $g_E(t,x) \coloneqq -E_t x$ the identity 
\begin{align}\label{eq:VolDistNew}
   \partial_t \rho + \nabla \cdot m - \nabla \cdot \big(\rho  g_E \big)=0 \quad\text{in }\mathcal{D}'((0,1)\times \mathbb{R}^d). 
\end{align}
We define the set
\begin{align*}
    M \coloneqq \lbrace (m,\rho)& \colon m/\sqrt{\rho} \in  L^2((0,1)\times \mathbb{R}^d)\text{ and } \rho \in C([0,1];L^1(\mathbb{R}^d))\text{ is a density}, \rho|_{t=0}=\rho_0, \, \rho|_{t=1}=\rho_1 \rbrace.
\end{align*}
The following continuous Lagrangian derivation is formal and is included to motivate the discrete saddle formulation. The rigorous saddle-point and convergence analysis is carried out after finite-dimensional discretization. We include the identity \eqref{eq:VolDistNew} by using a Lagrange multiplier 
\begin{align*}
\phi \in  \Sigma = H^1((0,1)\times \mathbb{R}^d)
\end{align*}
After an integration by parts with $\phi$ under the prescribed no-flux boundary conditions at spatial infinity 
for all elements $(m,\rho) \in M$ and $E \in L^2((0,1);\mathfrak{so}(d))$, the corresponding Lagrangian reads
\begin{align}\label{eq:L-preIBP}
\begin{aligned}
&\mathcal L(m,\rho,E;\phi)
 \coloneqq \int_0^1 \int_{\mathbb{R}^d}
\Big(
\frac{|m|^2}{2\rho}
+ \phi \big(\partial_t \rho + \nabla\cdot m - \nabla\cdot(\rho g_E)\big)
\Big) \dx  \dt  \\
&\quad = \int_0^1 \int_{\mathbb{R}^d}
\Big(
\frac{|m|^2}{2\rho}
- \rho \partial_t \phi - m\cdot\nabla \phi + \rho g_E\cdot\nabla \phi
\Big)\dx\dt + \int_{\mathbb{R}^d}(\phi_1\rho_1-\phi_0\rho_0)\dx.
\end{aligned}
\end{align}
Minimizing the momentum formulation \eqref{eq:MomentumPrimal} over $\mathcal A$ formally leads to the min-max problem
\begin{align}\label{eq:Total-lagrangian}
\min_{{\substack{(m,\rho)\in M,\, E \in L^2((0,1);\mathfrak{so}(d))\\{\rho(0) = \rho_0,\ \rho(1) = \rho_1}}}}  \sup_{\phi \in \Sigma}\, \mathcal{L}(m,\rho,E;\phi).
\end{align}
Before introducing the numerical approximation scheme, we examine the convexity properties of the momentum formulation in \eqref{eq:Total-lagrangian}. 
\begin{lemma}[Convex-concave structure of the MBB Lagrangian]
\label{lem:convexity}
The Lagrangian $\mathcal{L}$ defined in \eqref{eq:L-preIBP} is convex in each primal argument. More precisely, the mapping $(m,\rho,E;\phi) \mapsto \mathcal L(m,\rho,E;\phi)$ 
\begin{enumerate}[nolistsep]
    \item is convex in $(\rho,m)$ for fixed $E\in L^2((0,1);\mathfrak{so}(d))$ and $\phi\in\Sigma$,
     \item is affine and hence concave in $\phi\in \Sigma$ for fixed $(m,\rho) \in M$ and $E \in L^2((0,1);\mathfrak{so}(d))$,
    \item is affine in $E$ for fixed functions  $(m, \rho)\in M$ and $\phi\in\Sigma$.
\end{enumerate}
\end{lemma}
\begin{proof}
The kinetic term $(\rho,m)\mapsto \int_{\mathbb{R}^d} |m|^2/(2\rho) \dx $ is jointly convex on $\{\rho>0\}$, see
\cite[Ch.~7]{villani2009optimal}. All remaining terms in $\mathcal L$ depend linearly
on $(\rho,m)$ and are affine in $E$.
Linearity in $\phi$ is immediate from the definition of $\mathcal L$ in \eqref{eq:L-preIBP}.
\end{proof}
%
\begin{prop}[First-order optimality conditions for the MBB problem]
\label{prop:mbb-optimality}
Let $(m,\rho,E)\in M\times L^2((0,1);\mathfrak{so}(d))$ be a sufficiently regular minimizer of the saddle point problem in \eqref{eq:Total-lagrangian} with $ \rho>0$.
Then there exists a potential $\phi \in \Sigma$
such that one has in the sense of distributions
\begin{align*}
\begin{aligned}
\text{(SR)}\quad 
& m = \rho\,\nabla\phi
&\quad& \text{(Optimal momentum relation)}, \\[0.5ex]
\text{(HJ)}\quad 
& \partial_t\phi+\tfrac12|\nabla\phi|^2+(E_t x)\cdot\nabla\phi = 0
&& \text{(Hamilton--Jacobi with drift)}, \\[0.5ex]
\text{(CE)}\quad 
&  \partial_t\rho+\nabla\!\cdot(\rho\nabla\phi)+\nabla\!\cdot(\rho\,E_t x)=0
&& \text{(continuity in momentum form)}. \\[0.5ex]
\intertext{Moreover, for almost every $t\in [0,1]$ one has a symmetric matrix}
\text{(SC)}\quad 
& \int_{\mathbb R^d}\rho_t(x)\,\nabla\phi_t(x)\,x^\top\dx
\ \in\ \mathfrak{so}(d)^\perp
&& \text{(Symmetry condition)}. 
\end{aligned}
\end{align*}
\end{prop}

\begin{proof}
Let $(\rho,m) \in M$ and $E \in L^2((0,1);\mathfrak{so}(d))$ be a minimizer of \eqref{continuousMBB} whose existence is guaranteed by Theorem~\ref{thm:existence-full}.
By Lemma~\ref{lem:convexity}, the momentum formulation
\eqref{eq:Total-lagrangian} defines a blockwise convex--concave saddle-point problem.
Hence, by convex--concave duality theory for every fixed $E$, there exists a potential
$\phi \in \Sigma $ such that $(\rho,m,\phi)$ is a saddle point of $\mathcal L$ see \cite[Theorem 28.1]{rockafellar1997convex}.
Stationarity with respect to $m$  and  $\rho$ yields $ m = \rho\,\nabla\phi $ and the Hamilton--Jacobi equation
\[
\partial_t\phi+\tfrac12|\nabla\phi|^2+(E_t x)\cdot\nabla\phi = 0 .
\]
Substituting $m=\rho\nabla\phi$ into the constraint recovers the continuity
equation. The variation of $\mathcal L$ with respect to $E_t\in\mathfrak{so}(d)$
gives
\[
\int_{\mathbb R^d} \rho_t(x)\,
\big(\nabla\phi_t(x)\otimes x\big) : \delta E_t \, \dx = 0
\qquad
\text{for all } \delta E_t \in \mathfrak{so}(d).
\]
This is equivalent to the orthogonality condition in (SC).
\end{proof}

\subsection{An iterative multilevel scheme for the MBB saddle-point problem}\label{sectiongproxyforMBB} 
As shown in Lemma~\ref{lem:convexity}, the variables in the  saddle-point problem \eqref{eq:Total-lagrangian} split into three blocks: the primal variables $(\rho,m)$, the dual variable $\phi$, and the skew field $E$. We exploit this structure by using an iterative multilevel optimization scheme. At each iteration $n\in \mathbb{N}$, the algorithm performs the following steps in an inner and outer loop:
\begin{enumerate}[nolistsep]
    \item Given $E^n$ compute $(\rho^{n+1},m^{n+1},\phi^{n+1})$;
    \item Given $(\rho^{n+1},m^{n+1},\phi^{n+1})$, compute $E^{n+1}$.
\end{enumerate}
For the first step, at a fixed $E^n$, the computation of $ (\rho^{n+1},m^{n+1},\phi^{n+1})$ is not obtained directly by solving \eqref{eq:Total-lagrangian}. 
 Instead the resulting subproblem is solved using the preconditioned primal-dual hybrid gradient (PDHG) method \cite{chambolle2011first,jacobs2019solving} which leads to an inner loop:
\begin{enumerate}[nolistsep,label=1\alph*.]
    \item Given $\phi^{n,j}$, compute $(\rho^{n,j+1},m^{n,j+1})$;
    \item Given \((\rho^{n,j+1},m^{n,j+1})\), compute \(\phi^{n,j+1}\).
\end{enumerate}
The remainder of this section discusses the steps in more detail. Let $Q \subset (0,1) \times \mathbb{R}^d$ bounded.
%
%
\subsubsection{Step 1a: Primal update}\label{subsubsec:Step1a}
For fixed $E^n\in L^2((0,1);\mathfrak{so}(d))$ and extrapolated dual variable $\bar{\phi}^{n,j}\coloneqq 2\phi^{n,j} -\phi^{n,j-1} \in \Sigma$, 
the $(\rho, m)$-update seeks the minimizer
\[
(\rho^{n,j+1}, m^{n,j+1}) 
= \argmin_{(\rho, m)\in M} \ \mathcal{L} (\rho, m, E^n;\bar{\phi}^{n,j})
+ \frac{1}{2\tau_\rho} \|\rho - \rho^{n,j}\|_{L^2(Q)}^2
+ \frac{1}{2\tau_m} \|m - m^{n,j}\|_{L^2(Q)}^2.
\]

The two $L^2$ proximal terms render the subproblem strictly convex on its
effective domain and therefore yield a unique minimizer. The quadratic
penalties serve as numerical stabilization terms. The first-order optimality conditions for the unknown $m=m^{n,j+1},\, \rho=\rho^{n,j+1}.$ yield
\begin{align}
\begin{aligned}
\partial_t \bar{\phi}^{n,j} 
+ \nabla \bar{\phi}^{n,j} \cdot E^n  x 
+ \frac{\| m(x) \|^2}{2 \rho(x)^2} 
- \frac{\rho(x) - \rho^{n,j}(x)}{\tau_\rho} &= 0,\\
- \nabla \bar{\phi}^{n,j}(x) 
+ \frac{m(x)}{\rho(x)} 
+ \frac{m(x) - m^{n,j}(x)}{\tau_m}&= 0.
\end{aligned}
\end{align}
The second identity provides an explicit relation between $m$ and $\rho$, which can be substituted into the first identity, leading to a scalar nonlinear equation for $\rho$. 
\subsubsection{Step 1b: Dual update}\label{subsubsec:Step1b}
Given $(\rho^{n,j+1}, m^{n,j+1})\in M$, $\phi^{n,j}\in \Sigma$, and $E^n\in L^2((0,1);\mathfrak{so}(d))$, the $\phi$-update solves
\begin{equation}
\phi^{n,j+1}  = \argmax_{\phi\in \Sigma} \, \mathcal{L}(\rho^{n,j+1}, m^{n,j+1}, E^n; \phi) - \frac{1}{2\tau_\phi} \|\phi - \phi^{n,j}\|_{H^1(Q)}^2. 
\end{equation} 
Here the $H^1$-type proximal term is a numerical choice used to precondition the dual update: it defines the metric in which the proximal step is performed at the discrete level. Such metric preconditioning is used to obtain grid-independent convergence rates in primal-dual algorithms, cf.~\cite{jacobs2019solving}. The corresponding optimality condition for the unknown $\phi$ reads
\begin{equation}
\label{eq:phi-update}
\partial_t \rho^{n,j+1} 
+ \nabla \cdot \Big( \rho^{n,j+1} E^n x + m^{n,j+1} \Big) 
+ \frac{\partial_{tt}(\phi - \phi^{n,j})}{\tau_\phi} 
+ \frac{\Delta_x (\phi -\phi^{n,j})}{\tau_\phi} = 0.
\end{equation}
We supplement this identity with the natural Neumann-in-time boundary conditions
\begin{align*}
\begin{aligned}
\rho^{n,j+1}(0, x) - \rho_0(x) 
+ 
\frac{\partial_t(\phi(0, x) -\phi^{n,j}(0, x))}{\tau_\phi} &= 0, \\
\rho^{n,j+1}(1, x) - \rho_1(x) 
+ \frac{\partial_t(\phi(1, x) - \phi^{n,j}(1, x))}{\tau_\phi} &= 0,
\end{aligned}
\end{align*}
and the zero-mean normalization
\begin{equation}
\label{eq:phi-normalization}
\int_{\mathbb{R}^d} \phi^{n,j+1}(1, x) \dx = 0.
\end{equation}
Equation \eqref{eq:phi-update} is a Poisson-type PDE for $\phi^{n+1}$ with (partial) Neumann boundary condition. The solution is known to exist and is unique up to an additive constant. The gauge condition in \eqref{eq:phi-normalization} makes the solution unique.

\subsubsection{Step 2: Optimization in the skew field} \label{sec:step3}
Given $(\rho^{n+1},m^{n+1})\in M$ and $\phi^{n+1} \in \Sigma$, define the matrix
\[ S_t^{n+1}\coloneqq \int_{\mathbb R^d}\rho_t^{n+1}(x)\,\nabla\phi_t^{n+1}(x)\,x^\top\dx
\in\mathbb R^{d\times d} \quad\text{for almost every }t\in (0,1).
\]
Since $g_E(t,x)\coloneqq -E_t x$, the dependence of the Lagrangian on $E$ is affine because the only term in $\mathcal{L}$ involving $E$ is the linear term 
\[
\int_{\mathbb R^d}\rho_t^{n+1}(x)\,g_E(t,x)\cdot\nabla\phi_t^{n+1}(x)\dx
= -\mathrm{Tr}\big(E_t(S_t^{n+1})^\top\big) = -\langle E_t,S^{n+1}_t\rangle_F.
\]
The first-order condition (Proposition \ref{prop:mbb-optimality}) associated with variations in $E$ implies that $\langle \delta E_t,S_t^\star\rangle_F = 0$ for all variations $\delta E_t \in \mathfrak{so}(d)$, which yields that $S_t^\star$ is symmetric for almost every $t\in (0,1)$. Consequently, the linear form
$E\mapsto -\mathrm{Tr}(E_t(S_t^\star)^\top)$ vanishes identically on $\mathfrak{so}(d)$ according to the identity 
\begin{align*}
    \Tr(AS) = \Tr(S^\top A^\top) = \Tr(-SA) = - \Tr(AS)\quad\text{for } A \in \mathfrak{so}(d)\text{ and }  S = S^\top.
\end{align*}
The corresponding minimization problem in $E$ is degenerate in the sense that minimizers are not unique since every skew-symmetric field yields the same value, namely zero, for the $E$-dependence of the objective.
In the numerical scheme, however, the matrix $S_t^{n+1}$ is computed from an inexact inner solve with a finite number of PDHG iterations in the inner loop, and is therefore only approximately symmetric. 
In this case, the minimization in $E$ becomes ill-posed, as the affine functional is unbounded from below as soon as $\Skew(S_t^{n+1}) \neq 0$ with
\begin{align*}
    \Skew(A) \coloneqq \frac12 (A - A^\top)\quad\text{for all }A \in \mathbb{R}^{d\times d}.
\end{align*}
To obtain a well-posed and stable update, we introduce a proximal operator in the skew field $E$, leading to
\[
E^{n+1}\in\argmin_{E\in L^2((0,1);\mathfrak{so}(d))}
\int_0^1  \Big( -\langle E_t,S_t^{n+1}\rangle_F
+\frac{1}{2\tau_E} \|E_t-E_t^n\|_F^2\Big) \dt.
\]
This problem is strictly convex on the Hilbert space $L^2((0,1);\mathfrak{so}(d))$ and admits a unique minimizer, which can be computed pointwise in time. More precisely, for almost every $t\in(0,1)$, the proximal subproblem seeks
\[
\argmin_{E_t\in\mathfrak{so}(d)}  -\langle E_t,S_t^{n+1}\rangle_F +\frac{1}{2\tau_E}\|E_t-E_t^n\|_F^2. 
\] 
The unconstrained minimizer in $\mathbb R^{d\times d}$ satisfies the first-order optimality condition \[ 0=-S_t^{n+1}+\frac{1}{\tau_E}(E_t-E_t^n), \quad\text{hence}\quad E_t=E_t^n+\tau_E S_t^{n+1}. \]  
Enforcing the constraint $E_t\in\mathfrak{so}(d)$ amounts to orthogonal projection  onto $\mathfrak{so}(d)$ with respect to the Frobenius inner product, yielding
\[
E_t^{n+1} = \Skew\big(E_t^n+\tau_E S_t^{n+1}\big) = E_t^n+\tau_E \Skew(S_t^{n+1}).
\]
In particular, if $S_t^{n+1}$ is symmetric, then $E_t^{n+1}=E_t^n$, showing that the proximal step acts as a stable selection rule in the degenerate limit while providing robustness with respect to numerical and discretization errors.

\section{A bilevel saddle framework with explicit outer descent}\label{sec:abstractconv}
The scheme studied below combines a primal--dual solution of the frozen inner
saddle problem with an explicit update of the outer variable. Primal--dual,
penalty, and regularization-based approaches have been developed for so-called bilevel
problems with nonunique, constrained, or saddle-structured lower-level
problems; see, for instance,
\cite{jiang2024primal,yao2025overcoming,solla2026optimistic,shen2026penalty}.
These works obtain stability and convergence through constrained
reformulations, penalty terms, regularized gap functions, or regularized
lower-level solution mappings.

The present analysis follows a different route. We retain the original
parameter-dependent saddle problem and work locally around a regular block
solution, without adding a regularization to the frozen inner functional. An
implicit-function argument yields a locally unique Lipschitz saddle branch,
while a computable optimality residual controls the error of the accepted
inner iterate. These ingredients allow the inner inexactness to be incorporated
into the explicit outer descent and lead to a conditional subsequential
convergence result. The framework is subsequently adapted for the fully
discrete modified Benamou--Brenier problem in Section \ref{sec:conv}.
Let \(X,Y,W\) be finite-dimensional real Hilbert spaces. We consider a
finite-dimensional saddle functional of the form
\[
    \mathcal L(x,y,w)
    =
    F(x)+\langle K(x,w),y\rangle_Y-G^\ast(y),
\]
and consider the corresponding bilevel saddle problem
\[
    \min_{w\in W}\ \min_{x\in X}\ \max_{y\in Y}
    \mathcal L(x,y,w).
\]
Here \(w\) denotes the outer variable, while \((x,y)\) are the inner
primal-dual variables. 
\begin{assumption}[Standing assumptions for the scheme]
\label{ass:abstract-standing}
The following properties hold.
\begin{enumerate}
\item  The functionals $F\colon X\to(-\infty,+\infty]$ and $G^\ast\colon Y\to(-\infty,+\infty]$ are proper, convex, and lower semicontinuous. 
\item $K:X \times W\to Y$ is $C^1$ and for every \(w\in W\), the map $K(\cdot,w):X\to Y$
is linear. Moreover, \(w\mapsto K(x,w)\) is affine for all $x\in X$.
\item For every \(w\in W\), the saddle problem \(\min_{x\in X}\max_{y\in Y}\mathcal L(x,y,w) \) admits at least one saddle point, denoted $(x^{\star}(w),y^{\star}(w)).$
\end{enumerate}
\end{assumption}

\begin{algorithm}[H]
\caption{Alternating saddle iteration with explicit outer descent}
\label{alg:alternating-saddle}
Choose $\tau_x,\tau_y,\tau_w>0$ and $(x^0,y^0,w^0)\in X\times Y\times W$.
Then we iteratively compute \((x^n,y^n,w^n)\) as follows.
\medskip

\textbf{Step 1 (Frozen inner saddle step).}
For fixed \(w^n\), approximately solve the saddle point problem $\min_{x\in X}\max_{y\in Y}
\mathcal L(x,y,w^n) $  using a primal–dual solver (For example \cite{chambolle2011first,jacobs2019solving}) and obtain an accepted pair $ (x^{n+1},y^{n+1})$.
\medskip

\textbf{Step 2 (Outer step).}
Perform a  gradient descent step
\[ w^{n+1} = w^n - \tau_w \nabla_w \mathcal L(x^{n+1},y^{n+1},w^n) .
\]

\end{algorithm}
\subsection{A local fixed-parameter saddle framework}

For each fixed \(w\in W\), we consider the inner
saddle problem
\( \min_{x\in X}\max_{y\in Y}\mathcal L(x,y,w).\)
Writing  $  z=(x,y)\in Z\coloneqq X\times Y$, we define wherever the derivatives are well defined, the fixed-\(w\) optimality operator $   \mathcal A_w\colon Z\to Z$ by
\[
    \mathcal A_w(z)
    =
    \begin{pmatrix}
        \nabla_x\mathcal L(x,y,w)\\
        -\nabla_y\mathcal L(x,y,w)
    \end{pmatrix}.
\]
We note that all gradients are understood as Hilbert gradients with respect to the inner products induced by the chosen norms on the corresponding spaces.
Thus a fixed-\(w\) saddle point satisfies \(\mathcal A_w(z)=0.\) Since \(K(\cdot,w)\) is affine by Assumption \ref{ass:abstract-standing}, we may write \(K(x,w)= K_wx + b \), with $b$ constant in $x$. Hence,
\[  \mathcal A_w(x,y) =\begin{pmatrix} \nabla F(x)+K_w^\ast y\\ \nabla G^\ast(y)-K_wx \end{pmatrix}. \]
\begin{assumption}[Regular reference inner saddle point]
\label{ass:regular-reference-saddle}
There exist $w^\star\in W$ and a saddle point $z^\star\coloneqq z^\star(w^\star)
=(x^\star(w^\star),y^\star(w^\star))\in Z$
of the frozen problem at $w^\star$ such that the mapping $(w,z)\mapsto\mathcal A_w(z)$
is continuously differentiable in a neighborhood of
$(w^\star,z^\star)$ and
\[
D_z\mathcal A_{w^\star}(z^\star)\colon Z\to Z
\]
is an isomorphism. The functionals $F$ and $G^\ast$ are finite and twice continuously
differentiable in neighborhoods of $x^\star$ and $y^\star$, respectively.
\end{assumption}
\begin{lemma}[Local branch and local reliability estimate]
\label{lem:abstract-local-branch}
Under Assumption~\ref{ass:regular-reference-saddle}, there exist open convex neighborhoods
\[
    U_{w^*}\subset W
    \quad\text{of }w^\star,
    \qquad
    U_{z^*}\subset Z
    \quad\text{of }z^\star(w^\star),
\]
a unique \(C^1\) map
\[
    U_{w^*}\ni w\mapsto z^\star(w)\in U_{z^*},
\]
and constants \(L_w,C_{\rm reg}>0\), uniform for \(w\in U_{w^*}\), such that \(\mathcal A_w(z^\star(w))=0 \text{ for all }w\in U_{w^*}\). Moreover,
\[
    \|z^\star(w_1)-z^\star(w_2)\|_Z
    \le
    L_w\|w_1-w_2\|_W
    \quad \text{ for all } w_1,w_2\in U_{w^*},
\]
and
\[
    \|z-z^\star(w)\|_Z
    \le
    C_{\rm reg}\|\mathcal A_w(z)\|_Z
\]
for all \(w\in U_{w^*}\) and all \(z\in U_{z^*}\).
\end{lemma}

\begin{proof}
The proof follows from the implicit function theorem. Indeed,
\((w,z)\mapsto \mathcal A_w(z)\) is \(C^1\) near
\((w^\star,z^\star(w^\star))\),
\[
    \mathcal A_{w^\star}(z^\star(w^\star))=0,
\]
and
\[
    D_z\mathcal A_{w^\star}(z^\star(w^\star)):Z\to Z
\]
is an isomorphism. Hence the implicit function theorem, see e.g. \cite[Section 1.1]{dontchev2009implicit}, gives
neighbourhoods \(U_{w^*}\), \(U_{z^*}\), and a unique \(C^1\)
\(w\mapsto z^\star(w)\) satisfying
\[
    \mathcal A_w(z^\star(w))=0.
\]
By possibly shrinking the neighbourhoods, we may take them to be bounded
balls whose closures remain inside the implicit-function-theorem
neighbourhood. Since \(z^\star(\cdot)\) is \(C^1\), its derivative is
bounded on \(U_{w^*}\). Therefore, for some \(L_w>0\),
\[
    \|z^\star(w_1)-z^\star(w_2)\|_Z
    \le
    L_w\|w_1-w_2\|_W.
\]
The local reliability estimate follows from a quantitative version of the local invertibility of the fixed-\(w\) optimality map. Set \[ L_\ast \coloneqq  D_z\mathcal A_{w^\ast}(z^\ast(w^\ast)). \]
By Assumption~\ref{ass:regular-reference-saddle}, \(L_\ast:Z\to Z\) is an isomorphism. Hence
\(M\coloneqq \|L_\ast^{-1}\|_{\mathcal L(Z,Z)}<\infty\). Since
\((w,z)\mapsto D_z\mathcal A_w(z)\) is continuous, the neighbourhoods \(U_{w^*}\) and \(U_{z^*}\) may be chosen sufficiently small, with \(U_{z^*}\) convex, so that \[ \|D_z\mathcal A_w(\zeta)-L_\ast\|_{\mathcal L(Z,Z)} \le \frac{1}{2M} \] for all \(w\in U_{w^*}\) and all \(\zeta\in U_{z^*}\). Let \(w\in U_{w^*}\) and \(z\in U_{z^*}\). Since \(z^\ast(w)\in U_{z^*}\) and \(U_{z^*}\) is convex, the segment \(z^\ast(w)+s(z-z^\ast(w)),\, s\in[0,1], \) is contained in \(U_{z^*}\). We set \(h\coloneqq z-z^\ast(w)\). By the fundamental theorem of calculus applied to the \(C^1\) map \(z\mapsto\mathcal A_w(z)\) along this segment, \[ \mathcal A_w(z)-\mathcal A_w(z^\ast(w))= \int_0^1 D_z\mathcal A_w(z^\ast(w)+sh)h\,ds . \] Adding and subtracting \(L_\ast h\), and using the reverse triangle inequality, we obtain \[\begin{aligned} \|\mathcal A_w(z)-\mathcal A_w(z^\ast(w))\|_Z &\ge \|L_\ast h\|_Z - \int_0^1 \|D_z\mathcal A_w(z^\ast(w)+sh)-L_\ast\|_{\mathcal L(Z,Z)} \|h\|_Z\,ds . \end{aligned} \] Moreover, since \(h=L_\ast^{-1}L_\ast h\), we have \(\|L_\ast h\|_Z \ge \frac{1}{M}\|h\|_Z . \) Therefore, \[ \|\mathcal A_w(z)-\mathcal A_w(z^\ast(w))\|_Z \ge \frac{1}{M}\|h\|_Z-\frac{1}{2M}\|h\|_Z =
\frac{1}{2M}\|h\|_Z .
\]
Consequently,
\(\|z-z^\ast(w)\|_Z \le 2M\|\mathcal A_w(z)-\mathcal A_w(z^\ast(w))\|_Z .\)
Since \(\mathcal A_w(z^\ast(w))=0\), this gives
\[\|z-z^\ast(w)\|_Z\le C_{\rm reg}\|\mathcal A_w(z)\|_Z\qquad\text{and} \qquad C_{\rm reg} = 2\|L_\ast^{-1}\|_{\mathcal L(Z,Z)}. \qedhere
\]
\end{proof}

\begin{mydef}[Fixed-point residual]
We define for all \(z_1, \, z_2 \in Z\) the fixed-point residual 
\begin{align}\label{KKTresidual}
\begin{aligned}
\mathrm{Res}(z_1, \, z_2 ) &= \|z_2-z_1\|_Z \coloneqq  \big(
\|x_2-x_1\|_{X}^2
+\|y_2-y_1\|_{Y}^2
\big)^{1/2}.
\end{aligned}
\end{align}
For each n, let \(z^{n,j}\in Z\) the sequence of inner iterates generated by the primal–dual solver ( for example \cite{chambolle2011first}), for all $j,n\in \mathbb{N}_*$, we abbreviate this fixed-point residual 
\begin{align*}
 \mathrm{Res}_{n,j} \coloneqq   \mathrm{Res}(z^{n,j-1},z^{n,j}).
\end{align*}
\end{mydef}
\begin{prop}[Convergence of the frozen inner loop and enforceable accuracy]
\label{prop:frozen-inner-loop}
Fix an outer iteration index \(n\ge 0\) and freeze the outer variable
\(w^n\in U_{w^*}\). Consider the frozen inner saddle problem \[ \min_{x\in X}\max_{y\in Y}\mathcal L(x,y,w^n). \] Let \((z^{n,j})_{j\ge0}\), with \(z^{n,j}=(x^{n,j},y^{n,j})\), be the inner primal-dual iterates generated by the Chambolle--Pock scheme. Choose the primal and dual step sizes satisfying the standard Chambolle--Pock condition \[ \tau_x\tau_y\|K(\cdot,w^n)\|_{\mathcal L(X,Y)}^2<1.
\]
Then the inner iterates converge to a saddle point $ z^\star(w^n)=(x^\star(w^n),y^\star(w^n))$ of the frozen inner problem. In the local regime considered below, this saddle point is the branch value. In particular, the corresponding fixed-point residuals satisfy
\(\operatorname{Res}_{n,j}\to 0\text{ as }j\to\infty. \) Let \[\varepsilon_n = \frac{\varepsilon_0}{(n+1)^{1+\delta}},\qquad \varepsilon_0>0,\quad \delta>0. \] Then there exists a finite iteration index \(j_n\in\mathbb N\) such that \(\operatorname{Res}_{n,j_n} + \operatorname{Res}_{n,j_n-1} \le \varepsilon_n. \) We define the accepted inner output by \[ z^{n+1}\coloneqq z^{n,j_n}. \] 
\end{prop}

\begin{proof}
For fixed \(w^n\), Assumption \ref{ass:abstract-standing} implies that the frozen inner problem is a finite-dimensional
convex--concave saddle problem with linear coupling. 
Under the uniform stability step-size condition $ \tau_x\tau_y\|K(\cdot,w^n)\|_{\mathcal L(X,Y)}^2<1$, the convergence theorem \cite[Theorem 1]{pock2011diagonal}, implies weak convergence of the inner iterates \(z^{n,j}\) to a saddle point \(z^\star(w^n)\) since we are in finite dimension, the convergence is in fact strong. Consequently, the fixed-point residuals converge to zero: $\operatorname{Res}_{n,j}\to0$.
Hence also
\[
    \operatorname{Res}_{n,j}
    +
    \operatorname{Res}_{n,j-1}
    \to0.
\]
Therefore, for the prescribed tolerance \(\varepsilon_n= \varepsilon_0 (n+1)^{-(1+\delta)},\) there exists a finite index \(j_n\) with
\[
    \operatorname{Res}_{n,j_n}
    +
    \operatorname{Res}_{n,j_n-1}
    \le
    \varepsilon_n.
\] This concludes the proof.\qedhere
\end{proof}
\begin{rmk} As a consequence, if we assume in addition that the chosen primal-dual solver satisfies the a posteriori estimate $\|\mathcal A_{w^n}(z^{n,j})\|_Z    \le C_{\rm est}   \bigl(  \operatorname{Res}_{n,j} +  \operatorname{Res}_{n,j-1} \bigr)$ with a fixed $C_{\rm est}> 0$ for the chosen residual, we obtain
\[
    \|\mathcal A_{w^n}(z^{n+1})\|_Z
    \le
    C_{\rm est}\varepsilon_n.
\]
Such an estimate will later be proved in Lemma \ref{lem:kkt-by-res}.
\end{rmk}

\subsection{Outer descent as an inexact gradient step}
\label{subsec:outer-inexact-gradient}

We now analyze the outer update in the abstract finite-dimensional setting.
For fixed \(w\), the exact inner
saddle point is denoted by
\[
    z^\star(w)=(x^\star(w),y^\star(w)).
\]
On the local neighbourhood given by the implicit-function-theorem
argument in Lemma \ref{lem:abstract-local-branch}, we define the reduced outer functional \[ \Psi(w)\coloneqq   \mathcal L(x^\star(w),y^\star(w),w). \]
\begin{prop}[Outer descent with inexact inner solves]
\label{prop:outer-inexact-descent}
Assume that the local branch constructed in Lemma~\ref{lem:abstract-local-branch} is defined on \(U_{w^*}\), and that the accepted inner outputs from Proposition \ref{prop:frozen-inner-loop} satisfy $w^n\in U_{w^*}$ and $z^{n+1}\in U_{z^*}$ for every $n\in \mathbb{N}$, the reduced functional \(\Psi\) is bounded from below on \(U_{w^*}\),
and that the outer step size satisfies
\(0<\tau_w<\frac{1}{L_\Psi}\) with $L_\Psi $ denoting the Lipschitz gradient constant of $\Psi$. Assume, in addition, that
\[
\|\mathcal A_{w^n}(z^{n+1})\|_Z
\le C_{\rm est}\varepsilon_n
\qquad\text{for all }n,
\]
where $C_{\rm est}>0$ is independent of $n$. Thus, \[ \sum_{n=0}^{\infty}\|w^{n+1}-w^n\|_W^2<\infty. \]
In particular, we have
\[
    \|w^{n+1}-w^n\|_W\to0.
\]
Moreover, $\nabla\Psi(w^n)\to0$.
Consequently, every accumulation point \(\bar w\) of the outer sequence
satisfies
\[
    \nabla\Psi(\bar w)=0.
\]
\end{prop}
\begin{proof}
Since \(z^\star(w)\) satisfies the fixed-\(w\) inner optimality system,
the derivative of the reduced functional using the chain rule gives the envelope identity  (\cite[Theorem 4.26]{bonnans2013perturbation})  
\[
    \nabla \Psi(w)
    =
    \nabla_w\mathcal L(x^\star(w),y^\star(w),w).
\]
In the present structure, \(F\) and \(G^\ast\) do not depend on \(w\), and
therefore
\[
    \nabla_w\mathcal L(x,y,w)
    =
    \nabla_w\langle K(x,w),y\rangle_Y.
\]
Since \(K\) is affine in \(w\) (see Assumption \ref{ass:abstract-standing}), this gradient depends only on the inner variables \((x,y)\). We write
\[
    G(x,y)\coloneqq \nabla_w\mathcal L(x,y,w).
\]
Thus, $   \nabla\Psi(w)=G(z^\star(w))$.
In the algorithm, the exact inner saddle point is not available. Instead,
at the \(n\)-th outer iteration, we use an accepted inner output \(z^{n+1}=(x^{n+1},y^{n+1})\) and perform the outer update
\(w^{n+1} = w^n-\tau_w G(z^{n+1}).\)
Equivalently,
\[w^{n+1} = w^n-\tau_w\bigl(\nabla\Psi(w^n)+e^n\bigr),
\]
where the inexact-gradient error is
\[e^n \coloneqq  G(z^{n+1})-G(z^\star(w^n)).
\]
We now estimate this error. Since \(G\) is induced by the bilinear
coupling term \(\langle K(x,w),y\rangle_Y\), it is Lipschitz in the inner variables on the bounded local neighbourhood \(U_{z^*}\). Hence
there exists \(C_G>0\) such that
\[
    \|G(z_1)-G(z_2)\|_W
    \le
    C_G\|z_1-z_2\|_Z
    \qquad
    \text{for all } z_1,z_2\in U_{z^*}.
\]
Using the local residual estimate for the fixed-\(w\) inner system, we
obtain
\[ \|z^{n+1}-z^\star(w^n)\|_Z \le C_{\rm reg}\|\mathcal A_{w^n}(z^{n+1})\|_Z. \]
Therefore, by choice of $z^{n+1}$ from Proposition \ref{prop:frozen-inner-loop}, if $\|\mathcal A_{w^n}(z^{n,j})\|_Z    \le C_{\rm est}   \bigl(  \operatorname{Res}_{n,j} +  \operatorname{Res}_{n,j-1} \bigr)$, the accepted inner output satisfies
\[ \|\mathcal A_{w^n}(z^{n+1})\|_Z\le C_{\rm est} \varepsilon_n,\]
and we have \(\|e^n\|_W \le C_G C_{\rm reg}C_{\rm est}\varepsilon_n.\) The same structure also gives the local Lipschitz continuity of the
reduced gradient. Indeed, for \(w_1,w_2\in U_{w^*}\),
\[
\begin{aligned}
    \|\nabla\Psi(w_1)-\nabla\Psi(w_2)\|_W
    &=
    \|G(z^\star(w_1))-G(z^\star(w_2))\|_W  \\
    &\le
    C_G\|z^\star(w_1)-z^\star(w_2)\|_Z  \\
    &\le
    C_G L_w\|w_1-w_2\|_W .
\end{aligned}
\]
Thus \(\nabla\Psi\) is Lipschitz on \(U_{w^*}\), with Lipschitz constant $L_\Psi\coloneqq C_G L_w$. We now assume that the outer step size satisfies $ 0<\tau_w<\frac{1}{L_\Psi}$. The descent lemma for $L$-smooth functions (see \cite[Section 1.2.3]{nesterov2018lectures}) yields
\[
    \Psi(w^{n+1})
    \le
    \Psi(w^n)
    +
    \langle \nabla\Psi(w^n),w^{n+1}-w^n\rangle
    +
    \frac{L_\Psi}{2}\|w^{n+1}-w^n\|_W^2 .
\]
Since $w^{n+1}-w^n = -\tau_w\bigl(\nabla\Psi(w^n)+e^n\bigr)$, we may write
\[ \nabla\Psi(w^n) = -\frac{1}{\tau_w}(w^{n+1}-w^n)-e^n.
\]
Substituting this identity into the descent estimate gives
\[
\begin{aligned}
    \Psi(w^{n+1})
    &\le
    \Psi(w^n)
    -
    \left(\frac{1}{\tau_w}-\frac{L_\Psi}{2}\right)
    \|w^{n+1}-w^n\|_W^2
    -
    \langle e^n,w^{n+1}-w^n\rangle .
\end{aligned}
\]
Using Young's inequality, we obtain constants \(c_w>0\) and \(C_w>0\),
independent of \(n\), such that
\[
    \Psi(w^{n+1})
    \le
    \Psi(w^n)
    -
    c_w\|w^{n+1}-w^n\|_W^2
    +
    C_w\|e^n\|_W^2 .
\]
Combining this with the estimate on \(e^n\), we find
\(\Psi(w^{n+1}) \le \Psi(w^n) - c_w\|w^{n+1}-w^n\|_W^2 + C\varepsilon_n^2 \)
for some constant $C>0$ independent of $n$. 
Consequently, if \(\Psi\) is bounded from below on the local region, then summing the previous inequality yields \(\sum_{n=0}^{\infty}\|w^{n+1}-w^n\|_W^2<\infty.\) In particular, \(\|w^{n+1}-w^n\|_W\to 0.\)
Moreover, since \[ w^{n+1}-w^n = -\tau_w\bigl(\nabla\Psi(w^n)+e^n\bigr),\] and \(e^n\to0\), we also obtain \(\nabla\Psi(w^n)\to 0.\) Hence every accumulation point \(\bar w\) of the outer sequence satisfies
\[\nabla\Psi(\bar w)=0.
\]
Thus, under ``the local-regime assumption" ($ w^n\in U_{w^*},\, z^{n+1}\in U_{z^*} \,\text{ for all }n.$) and square-summable tolerance choice of Proposition \ref{prop:frozen-inner-loop}, the outer iteration is an inexact gradient descent with summable gradient error \cite[Section 2, Proposition 1]{bertsekas2000gradient}. Consequently, the outer increments are square-summable  and every outer accumulation point is stationary for the reduced functional
\(\Psi\).
\end{proof}
The following theorem summarizes the resulting conditional local convergence statement.
\begin{thm}[Local subsequential convergence of the alternating scheme]
\label{thm:local-subsequential-convergence}
Let Assumptions~\ref{ass:abstract-standing} and \ref{ass:regular-reference-saddle} hold, and let
\(U_{w^*}\subset W\), \(U_{z^*}\subset Z\), \(L_w\), and \(C_{\rm reg}\) be given
by Lemma~\ref{lem:abstract-local-branch}. Suppose that the iterates
generated by Algorithm~\ref{alg:alternating-saddle} remain in the local
regularity region; that is,
\[ w^n\in U_{w^*}\qquad \text{and}\qquad z^{n+1}\in U_{z^*} \quad\text{for all }n \in \mathbb{N}.
\]
Let \((\varepsilon_n)_n\) be the tolerance sequence from Proposition~\ref{prop:frozen-inner-loop}, and let \(z^{n+1}\) be the
corresponding accepted inner output. 
Assume that the accepted inner iterates satisfy
\[
\|\mathcal A_{w^n}(z^{n+1})\|_Z
\leq
C_{\rm est}\varepsilon_n
\qquad\text{for all }n \in \mathbb{N}.
\]
Assume also that \(\Psi\) is bounded from below on \(U_{w^*}\), and that the
outer step size satisfies $0<\tau_w<\frac{1}{L_\Psi}$ where $L_\Psi$ denotes the local Lipschitz constant from Proposition \ref{prop:outer-inexact-descent}. Then the outer increments are square-summable:
\[
    \sum_{n=0}^{\infty}\|w^{n+1}-w^n\|_W^2<\infty.
\]
Furthermore, the accepted inner outputs satisfy $\nabla\Psi(w^n)\to0$ and
 \[
    \|z^{n+1}-z^\star(w^n)\|_Z
    \le C_{\rm est}
    C_{\rm reg}\varepsilon_n.
\]
Consequently, every accumulation point \(\bar w\) of the outer sequence is
stationary for the reduced functional; that is,
\[ \nabla\Psi(\bar w)=0 \text{ and }z^\star(w^{n_j})\to z^\star(\bar w),
\]
and the corresponding accepted inner outputs satisfying \(\varepsilon_{n_j}\to0\) obey
\[
    z^{n_j+1}\to z^\star(\bar w).
\]
Thus, if \(w^{n_j}\to\bar w\), then we have $z^{n_j+1}\to z^\star(\bar w)$ and $\nabla\Psi(\bar w)=0$.
Equivalently, every accumulation point \((\bar w,\bar z)\) of the local
alternating scheme satisfies $\bar z=z^\star(\bar w)$,
and $\nabla\Psi(\bar w)=0$.
\end{thm}

\begin{rmk}[On the local-regime assumption]
Such assumptions are standard in local convergence and perturbation
analysis. In particular, local solution-mapping theory typically gives
stability only in a neighbourhood of a regular reference solution; see~\cite{dontchev2009implicit}. Similarly,
local convergence results for first-order and Newton-type methods are
usually formulated under the condition that the iterates remain in a
neighbourhood where the required regularity and Lipschitz estimates hold;
see~\cite{nocedal2006numerical}. Related boundedness or
limit-point assumptions also appear in the convergence analysis of
block-coordinate and alternating minimization methods for nonconvex
problems; see, for example, \cite[Assumption 1, Corollary 1]{razaviyayn2013unified} and \cite[Section 2]{tseng2001convergence}.

\end{rmk}

\section{Application to the fixed-grid MBB problem}
\label{sec:conv}
In Section~\ref{sectiongproxyforMBB} we described, at the formal level, the alternating multilevel structure of the modified Benamou--Brenier saddle problem \eqref{eq:Total-lagrangian}. The variables naturally split into the primal density--momentum variables \((\rho,m)\), the
dual variable \(\phi\), and the skew-symmetric field \(E\). Section~\ref{sec:abstractconv} then
isolated the finite-dimensional convergence mechanism for a general bilevel saddle problem with explicit outer descent.\\
The purpose of this section is to connect these two parts. We first introduce a suitable discretization of our reformulated transport problem \eqref{eq:Total-lagrangian} and the corresponding numerical scheme. We then identify the discrete variables
with the abstract framework of Section~\ref{sec:abstractconv},
\[
        w=E,\qquad x=(\rho,m),\qquad y=\phi,
\]
and verify the assumptions required by the abstract convergence theorem for Algorithm \ref{alg:alternating-saddle}. \\
Throughout this section, all variables and operators are understood in their
fully discrete form. The convergence result is therefore a fixed-grid local conditional subsequential convergence result; no continuum limit is considered here.
\subsection{Discretization}
\label{subsec:setup}
We work on a bounded (hyper-)rectangular
computational domain $\Omega\subset\mathbb{R}^d$ that is sufficiently large so that the transported mass remains well contained in $\Omega$. This motivates no-flux boundary conditions.
Our discretization extends classical finite-volume schemes \cite[Section 2-3]{eymard2000finite}, \cite[Chapter 1-2]{leveque2002finite} to the space-time setting in the sense that we combine a symmetric finite difference in time and an interface-based discretization in space. More precisely, we use the following discretization.
We consider a uniform Cartesian discretization of $\Omega$ with mesh sizes $\Delta x_1,\dots,\Delta x_d$ and cell volume \(\Delta V\coloneqq \prod_{\ell=1}^d \Delta x_\ell\) in the sense that we decompose $\Omega$ into $N_x \approx |\Omega| / \Delta V$ (hyper-)rectangles of volume $\Delta V$. The partition containing all these cells is denoted by $\tria_x$. Furthermore, we exploit a uniform time discretization of $[0,1]$ with step size $\Delta t=1/K$ and time steps $t_k \coloneqq k\Delta t$ for all $k=0,\dots,K$ and some $K\in \mathbb{N}$. The resulting time grid reads
\[
 \tria_t =\{ [t_k,t_{k+1}]\colon k=0,\dots,K-1\}.
\]
The tensor product $\tria_{t,x} \coloneqq \tria_t \otimes \tria_x$ induces a partition of the time-space cylinder $Q \coloneqq [0,1]\times \Omega$.
Let $\mathbb{P}_0(\tria_{t,x}) \coloneqq \lbrace v \in L^2(Q)\colon v|_K \in \mathbb{R}\text{ for all }K\in \tria_{t,x}\rbrace$  denote the space of piecewise constant functions on the space-time grid.
We evaluate the piecewise constant functions with the following finite-volume operators: the discrete time derivative \(\partial_t\)  is a centered finite-difference operator defined for all $\phi \in \mathbb{P}_0(\tria_{t,x})$ with time evaluations $\phi_k \coloneqq \phi|_{[t_{k},t_{k+1}]\times \Omega} \in \mathbb{P}_0(\tria_{t,x})$ for $k=0,\dots,K-1$ and $\phi_{-1} \coloneqq \phi_1$, $\phi_{K}\coloneqq \phi_{K-2}$ by
\begin{align*}
    \partial_t \phi_k \coloneqq \frac{\phi_{k+1}- \phi_{k-1}}{2 \Delta t} \quad\text{for all }k=0,\dots,K-1. 
\end{align*}
The ghost-cell reflections $\phi_{-1}\coloneqq \phi_{1}$ and $\phi_{K}\coloneqq \phi_{K-2}$ include homogeneous time-Neumann boundary conditions corresponding to a zero temporal flux at $t=0$ and $t=1$; that is, we set the time derivative $\partial_t \phi_0 \coloneqq  \partial_t \phi_{K-1} \coloneqq 0$.
In space, the discrete gradient is defined in a face-based manner: spatial differences are evaluated across interfaces in
\[ \mathcal F_x \coloneqq \{ f \colon f \text{ is a face of $K\in \tria_{t,x}$ that is orthogonal to a spatial direction} \}. \]
We denote the space of piecewise constant functions on the interfaces $\mathcal{F}_x$ by $\mathbb P_0(\mathcal F_x)$. The discrete spatial gradient $\tilde \nabla_x\phi \in \mathbb P_0(\mathcal F_x;\mathbb R^d)$ is defined facewise as follows: For any $\phi \in \mathbb{P}_0(\tria_{t,x})$ and space neighboring cells $K_-,K_+ \in \tria$ in the sense that $K_-\cap K_+ = f \in \mathcal{F}_x$ and $\textup{mid}(K_-) + \Delta x_{\ell(f)} e_{\ell(f)} = \textup{mid}(K_+)$ with index $\ell(f) \in \lbrace 1,\dots, d\rbrace$ we define
\begin{align}\label{eq:DefDisreteGradient}
(\tilde  \nabla_x\phi)|_f \coloneqq
\frac{\phi|_{K_+}-\phi|_{K_-}}{\Delta x_{\ell(f)}}\, n_f\quad\text{with }n_f = e_{\ell(f)} \text{ denoting the normal vector of }f.
\end{align}
Accounting for Neumann boundary conditions, we set the gradient on the boundary as
\begin{align*}
    (\tilde  \nabla_x\phi)|_f \coloneqq 0 \quad\text{for all }f\in \mathcal{F}_x \cap [0,1] \times \partial \Omega.
\end{align*}
Since $\tilde{\nabla}_x $ is face-based, we also introduce a cell-centered gradient $\nabla_x \colon  \mathbb{P}_0(\tria_{t,x}) \to \mathbb P_0(\tria_{t,x};\mathbb R^d)$ obtained  for all $\phi \in \mathbb{P}_0(\tria_{t,x})$ and $K\in \tria_{t,x}$ by averaging face values; that is,
\begin{align*}
  (\nabla_x\phi)|_K  \coloneqq  \sum_{\ell=1}^d \Big(\sum_{f'\subset \partial K\cap\mathcal F_x,\ell(f')=\ell } |f'|\Big)^{-1} \sum_{f\subset \partial K\cap\mathcal F_x, \ell(f)=\ell} |f|\;   (\tilde \nabla_x\phi)_f.
\end{align*}
\begin{mydef}[Discrete divergence]\label{def:Avediv}
We define the discrete divergence $\nabla_x\cdot \colon  \mathbb P_0(\tria_{t,x};\mathbb R^d) \to  \mathbb P_0(\tria_{t,x})$
as the unique operator satisfying the discrete integration-by-parts identity
\[
\langle \nabla_x\!\cdot q,\phi\rangle_{L^2(Q)} = - \langle q,\nabla_x\phi\rangle_{L^2(Q)} \qquad \text{for all } q\in \mathbb P_0(\tria_{t,x}, \mathbb{R}^d),\phi\in\mathbb P_0(\tria_{t,x}).
\]
\end{mydef}
Since $\mathbb P_0(\tria_{t,x})$ is finite-dimensional, the duality relation above uniquely defines a cellwise constant representative of $\nabla_x\!\cdot q$. In the sequel, $ (\nabla_x\!\cdot q)_K$ denotes its value on the cell $ K$.\\
We use the discrete operators above to define a discrete version of the \(H^1(Q)\) semi-inner product and norm. Moreover, we introduce a version scaled by the step size  $\tau_\phi > 0$ used in the generalized dual update \eqref{eq:phi-update}. In particular, we set for all $\phi,\psi \in \mathbb{P}_0(\tria_{t,x})$ 
\begin{align}\label{eq:normofzh}   
\begin{aligned}
    \langle \phi,\psi\rangle_{H^1(Q)}
&\coloneqq \langle \partial_t\phi,\partial_t\psi\rangle_{L^2(Q)}
+\langle \nabla_x\phi,\nabla_x\psi\rangle_{L^2(Q)}\quad&\text{and}&\quad \lVert \phi \rVert_{H^1(Q)} = \langle \phi,\phi\rangle_{H^1(Q)}^{1/2},\\
\langle \phi,\psi\rangle_{\Sigma_h}
&\coloneqq  \tau_\phi^{-1}  \langle \phi,\psi\rangle_{H^1(Q)}\quad&\text{and}&\quad \lVert \phi \rVert_{\Sigma_h} = \langle \phi,\phi\rangle_{\Sigma_h}^{1/2}.
\end{aligned}
\end{align}
We set the discrete dual space in \eqref{eq:L-preIBP} by removing the kernel of the $H^1$-seminorm, that is
\begin{align*}
  \Sigma_h \coloneqq \Big\lbrace \phi \in \mathbb{P}_0(\tria_{t,x})&\colon \phi \perp \ker(\partial_t, \nabla_x)\Big\rbrace,
\end{align*}
where $\ker(\partial_t, \nabla_x)=\Big\lbrace \phi \in \mathbb{P}_0(\tria_{t,x})\colon \partial_t\phi=0 \text{ and } \nabla_x \phi =0\Big\rbrace.$
Equivalently, one may impose gauge conditions fixing the two parity-constant modes induced by the centered time derivative, for instance $ \int_\Omega \phi_{K-1} \dx = 0$ and $\int_\Omega \phi_{K-2} \dx = 0$ consistent with the fixed gauge at final time in~\eqref{eq:phi-normalization}.
This is a Hilbert space equipped with norm $\lVert \cdot \rVert_{\Sigma_h}$ and Riesz mapping 
\begin{align}\label{eq:ScalledRiesz}
\mathbf S\coloneqq \tau_\phi^{-1}(\partial_t^\top \partial_t-\Delta_x) \colon \Sigma_h \to \Sigma_h^*.
\end{align}
We define the discrete primal space with the $L^2$ orthogonal projection $\Pi_0 \colon L^2(Q;\mathbb{R}^d) \to \mathbb{P}_0(\tria_{t,x};\mathbb{R}^d)$ onto piecewise constants as
\begin{align*}
    M_h \coloneqq \big\lbrace (\rho,m) \in  \mathbb{P}_0(\tria_{t,x}) \times \mathbb{P}_0(\tria_{t,x};\mathbb{R}^d) \big\rbrace.
 \end{align*}
 For the prescribed discrete endpoint densities $ \rho_0,\, \rho_{K-1}\in \mathbb{P}_0(\tria_{x})$, we restrict the density variable to the affine set 
\begin{align*}
    M_h(\rho_0, \rho_{K-1}) \coloneqq \big\lbrace (\rho,m) \in M_h: \, \rho_k|_{k=0}= \rho_0 \text{ and } \rho_k|_{k=K-1} =\rho_{K-1} \big\rbrace.
 \end{align*}
 We continue to denote this set by \(M_h\) for simplicity.
 \subsection{Discrete saddle formulation and identification with the abstract framework}
We now formulate the discrete problem in the saddle-point form introduced in Section~\ref{sec:abstractconv}. The
abstract variables are identified as
\[w=E,\, x=(\rho,m),\, y=\phi,
\text{ with } E\in W_h\coloneqq P_0(\mathcal T_t;\mathfrak{so}(d)),    \quad (\rho,m)\in M_h, \quad \phi\in\Sigma_h.
\]
For \(E\in W_h\), we set \( g_E(t,x)\coloneqq -E_t x .\) The discrete MBB Lagrangian \eqref{eq:L-preIBP} is given by
\[ \mathcal L_h(\rho,m,\phi,E) \coloneqq  F_h(\rho,m) +\big\langle K_h((\rho,m),E),\phi\big\rangle ,
\]
where \(F_h(\rho,m)\coloneqq \frac{1}{2}\| {m}/{\sqrt{\rho}}\|_{L^2(Q)}^2 \)
The integrand in $m/\sqrt{\rho}$ is understood in the convention of Benamou--Brenier \cite{benamou2000computational}; that is, we read 
\begin{align}\label{eq:BBconvention}
 \frac{|m|^2}{\rho}&= \begin{cases}
        \frac{|m|^2}{\rho}&\text{if }  \rho>0,\\
        0&\text{if }   (\rho,m)=(0,0),\\
        +\infty&\text{otherwise}.
    \end{cases}
\end{align}
For \((\rho,m)\in M_h\) and \(E\in W_h\), we define the discrete
total flux by \(J_k \coloneqq  m_k-\rho_k\Pi_0 g_{E_k}, \quad g_{E_k}(x)=-E_kx .\)
The cellwise representative of the discrete continuity residual is
\begin{equation}\label{eq:DefResidualKh}
R_{h,E}(\rho,m)_k\coloneqq  \partial_t\rho_k+\nabla_x\cdot J_k,\quad k=0,\ldots,K-1,
\end{equation}
that is,
\(R_{h,E}(\rho,m)_k=\frac{\rho_{k+1}-\rho_{k-1}}{2\Delta t} +\nabla_x\cdot \left(m_k-\rho_k\Pi_0 g_{E_k} \right). \)
The discrete continuity equation is \(R_{h,E}(\rho,m)_k=0, \quad k=0,\ldots,K-1.\)
The discretized version $K_h \coloneqq K_h(E)\colon M_h \to \Sigma_h^*$ of the operator $K(\cdot;E) \colon M\to \Sigma^*$ in \eqref{eq:VolDistNew} can be obtained by testing this cellwise residual against the gauge-fixed dual space: \[\langle K_h(\rho,m;E),\phi\rangle\coloneqq  \sum_{k=0}^{K-1} \int_\Omega R_{h,E}(\rho,m)_k\,\phi_k\,dx,\quad \phi\in\Sigma_h.
\]
Equivalently, \(\langle K_h(\rho,m;E),\phi\rangle =\int_Q \left( \partial_t\rho+\nabla_x\cdot m -\nabla_x\cdot(\rho\Pi_0 g_E) \right)\phi\,dz .\)
And the discretized version $K_h\colon M_h \to \Sigma_h^*$ of the operator $K \colon M\to \Sigma^*$ in \eqref{eq:VolDistNew} reads
\begin{align}\label{eq:DefKh} \langle K_h(\rho,m;E), \phi\rangle \coloneqq  \int_Q (\partial_t \rho + \nabla \cdot m - \nabla \cdot \big(\rho \Pi_0  g_E \big))\phi \dz \quad\text{for all }(\rho,m) \in M_h\text{ and }\phi \in \Sigma_h. \end{align}
Similarly, for fixed \(E\), we write $K_h(\rho,m)\coloneqq  K_h(\rho,m;E)$. Thus the discrete saddle problem has the abstract form
\begin{equation}\label{eq:DiscreteLgrangian}
    \mathcal L_h(x,y,w) = F_h(x)+\langle K_h(x,w),y\rangle_{\Sigma_h^\ast,\Sigma_h} -G_h^\ast(y), \qquad G_h^\ast\equiv 0.
\end{equation}
The following lemmas collect the mass conservative structure and the basic linearity and
boundedness properties of the discrete coupling operator \(K_h\).
%
As discussed above, we express piecewise constant functions in terms of vectors. Similarly, we order our cells $K_x \in \tria_{x}$ and associate them with an index $\alpha = 1,\dots, N_x$. Let $x_\alpha$ denote the center of gravity of the $\alpha$-th element $K_\alpha \in \tria_x$. Then the $L^2$ orthogonal projection satisfies $(\Pi_0 g_E)|_{[t_k,t_{k+1}] \times K_\alpha} = - E_k x_\alpha$ for all $\alpha = 1,\dots,N_x$.

\begin{lemma}[Compatibility of the reduced and cellwise constraints]
\label{lem:constraint-compatibility}
Let $(\rho,m)\in M_h$ and $E\in W_h$, and let
$R_{h,E}(\rho,m)$ denote the discrete continuity residual defined in \eqref{eq:DefResidualKh}. Assume that the number $K$ of time slabs is even. Then $K_h(\rho,m;E)=0 \quad\text{in }\Sigma_h^\ast$
implies $R_{h,E}(\rho,m)_k=0,
\quad k=0,\ldots,K-1.$
\end{lemma}
\begin{proof}
By the definition of $K_h$ \eqref{eq:DefKh}, the condition
$K_h(\rho,m;E)=0$ means that
\[\bigl\langle R_{h,E}(\rho,m),\phi\bigr\rangle_{L^2(Q)}=0
\quad\text{for every }\phi\in\Sigma_h.
\]
Since $\Sigma_h= \ker(\partial_t,\nabla_x)^\perp$,
it follows that
$R_{h,E}(\rho,m)\in\ker(\partial_t,\nabla_x)$. Hence the residual
is spatially constant and constant separately on the even and odd time indices. Consequently, there exist constants
$c_{\mathrm{ev}},c_{\mathrm{odd}}\in\mathbb R$ such that
\[R_{h,E}(\rho,m)_k =
\begin{cases} c_{\mathrm{ev}}, & k \text{ even},\\ c_{\mathrm{odd}}, & k \text{ odd}. \end{cases}
\]
At the initial time index, the ghost-cell convention gives $\partial_t\rho_0=0$. Therefore \(R_{h,E}(\rho,m)_0=\nabla_x\cdot J_0. \) Using the adjoint definition of the discrete divergence Definition \ref{def:Avediv} with the test function $\mathbf{1} \in \mathbb P_0(\tria_{t,x},\mathbb{R})$ such that it is $1$ in space on the k-th time slab, we obtain
\[\int_\Omega R_{h,E}(\rho,m)_0dx = \int_\Omega \nabla_x\cdot J_0,dx = -\langle J_0, \nabla_x \mathbf{1} \rangle_{L^2(\Omega)}=0. \]
Since $R_{h,E}(\rho,m)_0=c_{\mathrm{ev}}$ is spatially constant, this yields $c_{\mathrm{ev}}=0$. Similarly, the terminal ghost-cell convention gives $\partial_t\rho_{K-1}=0$, and therefore
\(\int_\Omega R_{h,E}(\rho,m)_{K-1},dx=0.\) Because $K$ is even, the index $K-1$ is odd. Hence $R_{h,E}(\rho,m)_{K-1}=c_{\mathrm{odd}}$, and consequently $c_{\mathrm{odd}}=0$. Thus both parity components vanish, and therefore $R_{h,E}(\rho,m)_k=0$ \text{for every }$k=0,\ldots,K-1$. \qedhere
\end{proof}

\begin{lemma}[Discrete mass conservation]
\label{lem:discrete-mass-conservation}
Let $(\rho,m)\in M_h$ and $E\in W_h$ satisfy $K_h(\rho,m;E)=0 \quad\text{in }\Sigma_h^\ast$. In particular, if  $\mu_0, \mu_{K-1} \in \mathcal{P}_2(\Omega) $ with $\int_\Omega \rho_0\,\mathrm dx = \int_\Omega \rho_{K-1}\,\mathrm dx = 1$ and $K$ is even, we have 
\[ \int_\Omega \rho_k\dx = 1 \qquad \text{for all }k=0,\dots,K-1 \text{ with } \rho_k \coloneqq \rho|_{[t_k,t_{k+1}]} \in \mathbb{P}_0(\tria_{x}) . \]
\end{lemma}

\begin{proof}
By Lemma~\ref{lem:constraint-compatibility}, the reduced constraint implies the cellwise discrete continuity equation \[ \partial_t\rho_k+\nabla_x\cdot J_k=0, \quad k=0,\ldots,K-1. \] 
Testing with an arbitrary test function $\phi_k \in  \mathbb P_0(\tria_{x})$ and integrating over $\Omega$ yields \[\int_\Omega \partial_t\rho_k \phi_k\dx + \int_\Omega (\nabla_x\cdot J_k )\phi_k \dx=0\qquad\text{for all }k=0,\dots,K-1. \]
By the definition of the discrete divergence see Definition \ref{def:Avediv}, testing with the constant function $\phi_k=\mathbf{1} \in \mathbb P_0(\tria_{t,x})$ such that it is $1$ in space on the k-th time slab yields
\begin{align*}
\int_\Omega \nabla_x\cdot J_k\dx&= \langle \nabla_x\cdot J_k ,\mathbf{1}\rangle_{L^2(\Omega)}  = \langle \nabla_x\!\cdot J_k,\mathbf{1} \rangle_{L^2(\Omega)} 
    = - \langle J_k,\nabla_x\mathbf{1} \rangle_{L^2(\Omega)}
    =0.
\end{align*}
This yields for the centered difference
\( \partial_t\rho_k= \frac{1}{2 \Delta t}(\rho_{k+1}- \rho_{k-1}) \) the identity $\int_\Omega \partial_t\rho_{k} \dx =0$. Thus
\[\int_\Omega \rho_{k+1}\dx = \int_\Omega \rho_{k-1} \dx.
\]
 Furthermore, this relation implies that the masses are constant separately on the even and odd time indices. Consequently, since $\mu_0, \mu_{K-1} \in \mathcal{P}_2(\Omega) $ with $\dmu_i=\rho_i\dx$ for $i \in \{0,K-1\}$, we have $\int_\Omega \rho_{K-1}\,\mathrm dx = \int_\Omega \rho_0\,\mathrm dx = 1$. Hence both parity classes carry unit mass and for every $k=0,\dots,K-1$,
\begin{align*}
&\int_\Omega \rho_k\,\mathrm d  x = 1. \qedhere
\end{align*}
\end{proof}

\begin{lemma}[Linearity and well-definedness of $K_h$]\label{BoundenessofKh}
The operator $K_h \colon M_h \to \Sigma_h^*$ defined in \eqref{eq:DefKh} is linear and well defined.
\end{lemma}
\begin{proof}
Since linearity follows by definition, it remains to verify the well-definedness. On our bounded domain \(\Omega\), one has the estimate
\( \|  g_E\|_{L^2(Q)}\le C_\Omega\,\|E\|_{L^2(0,1)} \) with the constant $C_\Omega =|\Omega|^{1/2} \max\lbrace |x| \colon x\in \Omega\rbrace $. 
Therefore, for every fixed \(E\) the map \(\rho\mapsto \nabla_x\!\cdot(\rho \Pi_0 g_E) \in \Sigma_h^*\) for all $\rho \in \mathbb{P}_0(\tria_{t,x})$ is a well-defined linear operator. Consequently, the operator $K_h$ defined in \eqref{eq:DefKh} is linear and well-defined for every fixed $E$.    
\end{proof}
Set the weight \(\mathbf T^{-1}\coloneqq \mathrm{diag}(\tau_\rho^{-1}I,\tau_m^{-1}I) \colon  \mathbb{P}_0(\tria_{t,x}) \times \mathbb{P}_0(\tria_{t,x};\mathbb{R}^d) \to  \mathbb{P}_0(\tria_{t,x}) \times \mathbb{P}_0(\tria_{t,x};\mathbb{R}^d)\); that is, $\mathbf T^{-1}(\rho,m) = (\tau_\rho^{-1} \rho, \tau_m^{-1} m)$.
Recall the discrete scaled $H^1(Q)$ inner product $\langle \cdot ,  \cdot \rangle_{\Sigma_h }= \tau^{-1}_\phi \langle \cdot,\cdot\rangle_{H^1(Q)}$ introduced in \eqref{eq:normofzh}, whose Riesz mapping is given by the operator $\mathbf S$ in \eqref{eq:ScalledRiesz} i.e. $\langle \cdot,\cdot \rangle_{\Sigma_h}\coloneqq \langle \cdot , \mathbf{S} \cdot \rangle_{L^2(Q)}$. The associated norm reads $\lVert \phi \rVert_{\Sigma_h}^2 =\tau^{-1}_\phi ( \|\partial_t \phi\|_{L^2(Q)}^2+ \|\nabla_x \phi \|_{L^2(Q)}^2)$.
Similarly, we set the weighted inner product $\langle \cdot ,  \cdot \rangle_{T^{-1}} = \langle \cdot ,\mathbf T^{-1} \cdot \rangle_{L^2(Q)}$ and its induced norm on $M_h $ by $\lVert (\rho,m) \rVert_{M_h}^2 = \frac1{\tau_\rho}\|\rho\|_{L^2(Q)}^2+\frac1{\tau_m}\|m\|_{L^2(Q)}^2$.
Let $\|\cdot\|_{\mathcal L(M_h,\Sigma_h)}$ denote the induced operator norm with weighted Hilbert spaces $(M_h,\langle\cdot,\cdot\rangle_{\mathbf T^{-1}})$ and
$(\Sigma_h,\langle\cdot,\cdot\rangle_{\mathbf S})$, i.e.,
\begin{equation}\label{normfromTandS}
    \|A\|_{\mathcal L(M_h,\Sigma_h)} \coloneqq \sup_{z \in M_h \setminus \lbrace  0 \rbrace}\frac{\|Az\|_{\Sigma_h}}{\|z\|_{M_h}}.
    \end{equation}
Throughout the sequel, we identify \(\Sigma_h^*\) with \(\Sigma_h\)
via the Riesz isomorphism induced by
\(\langle\cdot,\cdot\rangle_{\Sigma_h}\),
so that \(K_h\colon M_h\to \Sigma_h^*\) is identified with an operator \(K_h\colon M_h\to \Sigma_h\).
Let $\partial_{t,h}^{*}$ denote the $L^2(Q)$-adjoint of the centered temporal-difference operator. Since $\|\cdot\|_{H^1(Q)}$ is a norm on $\Sigma_h$, we set \[ C_{t,h} \coloneqq \sup_{\phi\in\Sigma_h\setminus\{0\}} \frac{\|\partial_{t,h}^{*}\phi\|_{L^2(Q)}} {\|\phi\|_{H^1(Q)}}<\infty.\]
The constant $C_{t,h}$ may depend on the fixed discretization.
\begin{lemma}[Quantitative operator norm estimate]
\label{lem:metric-operator-norm}
With the above conventions, the identification of $\Sigma_h^*$ with $\Sigma_h$ and the induced operator norm of
\(K_h \colon  M_h \to \Sigma_h^*\),
there exists a constant $C_E= \|E\|_{L^\infty(0,1)}\max_{x\in \Omega}|x|>0$, depending only on $\Omega$ and $E$, but independent of the discretization, such that 
\begin{equation}
\label{eq:uniform-precond-bound}
\|K_h\|_{\mathcal L(M_h,\Sigma_h^*)}
\leq
\sqrt{\tau_\phi}\,\sqrt{(C_{t,h} +C_E)^2\,\tau_\rho+\tau_m}.
\end{equation}
Hence, the operator norm is bounded for every fixed discretization and uniformly for $E$ in bounded subsets of $W_h$.
\end{lemma}
\begin{proof}
Using the definition \eqref{eq:DefKh}, the $L^2(Q)$-adjoint relation in time and the discrete integration-by-parts identity in space, we obtain \[ \langle K_h(\rho,m),\phi\rangle = \langle \rho,\partial_{t,h} ^{*}\phi\rangle_{L^2(Q)} - \langle m-\rho\Pi_0g_E,\nabla_x\phi\rangle_{L^2(Q)}. \] Consequently, by the definition of $C_{t,h}$ and the Cauchy--Schwarz inequality, \[ \big|\langle K_h(\rho,m),\phi\rangle\big| \le \left( C_{t,h}\|\rho\|_{L^2(Q)} + \|m-\rho\Pi_0g_E\|_{L^2(Q)} \right) \|\phi\|_{H^1(Q)}. \]
Dividing by $\|\phi\|_{H^1(Q)}$ and taking the supremum gives
\[
\|K_h(\rho,m)\|_{\Sigma_h^*}
\le
\sqrt{\tau_\phi}\,
\big( C_{t,h}\|\rho\|_{L^2(Q)}+\|m \|_{L^2(Q)}  + \|\rho \Pi_0 g_E\|_{L^2(Q)}\big).
\]
Since $g_E(x)=-E_k x$ on each space-time cell and $E$ is fixed here and piecewise constant in time with $\Omega$ is bounded, i.e. $g_E \in {L^\infty(Q)}$; hence there exists $C_E>0$, depending only on $\Omega$ and $\|E\|_{L^\infty(0,1)}$, such that
\[
\|\rho g_E\|_{L^2(Q)}\le C_E\,\|\rho\|_{L^2(Q)}.
\]
Hence, we obtain the bound
\[
\|K_h(\rho,m)\|_{\Sigma_h^*}
\le \sqrt{\tau_\phi} \, \begin{pmatrix}
(C_{t,h}+C_E) \tau_\rho^{1/2} \\
\tau_m^{1/2}
\end{pmatrix} \cdot 
\begin{pmatrix}
\tau_\rho^{-1/2} \lVert \rho \rVert_{L^2(Q)}\\
\tau_m^{-1/2} \lVert m \rVert_{L^2(Q)}
\end{pmatrix}.
\]
The Cauchy--Schwarz inequality thus yields
\begin{align*}
\|K_h(\rho,m)\|_{\Sigma_h^*}
& \le
\sqrt{\tau_\phi}\,
\sqrt{(C_{t,h}+C_E)^2\,\tau_\rho+\tau_m}\,
\|(\rho,m)\|_{M_h}. \qedhere
\end{align*}
\end{proof}

\begin{corollary}[Uniform stability]
\label{cor:uniform-stability}
If one chooses step sizes such that for some fixed $\varepsilon \in (0,1)$
\[
\tau_\phi\big((C_{t,h}+C_E)^2\,\tau_\rho+\tau_m\big) \leq 1 - \varepsilon  < 1,
\]
the operator norm of $K_h$ is bounded away from one in the sense that we have the bound
\[ \|K_h\|_{\mathcal L(M_h,\Sigma_h^*)} \leq 1 - \delta < 1\qquad\text{with}\qquad \delta \coloneqq 1 - \sqrt{1-\varepsilon}. \]
\end{corollary}
\begin{lemma}[Existence of a saddle point for the frozen-\(E\) problem]
\label{lem:frozen-E-existence}
For every fixed \(E\in W_h\), the discrete frozen-\(E\) problem admits at
least one saddle point
\[
    z^\star(E)
    =
    \bigl(\rho^\star(E),m^\star(E),\phi^\star(E)\bigr)
    \in M_h\times\Sigma_h .
\]
\end{lemma}

\begin{proof}
We first construct an admissible pair with finite energy. Choose a discrete
curve \(\bar\rho\) satisfying the prescribed endpoint conditions $\mu_0, \mu_{K-1} \in \mathcal{P}_2(\Omega) $, and whose free density components are strictly positive, such a curve can be obtained for example by interpolating the endpoint densities. Since $ \int_\Omega \bar\rho_k\,dx=1 \text{ for every time index }k$, the discrete time derivative preserves the spatial mean and we have
\[
    \int_\Omega \partial_t\bar\rho_k\,dx=0
    \qquad\text{for every time index }k.
\]
Thus \(-\partial_t\bar\rho_k\) is orthogonal to the spatially constant discrete functions. By the definition of \(\nabla_x\!\cdot\) as the negative adjoint of \(\nabla_x\) see Definition \ref{def:Avediv}, the finite-dimensional identity
\(\operatorname{Ran}(\nabla_x\!\cdot) = \ker(\nabla_x)^\perp \) holds. Since \(\ker(\nabla_x)\) consists of the spatially constant
functions, there exists, for every \(k\), a discrete vector field \(J_k\) such that \(\nabla_x\!\cdot J_k=-\partial_t\bar\rho_k . \)
At the initial and final time indices, the ghost-point convention gives
\(\partial_t\bar\rho_k=0\), and we take \(J_k=0\). We now set
\(\bar m_k\coloneqq J_k+\bar\rho_k\Pi_0g_{E,k}. \)
Then
\[\partial_t\bar\rho_k + \nabla_x\!\cdot \bigl(\bar m_k-\bar\rho_k\Pi_0g_{E,k}\bigr) =0,
\]
and therefore \((\bar\rho,\bar m)\) is admissible by the compatibility result established above see Lemma \ref{lem:constraint-compatibility}. The strict positivity of the free components of
\(\bar\rho\), together with \(J_k=0\) at the endpoint time levels, implies
that
\[ F_h(\bar\rho,\bar m)<\infty .
\]
Let \((\rho^\ell,m^\ell)_{\ell\in\mathbb N}\) be a minimizing sequence for the frozen-\(E\) primal problem. The admissible competitor shows that the sequence may be chosen so that
\[
    F_h(\rho^\ell,m^\ell)\leq C
\]
for some \(C>0\) independent of \(\ell\). By the discrete mass-conservation (Lemma \ref{lem:discrete-mass-conservation}) and the nonnegativity of \(\rho^\ell\) assumed, the sequence
\((\rho^\ell)\) is bounded on the fixed space--time grid, and hence in \(L^\infty(Q)\cap L^2(Q)\). Moreover, using the Benamou-Brenier convention on the Kinetic energy \eqref{eq:BBconvention},
\[\|m^\ell\|_{L^2(Q)}^2= \int_{\{\rho^\ell>0\}} \rho^\ell\frac{|m^\ell|^2}{\rho^\ell} \leq 2\|\rho^\ell\|_{L^\infty(Q)} F_h(\rho^\ell,m^\ell). \]
Thus \((\rho^\ell,m^\ell)\) is bounded in the finite-dimensional space \(M_h\). After extraction of a subsequence, it converges to some \((\rho^\star,m^\star)\in M_h\). The admissible set is closed, 
hence \((\rho^\star,m^\star)\) is admissible. The convex kinetic functional \(F_h\) is lower semicontinuous, and therefore
\[ F_h(\rho^\star,m^\star) \leq \liminf_{\ell\to\infty}  F_h(\rho^\ell,m^\ell). \]
Consequently, \((\rho^\star,m^\star)\) is a minimizer of the frozen-\(E\) primal problem. Finally, the admissible pair \((\bar\rho,\bar m)\) constructed above belongs to the relative interior of the effective domain with respect to the free
density variables. The standard convex duality theorem for a proper convex functional under affine equality constraints therefore yields a multiplier \(\phi^\star\in\Sigma_h\); see \cite[Chapter~III, Section~2]{ekeland1999convex}. Hence \[ \bigl(\rho^\star,m^\star,\phi^\star\bigr) \] is a saddle point of the discrete frozen-\(E\) Lagrangian in \eqref{eq:DiscreteLgrangian}.
\end{proof}
The preceding results verify the standing structural assumptions of the abstract
saddle framework for the discrete MBB problem. The only local restriction
concerns the differentiability of the kinetic perspective, which is smooth on
the set of strictly positive discrete densities.

\begin{prop}[Verification of Assumption~\ref{ass:abstract-standing}]
\label{prop:verification-assumption-one}
Let the discrete spaces \(M_h\), \(\Sigma_h\), and \(W_h\), the functional
\(F_h\), and the coupling operator \(K_h\) be defined as above. Then the
discrete MBB saddle problem satisfies the structural requirements of
Assumption~\ref{ass:abstract-standing}. More precisely:
\begin{enumerate}
    \item \(M_h\), \(\Sigma_h\), and \(W_h\) are finite-dimensional Hilbert
    spaces. The functional \(F_h\) is proper, convex, and lower
    semicontinuous on \(M_h\), and it is continuously differentiable on the
    subset of \(M_h\) on which all free density components are strictly
    positive. Moreover, \(G_h^\ast\equiv 0\) is convex and continuously
    differentiable.

    \item The map
    \[
        K_h:M_h\times W_h\longrightarrow \Sigma_h^\ast
    \]
    is continuously differentiable. For every fixed \(E\in W_h\), the map
    \[
        K_h(\,\cdot\,;E):M_h\longrightarrow \Sigma_h^\ast
    \]
    is linear and bounded, and for every fixed \((\rho,m)\in M_h\), the map
    \[
        E\longmapsto K_h(\rho,m;E)
    \]
    is affine.

    \item For every \(E\in W_h\), the frozen-\(E\) saddle problem admits at
    least one saddle point.
\end{enumerate}
\end{prop}

\begin{proof}
The finite-dimensional Hilbert structure follows from the definitions of
\(M_h\), \(\Sigma_h\), and \(W_h\). The claimed convexity and lower
semicontinuity properties of \(F_h\) follow from the Benamou--Brenier
convention in~\eqref{eq:BBconvention}, while
continuous differentiability holds on the positive-density region. Since
\(G_h^\ast\equiv 0\), the corresponding properties are immediate. For fixed \(E\), the linearity and well-definedness of
\(K_h(\,\cdot\,;E)\) are given by Lemma~\ref{BoundenessofKh}, and its boundedness follows by
Lemma~\ref{lem:metric-operator-norm}. The dependence on \(E\) is linear by the definition of \(g_E\) and
the representation~\eqref{eq:DefKh}; hence \(K_h\) is
continuously differentiable on \(M_h\times W_h\). Finally, the existence of a saddle point for every frozen \(E\in W_h\) is established in Lemma~\ref{lem:frozen-E-existence}.
\end{proof}
We next introduce the frozen $E$-Optimality Operator and present some structural properties needed to verify the second main assumption of Section \ref{sec:abstractconv}, namely Assumption \ref{ass:regular-reference-saddle}.
\begin{mydef}[Frozen-\(E\) optimality operator]
\label{def:frozen-optimality-operator}
Let $Z_h\coloneqq M_h\times\Sigma_h$ with $z=(\rho,m,\phi)\in Z_h$,
and fix \(E\in W_h\). On the set where the free density components of
\(\rho\) are strictly positive, we define the frozen-\(E\) optimality
operator
\begin{equation}
\label{eq:frozen-optimality-operator-dual}
    A_E:Z_h\longrightarrow Z_h^\ast
\end{equation}
by
\[
    A_E(\rho,m,\phi) \coloneqq 
    \begin{pmatrix}
        \nabla_{(\rho,m)}
        \mathcal L_h(\rho,m,\phi,E)\\[1mm]
        -\nabla_\phi
        \mathcal L_h(\rho,m,\phi,E)
    \end{pmatrix}.
\]
Since \(Z_h\) is a finite-dimensional Hilbert space, the Riesz isomorphism
identifies \(Z_h^\ast\) with \(Z_h\). Under this identification, we regard
\(A_E\) as an operator from \(Z_h\) into \(Z_h\), namely
\begin{equation}
\label{eq:frozen-optimality-operator}
    A_E(\rho,m,\phi)
    =
    \begin{pmatrix}
        \nabla F_h(\rho,m)+K_E^\ast\phi\\[1mm]
        -K_E(\rho,m)
    \end{pmatrix},
\end{equation}
where $K_E\coloneqq K_h(\,\cdot\,;E):M_h\longrightarrow\Sigma_h^\ast$
and \(K_E^\ast:\Sigma_h\to M_h^\ast\) denotes its adjoint. 
\end{mydef}
\begin{lemma}[Characterization of frozen-\(E\) saddle points]
\label{lem:frozen-saddle-characterization}
Fix \(E\in W_h\), and let
\[
    z^\star(E)
    =
    \bigl(\rho^\star(E),m^\star(E),\phi^\star(E)\bigr)
    \in Z_h
\]
be such that the free components of \(\rho^\star(E)\) are strictly positive.
Then \(z^\star(E)\) is a saddle point of the frozen-\(E\) discrete
Lagrangian if and only if
\begin{equation}
\label{eq:frozen-optimality-equation}
    A_E\bigl(z^\star(E)\bigr)=0 .
\end{equation}
\end{lemma}

\begin{proof}
Suppose first that \(z^\star(E)\) is a saddle point. Then we have
\[
    (\rho^\star(E),m^\star(E))
    \in
    \operatorname*{argmin}_{(\rho,m)\in M_h}
    \mathcal L_h
    \bigl(\rho,m,\phi^\star(E),E\bigr).
\]
Since the frozen Lagrangian is convex and continuously differentiable with
respect to \((\rho,m)\) on the positive-density region, the first-order
optimality condition gives
\[
    \nabla_{(\rho,m)}
    \mathcal L_h
    \bigl(\rho^\star(E),m^\star(E),\phi^\star(E),E\bigr)
    =0.
\]
Similarly, $\phi^\star(E)
    \in
    \operatorname*{argmax}_{\phi\in\Sigma_h}
    \mathcal L_h
    \bigl(\rho^\star(E),m^\star(E),\phi,E\bigr)$.
The Lagrangian is affine in \(\phi\), and hence $\nabla_\phi
    \mathcal L_h
    \bigl(\rho^\star(E),m^\star(E),\phi^\star(E),E\bigr)
    =0$.
By Definition~\ref{def:frozen-optimality-operator}, these two identities are equivalent to~\eqref{eq:frozen-optimality-equation}. Conversely, assume that
\(A_E\bigl(z^\star(E)\bigr)=0\). Then
\[
    \nabla_{(\rho,m)}
    \mathcal L_h
    \bigl(\rho^\star(E),m^\star(E),\phi^\star(E),E\bigr)
    =0.
\]
By convexity of the frozen Lagrangian with respect to \((\rho,m)\), it follows
that, for every \((\rho,m)\in M_h\),
\[
    \mathcal L_h
    \bigl(\rho^\star(E),m^\star(E),\phi^\star(E),E\bigr)
    \leq
    \mathcal L_h
    \bigl(\rho,m,\phi^\star(E),E\bigr).
\]
Moreover, $\nabla_\phi \mathcal L_h
\bigl(\rho^\star(E),m^\star(E),\phi^\star(E),E\bigr)
=0$. Since the dependence on \(\phi\) is affine, this implies that, for every
\(\phi\in\Sigma_h\), $\mathcal L_h \bigl(\rho^\star(E),m^\star(E),\phi,E\bigr) = \mathcal L_h \bigl(\rho^\star(E),m^\star(E),\phi^\star(E),E\bigr)$.
Therefore,
\[
    \mathcal L_h
    \bigl(\rho^\star(E),m^\star(E),\phi,E\bigr)
    \leq
    \mathcal L_h
    \bigl(\rho^\star(E),m^\star(E),\phi^\star(E),E\bigr)
    \leq
    \mathcal L_h
    \bigl(\rho,m,\phi^\star(E),E\bigr),
\]
which is precisely the saddle-point property.
\end{proof}
For a fixed \(E\in W_h\), the saddle point characterized in
Lemma~\ref{lem:frozen-saddle-characterization} is approximated by the
primal--dual scheme of Chambolle and Pock. Since Section~\ref{sec:abstractconv} only specifies the
use of this method at the abstract level, we now record its specialization to
the discrete frozen-\(E\) problem. This specialization will be used below to define the computable fixed-point residual associated with the inner iteration. Given \((\rho^j,m^j,\phi^j)\in M_h\times\Sigma_h\), set $\bar\phi^j\coloneqq 2\phi^j-\phi^{j-1}$.
The next iterate is defined by
\begin{equation}
\label{eq:frozen-E-CP-iteration}
\left\{
\begin{aligned}
(\rho^{j+1},m^{j+1})
&=
\operatorname*{argmin}_{(\rho,m)\in M_h}
\left\{
\mathcal L_h(\rho,m,\bar\phi^j,E)
+
\frac12
\left\|(\rho,m)-(\rho^j,m^j)\right\|_{M_h}^{2}
\right\},
\\[1mm]
\phi^{j+1}
&=
\operatorname*{argmax}_{\phi\in\Sigma_h}
\left\{
\mathcal L_h(\rho^{j+1},m^{j+1},\phi,E) -
\frac12
\left\|\phi-\phi^j\right\|_{\Sigma_h}^{2}
\right\}.
\end{aligned}
\right.
\end{equation}
We use the fixed-point residual introduced in Section~\ref{sec:abstractconv} and defined in
\eqref{KKTresidual}, specialized to the iterates generated by
\eqref{eq:frozen-E-CP-iteration}.

\begin{lemma}[A posteriori estimator]
\label{lem:kkt-by-res}
Let \(E\in W_h\) be fixed. The inner PDHG iterates
\[
    z^j=(\rho^j,m^j,\phi^j)\in M_h\times\Sigma_h
\]
generated by~\eqref{eq:frozen-E-CP-iteration} satisfy
\begin{equation}
\label{eq:kkt-by-res-final}
    \|A_E(z^{j+1})\|_{(M_h\times\Sigma_h)^*}
    \leq
    2\sqrt{2}\bigl(\mathrm{Res}_{j+1}+\mathrm{Res}_{j}\bigr),
    \qquad j\in\mathbb N_*,
\end{equation}
where \(\mathrm{Res}_j\) denotes the fixed-point residual introduced
in~\eqref{KKTresidual}.
\end{lemma}
As a consequence, since Proposition \ref{prop:verification-assumption-one}, Lemma~\ref{lem:frozen-E-existence} and Corollary~\ref{cor:uniform-stability} ensure that the assumptions of the inner convergence result of Section~\ref{sec:abstractconv} are satisfied, thus, \(z^j\) converges to a frozen-\(E\) saddle point and \(\mathrm{Res}_j\to0\). 
Consequently, for every \(\varepsilon>0\), the stopping index \[ j_\varepsilon \coloneqq  \min\bigl\{ j\in\mathbb N_*: \mathrm{Res}_j+\mathrm{Res}_{j-1}\leq\varepsilon \bigr\} \] is finite, and the accepted iterate satisfies $ \|A_E(z^{j_\varepsilon})\|_{(M_h\times\Sigma_h)^*}    \leq 2\sqrt{2} \varepsilon $.
\begin{proof} Set
$\delta\rho^j\coloneqq \rho^{j+1}-\rho^j,
    \qquad
    \delta m^j\coloneqq m^{j+1}-m^j,
    \qquad
    \delta\phi^j\coloneqq \phi^{j+1}-\phi^j$.
The primal and dual iterates \((\rho^{j+1},m^{j+1})\) and
\(\phi^{j+1}\) solve exactly the proximal subproblems
in~\eqref{eq:frozen-E-CP-iteration}, with $\bar\phi^{j+1}=2\phi^{j+1}-\phi^j$. Thus,
\begin{align*}
(\rho^{j+1},m^{j+1})
&=
\operatorname*{argmin}_{(\rho,m)\in M_h}
\left\{
\mathcal L_h(\rho,m,\bar\phi^j,E)
+
\frac12
\|(\rho,m)-(\rho^j,m^j)\|_{M_h}^2
\right\},
\\
\phi^{j+1}
&=
\operatorname*{argmax}_{\phi\in\Sigma_h}
\left\{
\mathcal L_h(\rho^{j+1},m^{j+1},\phi,E)
-
\frac12
\|\phi-\phi^j\|_{\Sigma_h}^2
\right\}.
\end{align*}

The first-order optimality condition for the primal proximal step gives
\begin{equation}
\label{eq:oc-primal}
    0
    =
    \nabla_{(\rho,m)}F_h(\rho^{j+1},m^{j+1})
    +
    K_h(\cdot;E)^*\bar\phi^j
    +
    \mathbf T^{-1}(\delta\rho^j,\delta m^j).
\end{equation}
Similarly, optimality of the dual proximal step yields
\begin{equation}
\label{eq:oc-dual}
    0
    =
    \nabla_\phi G_h^*(\phi^{j+1})
    -
    K_h(\rho^{j+1},m^{j+1};E)
    +
    \mathbf S\,\delta\phi^j.
\end{equation}

By~\eqref{eq:frozen-optimality-operator},
\[
    A_E(z^{j+1})
    =
    \begin{pmatrix}
    \nabla_{(\rho,m)}F_h(\rho^{j+1},m^{j+1})
    +
    K_h(\cdot;E)^*\phi^{j+1}
    \\[1mm]
    \nabla_\phi G_h^*(\phi^{j+1})
    -
    K_h(\rho^{j+1},m^{j+1};E)
    \end{pmatrix}.
\]
Using~\eqref{eq:oc-primal}, its first component becomes
\[
\begin{aligned}
&\nabla_{(\rho,m)}F_h(\rho^{j+1},m^{j+1}) + K_h(\cdot;E)^*\phi^{j+1} =
K_h(\cdot;E)^*
\bigl(\phi^{j+1}-\bar\phi^j\bigr)
-
\mathbf T^{-1}(\delta\rho^j,\delta m^j).
\end{aligned}
\]
Since \(\bar\phi^j=2\phi^j-\phi^{j-1}\), we have $\phi^{j+1}-\bar\phi^j= \delta\phi^j-\delta\phi^{j-1}$.
Hence,
\[
\begin{aligned}
&\nabla_{(\rho,m)}F_h(\rho^{j+1},m^{j+1}) + K_h(\cdot;E)^*\phi^{j+1} = K_h(\cdot;E)^*
\bigl(\delta\phi^j-\delta\phi^{j-1}\bigr)
-
\mathbf T^{-1}(\delta\rho^j,\delta m^j).
\end{aligned}
\]
Using~\eqref{eq:oc-dual}, the second component is $\nabla_\phi G_h^*(\phi^{j+1}) - K_h(\rho^{j+1},m^{j+1};E) = -\mathbf S\,\delta\phi^j$.
Taking the product dual norm and using the triangle inequality, we obtain
\begin{align*}
\|A_E(z^{j+1})\|_{(M_h\times\Sigma_h)^*}
&\leq \left\| \mathbf T^{-1}(\delta\rho^j,\delta m^j)\right\|_{M_h^*} + \left\| K_h(\cdot;E)^* \bigl(\delta\phi^j-\delta\phi^{j-1}\bigr) \right\|_{M_h^*} +
\|\mathbf S\,\delta\phi^j\|_{\Sigma_h^*}.
\end{align*}
By the definitions of the norms on \(M_h\) and \(\Sigma_h\),
\[
    \left\|
    \mathbf T^{-1}(\delta\rho^j,\delta m^j)
    \right\|_{M_h^*}
    =
    \|(\delta\rho^j,\delta m^j)\|_{M_h},
    \qquad
    \|\mathbf S\,\delta\phi^j\|_{\Sigma_h^*}
    =
    \|\delta\phi^j\|_{\Sigma_h}.
\]
Moreover, $\left\|
K_h(\cdot;E)^*
\bigl(\delta\phi^j-\delta\phi^{j-1}\bigr)
\right\|_{M_h^*}
\leq
\|K_h(\cdot;E)\|_{\mathcal L(M_h,\Sigma_h^*)}
\,
\|\delta\phi^j-\delta\phi^{j-1}\|_{\Sigma_h}$.
By Corollary~\ref{cor:uniform-stability}, $\left\| K_h(\cdot;E)^* \bigl(\delta\phi^j-\delta\phi^{j-1}\bigr) \right\|_{M_h^*} \leq \|\delta\phi^j-\delta\phi^{j-1}\|_{\Sigma_h}$.
Combining these estimates gives
\begin{align*}
\|A_E(z^{j+1})\|_{(M_h\times\Sigma_h)^*}
&\leq
\|(\delta\rho^j,\delta m^j)\|_{M_h}
+
2\|\delta\phi^j\|_{\Sigma_h}
+
\|\delta\phi^{j-1}\|_{\Sigma_h}
\\
&\leq
2\Bigl(
\|(\delta\rho^j,\delta m^j)\|_{M_h}
+
\|\delta\phi^j\|_{\Sigma_h}
\\
&\hspace{20mm}
+
\|(\delta\rho^{j-1},\delta m^{j-1})\|_{M_h}
+
\|\delta\phi^{j-1}\|_{\Sigma_h}
\Bigr).
\end{align*}
Using the definition of
\(\mathrm{Res}_j\) in~\eqref{KKTresidual},  we obtain
\[ \|A_{E}(z^{j+1})\|_{(M_h\times\Sigma_h)^*} 
\le 2\sqrt{2}(\mathrm{Res}_{j+1}+ \mathrm{Res}_{j}).
\] \qedhere
\end{proof}
The following corollary records the specialization of the inner convergence result from Section~\ref{sec:abstractconv}, Proposition \ref{prop:frozen-inner-loop} to the discrete frozen-\(E\) problem.


We now verify Assumption~\ref{ass:regular-reference-saddle} at a reference  solution. We first record
two consequences of the fixed finite-dimensional discretization.

\begin{lemma}[Coercivity of the centered time difference]
\label{lem:coercivity_discrete_time}
Assume that the number \(K\) of time slabs is even. Then there exists a
constant \(c_h>0\), depending only on the fixed discretization, such that
\begin{equation}
\label{eq:coercivity_discrete_time}
    c_h\|(\delta\rho,0)\|_{M_h}
    \leq
    \|\partial_t\delta\rho\|_{\Sigma_h^*}
\end{equation}
for every density variation \((\delta\rho,0)\in M_h\) satisfying $\delta\rho_0=\delta\rho_{K-1}=0$.
\end{lemma}

\begin{proof}
It is enough to show that the centered discrete time derivative is injective
on the subspace of density variations satisfying the homogeneous endpoint
conditions. Let \(\delta\rho\) belong to this subspace and suppose that $\|\partial_t\delta\rho\|_{\Sigma_h^*}=0$. 
Then $\bigl\langle \partial_t\delta\rho,\phi\bigr\rangle_{L^2(Q)} =0$ for every $\phi \in \Sigma_h$. By the compatibility Lemma \ref{lem:constraint-compatibility}, this implies $\partial_t \delta \rho=0$ cellwise.
By the definition of the centered time difference,
\[
    \delta\rho_{k+1}=\delta\rho_{k-1},
    \qquad k=1,\ldots,K-2.
\]
Hence, all even time components coincide with \(\delta\rho_0\), whereas
all odd time components coincide with \(\delta\rho_1\). Since
\(\delta\rho_0=0\), all even components vanish. Moreover, \(K\) is even, so
\(K-1\) is odd, and therefore
\[
    0=\delta\rho_{K-1}=\delta\rho_1.
\]
Hence all odd components vanish as well, and thus \(\delta\rho=0\). The restricted centered time-difference operator is therefore injective. Since its domain is finite-dimensional, its smallest singular value with
respect to the chosen discrete norms is strictly positive. This gives
\eqref{eq:coercivity_discrete_time}.
\end{proof}

\begin{lemma}[Boundedness of the discrete transport term]
\label{lem:boundedness_discrete_transport}
There exists a constant \(C_h>0\), depending only on the fixed discretization,
such that
\begin{equation}
\label{eq:discrete_transport_bound}
    \left\|
        \nabla_x\!\cdot\!\bigl(b\,\delta\rho\bigr)
    \right\|_{\Sigma_h^*}
    \leq
    C_h\|b\|_{L^\infty(Q)}
    \|(\delta\rho, 0)\|_{M_h}
\end{equation}
for every discrete vector field \(b\) and every density variation
\((\delta\rho,0)\in M_h\).
\end{lemma}

\begin{proof}
Using the discrete integration-by-parts formula, for every
\(\psi\in\Sigma_h\),
\[
    \left\langle
        \nabla_x\!\cdot\!\bigl(b\,\delta\rho\bigr),\psi
    \right\rangle
    =
    -\left\langle
        b\,\delta\rho,\nabla_x\psi
    \right\rangle.
\]
Therefore,
\[
    \left|
        \left\langle
            \nabla_x\!\cdot\!\bigl(b\,\delta\rho\bigr),\psi
        \right\rangle
    \right|
    \leq
   \sqrt{\tau_\rho} \|b\|_{L^\infty(Q)}
    \|(\delta\rho,0)\|_{M_h}
    \|\nabla_x\psi\|.
\]
Since the discretization is fixed and the spaces are finite-dimensional,
there exists \(C_h= \sqrt{\tau_\phi \tau_\rho }>0\) such that 
\eqref{eq:discrete_transport_bound}.
\end{proof}

\begin{prop}[A sufficient condition for Assumption~\ref{ass:regular-reference-saddle}]
\label{prop:verification_assumption_2}
Let $\bigl(E^\star,z^\star(E^\star)\bigr) = \bigl( E^\star, \rho^\star(E^\star), m^\star(E^\star), \phi^\star(E^\star) \bigr)$
be a solution of the discrete problem. Assume that \(K\) is even, the free components
of \(\rho^\star(E^\star)\) are strictly positive, and 
\begin{equation}
\label{eq:effective_velocity_smallness}
    C_h
    \left\|
        \frac{m^\star(E^\star)}
             {\rho^\star(E^\star)}
        -
        \Pi_0g_{E^\star}
    \right\|_{L^\infty(Q)}
    <c_h,
\end{equation}
where \(c_h\) and \(C_h\) are the constants from
Lemmas~\ref{lem:coercivity_discrete_time}
and~\ref{lem:boundedness_discrete_transport}, respectively. Then
\[
    D_zA_{E^\star}\bigl(z^\star(E^\star)\bigr)
    \colon Z_h\longrightarrow Z_h^*
\]
is an isomorphism. Consequently, Assumption~\ref{ass:regular-reference-saddle} is satisfied at
\(\bigl(E^\star,z^\star(E^\star)\bigr)\).
\end{prop}

\begin{proof}
Let $(\delta\rho,\delta m,\delta\phi)\in Z_h$ satisfy
\begin{equation}
\label{eq:kernel_optimality_derivative}
    D_zA_{E^\star}\bigl(z^\star(E^\star)\bigr)
    (\delta\rho,\delta m,\delta\phi)=0.
\end{equation}
By the definition of \(A_E\), equation
\eqref{eq:kernel_optimality_derivative} is equivalent to
\begin{align}
    D^2F_h\bigl(
        \rho^\star(E^\star),m^\star(E^\star)
    \bigr)
    (\delta\rho,\delta m)
    +
    K_h(\cdot;E^\star)^*\delta\phi
    &=0,
    \label{eq:kernel_first_block}
    \\
    K_h(\delta\rho,\delta m;E^\star)
    &=0.
    \label{eq:kernel_second_block}
\end{align}
Taking the inner product of  \eqref{eq:kernel_first_block} with
\((\delta\rho,\delta m)\) and using
\eqref{eq:kernel_second_block}, we obtain
\[
    \left\langle
        D^2F_h\bigl(
            \rho^\star(E^\star),m^\star(E^\star)
        \bigr)
        (\delta\rho,\delta m),
        (\delta\rho,\delta m)
    \right\rangle
    =0.
\]
Since the free components of \(\rho^\star(E^\star)\) are strictly positive,
the Hessian of the discrete kinetic functional is given by
\[
    \int_Q
    \frac{1}{2\rho^\star(E^\star)}
    \left|
        \delta m
        -
        \frac{m^\star(E^\star)}
             {\rho^\star(E^\star)}
        \delta\rho
    \right|^2
    =0.
\]
It follows that
\begin{equation}
\label{eq:hessian_kernel_relation}
    \delta m
    =
    \frac{m^\star(E^\star)}
         {\rho^\star(E^\star)}
    \delta\rho.
\end{equation}

Substituting \eqref{eq:hessian_kernel_relation} into
\eqref{eq:kernel_second_block} and using Lemma~\ref{lem:constraint-compatibility} of the
discrete continuity constraint yields
\begin{equation}
\label{eq:homogeneous_linearized_continuity}
    \partial_t\delta\rho
    +
    \nabla_x\!\cdot\!\left[
        \left(
            \frac{m^\star(E^\star)}
                 {\rho^\star(E^\star)}
            -
            \Pi_0g_{E^\star}
        \right)
        \delta\rho
    \right]
    =0.
\end{equation}
Since the endpoint densities are fixed, $\delta\rho_0=\delta\rho_{K-1}=0$. Applying Lemmas~\ref{lem:coercivity_discrete_time}
and~\ref{lem:boundedness_discrete_transport} to
\eqref{eq:homogeneous_linearized_continuity}, we obtain
\begin{align*} 
c_h\|\delta\rho\|_{M_h}
&\leq \|\partial_t\delta\rho\|_{\Sigma_h^*} = \left\|   \nabla_x\!\cdot\!\left[  \left(    \frac{m^\star(E^\star)} {\rho^\star(E^\star)}-               \Pi_0g_{E^\star} \right) \delta\rho \right] \right\|_{\Sigma_h^*} \\
    &\leq C_h \left\| \frac{m^\star(E^\star)}           {\rho^\star(E^\star)} - \Pi_0g_{E^\star} \right\|_{L^\infty(Q)}   \|\delta\rho\|_{M_h}.
\end{align*}
By condition \eqref{eq:effective_velocity_smallness}, it follows that $\delta\rho=0$. Equation \eqref{eq:hessian_kernel_relation} then gives $\delta m=0$. The first block equation consequently reduces to $ K_h(\cdot;E^\star)^*\delta\phi=0$.
Testing this identity against momentum variations gives
\(\nabla_x\delta\phi=0\).
Testing it against density variations satisfying the homogeneous
endpoint conditions gives
\[
\langle\partial_t\delta\rho,\delta\phi\rangle_{L^2(Q)}=0
\qquad\text{for every admissible }\delta\rho.
\]
Equivalently, the free temporal components of
\(\partial_{t,h}^*\delta\phi\) vanish. By the explicit form of the
centered-difference stencil, all interior temporal components of
\(\delta\phi\) therefore vanish. Since
\(\delta\phi\in\Sigma_h\), the two parity gauge conditions also force
the two endpoint components to vanish. Hence \(\delta\phi=0\).
Thus, \[ \ker    D_zA_{E^\star}\bigl(z^\star(E^\star)\bigr) = \{0\}. \] Since \(Z_h\) and \(Z_h^*\) have the same finite dimension, \(D_zA_{E^\star}\bigl(z^\star(E^\star)\bigr)\) is an isomorphism.
\end{proof}

\begin{rmk}
Condition \eqref{eq:effective_velocity_smallness} is the following smallness condition on the effective non-rigid velocity of the selected solution. This condition is expected to be satisfied when two endpoints density are close to being related by a rigid motion. 
\end{rmk}
As a consequence of Propositions~\ref{prop:verification-assumption-one} and~\ref{prop:verification_assumption_2}, the abstract local regularity result from Lemma \ref{lem:abstract-local-branch} applies to the discrete frozen-\(E\) problem.
\begin{corollary}[Local saddle branch and reliability estimate]
\label{cor:local_branch_reliability_discrete}
Let $\bigl(E^\star,z^\star(E^\star)\bigr)$ be the reference solution considered in Proposition~\ref{prop:verification_assumption_2}. Then there exist
neighborhoods $U_{E^\star}\subset W_h$ and $U_{z^\star}\subset Z_h$ of \(E^\star\) and \(z^\star(E^\star)\), respectively, and a locally unique
mapping $z^\star(\,\cdot\,) \colon U_{E^\star}\longrightarrow U_{z^\star}$ such that $A_E\bigl(z^\star(E)\bigr)=0$ \text{for every }$E\in U_{E^\star}$. Moreover, there exists a constant \(L_\star>0\) such that
\begin{equation}
\label{eq:lipschitz_discrete_saddle_branch}
    \bigl\|
        z^\star(E_1)-z^\star(E_2)
    \bigr\|_{Z_h}
    \leq
    L_\star\|E_1-E_2\|_{W_h} \qquad\text{for all } E_1,E_2\in U_{E^\star}.
\end{equation}
In addition, there exists a constant \(C_{\mathrm{reg}}>0\) such that
\begin{equation}
\label{eq:discrete_reliability_estimate}
    \bigl\|
        z-z^\star(E)
    \bigr\|_{Z_h}
    \leq
    C_{\mathrm{reg}}
    \|A_E(z)\|_{Z_h^*} \qquad \text{for all } (E,z)\in U_{E^\star}\times U_{z^\star}
\end{equation}
\end{corollary}

\begin{proof}
Proposition~\ref{prop:verification-assumption-one} verifies the smoothness and
well-posedness requirements of Assumption~\ref{ass:abstract-standing} for the discrete frozen-\(E\)
problem, while Proposition~\ref{prop:verification_assumption_2} shows that $D_zA_{E^\star}\bigl(z^\star(E^\star)\bigr)$ is an isomorphism. The conclusions therefore follow directly from the local
branch and reliability result established in Section~\ref{sec:abstractconv}.
\end{proof}

We can now specialize the abstract conditional convergence result to the fully discrete modified Benamou--Brenier problem with a fixed grid.

\begin{thm}[Conditional subsequential convergence of the discrete scheme]
\label{thm:conditional_convergence_discrete_scheme}
Let
\[
    \bigl(E^\star,z^\star(E^\star)\bigr)
\]
be a reference solution satisfying the hypotheses of
Proposition~\ref{prop:verification_assumption_2}, and let
\(U_{E^\star}\) and \(U_{z^\star}\) be the neighborhoods furnished by
Corollary~\ref{cor:local_branch_reliability_discrete}. Consider the discrete alternating scheme in which, at every outer iteration
\(n\), the frozen-\(E^n\) Chambolle--Pock loop is stopped at the first index
for which
\begin{equation}
\label{eq:accepted_inner_residual}
    \mathrm{Res}_{j_n}+ \mathrm{Res}_{j_n-1}\le\varepsilon_n \implies \bigl\|A_{E^n}(z^{n+1})\bigr\|_{Z_h^*}
    \leq 2\sqrt{2} \varepsilon_n,
\end{equation}
and the outer variable is then updated according to the inexact gradient step
introduced in Section~\ref{sec:abstractconv} Algorithm \ref{alg:alternating-saddle}. Assume that the step-size and tolerance hypotheses of
Theorem~\ref{thm:local-subsequential-convergence} are satisfied and that the generated iterates remain inside the local regularity neighbourhood
\begin{equation}
\label{eq:discrete_local_regime}
    E^n\in U_{E^\star},
    \qquad
    z^{n+1}\in U_{z^\star}
    \qquad
    \text{for every }n\geq0.
\end{equation}
Then the accepted inner iterates satisfy
\begin{equation}
\label{eq:accepted_inner_error}
    \bigl\|
        z^{n+1}-z^\star(E^n)
    \bigr\|_{Z_h}
    \leq
    2\sqrt{2} C_{\mathrm{reg}}\varepsilon_n
    \qquad
    \text{for every }n\geq0.
\end{equation}
In particular, $\bigl\| z^{n+1}-z^\star(E^n) \bigr\|_{Z_h} \longrightarrow 0$.
Moreover, the outer increments are square summable:
\begin{equation}
\label{eq:outer_increment_square_summability}
    \sum_{n=0}^{\infty}
    \bigl\|E^{n+1}-E^n\bigr\|_{W_h}^{2}
    <\infty.
\end{equation}
Consequently,
\begin{equation}
\label{eq:outer_increment_vanishing}
    \bigl\|E^{n+1}-E^n\bigr\|_{W_h}
    \longrightarrow0.
\end{equation}
Since \(W_h\) is finite-dimensional and \((E^n)_n\) remains in \(U_{E^\star}\), the sequence admits an accumulation point. For every subsequence \((E^{n_j})_j\) such that $E^{n_j}\longrightarrow\bar E$, one has $z^{n_j+1} \longrightarrow z^\star(\bar E)$, and the pair $\bigl(\bar E,z^\star(\bar E)\bigr)$ satisfies the optimality conditions of the discrete modified
Benamou--Brenier problem. Thus every accumulation point of the generated sequence is a discrete block
stationary point.
\end{thm}

\begin{proof}
Proposition~\ref{prop:verification-assumption-one} verifies Assumption~\ref{ass:abstract-standing} of the
abstract framework, whereas
Proposition~\ref{prop:verification_assumption_2} verifies Assumption~\ref{ass:regular-reference-saddle} at the
reference solution. Hence
Corollary~\ref{cor:local_branch_reliability_discrete} provides the locally
unique saddle branch \(E\mapsto z^\star(E)\) and the reliability estimate
\[
    \bigl\|
        z-z^\star(E)
    \bigr\|_{Z_h}
    \leq
    C_{\mathrm{reg}}
    \|A_E(z)\|_{Z_h^*}
\]
on \(U_{E^\star}\times U_{z^\star}\).

Applying this estimate with \(E=E^n\), \(z=z^{n+1}\), and using
\eqref{eq:accepted_inner_residual}, gives
\eqref{eq:accepted_inner_error}. All the remaining hypotheses of Theorem~\ref{thm:local-subsequential-convergence} are satisfied by
assumption, and therefore \eqref{eq:outer_increment_square_summability},
\eqref{eq:outer_increment_vanishing}, and the stationarity of every accumulation point follow directly from that theorem. Finally, if \(E^{n_j}\to\bar E\), the Lipschitz continuity of the local saddle
branch and \eqref{eq:accepted_inner_error} yield
\begin{align*}
    \bigl\|
        z^{n_j+1}-z^\star(\bar E)
    \bigr\|_{Z_h}
    &\leq
    \bigl\|
        z^{n_j+1}-z^\star(E^{n_j})
    \bigr\|_{Z_h}
    +
    \bigl\|
        z^\star(E^{n_j})-z^\star(\bar E)
    \bigr\|_{Z_h}
    \\
    &\leq
    2\sqrt{2} C_{\mathrm{reg}}\varepsilon_{n_j}
    +
    L_\star
    \bigl\|
        E^{n_j}-\bar E
    \bigr\|_{W_h}
    \longrightarrow0.
\end{align*}\qedhere
\end{proof}
\begin{rmk}
We would like to point out that Theorem \ref{thm:conditional_convergence_discrete_scheme} is conditional based on the ``local regime assumption". 
In numerical implementations, this local regime can be monitored or
encouraged by safeguards such as reducing the outer step size, regularization, rejecting
updates that leave the admissible region, monitoring positivity and residuals, or using a trust-region restriction on the outer step.

The iteration can be initialized  by a coarse-grid warm start. Nested refinement strategies of this type are standard in nonlinear optimization and multilevel optimization methods \cite{brandt1977multi,nash2000multigrid} and aim at producing initial iterates lying in a stable local regime where the convergence analysis applies.
\end{rmk}

\begin{rmk}[Scope of the fixed-grid analysis]
The convergence result established in this section concerns the finite-dimensional saddle problem associated with a fixed spatial--temporal grid. It is local and conditional: the iterates are assumed to remain in a neighborhood of a regular discrete saddle branch, and the accepted inner iterates are required to satisfy the stated residual conditions. We do not establish convergence to a global minimizer of the discrete problem, nor consistency of the discretization or convergence of discrete solutions to solutions of the continuous MBB problem as the mesh is refined. The numerical experiments should therefore be interpreted as illustrations of the behavior of the alternating scheme on the selected fixed grid.
\end{rmk}

\section{Numerical Experiments}
\label{sec:numerics}

This section presents numerical illustrations of the proposed Modified Benamou--Brenier (MBB)
formulation and the corresponding solver. 
We consider three representative scenarios: (i) a {pure rotation} case, where the motion is entirely explained by the orthogonal component (\S~\ref{subsec:pure-rotation}); (ii) an {effective transport} case, where no exploitable orthogonal structure is present and the orthogonal component remains negligible, so that the MBB formulation recovers the classical BB behaviour (\S ~\ref{subsec:effective-transport}); and (iii) a {mixed rotation--deformation} case, in which both the advective velocity field and the orthogonal component contribute to the transport (\S~\ref{subsec:mixed}).

\subsection{General setup and implementation details}

All computations are performed on a uniform Cartesian grid on the domain \(\Omega=(-0.5,0.5)^2\), equipped with no-flux (homogeneous Neumann) boundary conditions. Unless stated otherwise, we use \(n_x=n_y=n_t=64\) grid points in space and time, with time step \(h_t=1/(n_t-1)\). The same inner primal--dual hybrid gradient (PDHG) parameters are used across all experiments.

The primal--dual step sizes \((\tau_\rho,\tau_m,\tau_\phi)\) are adapted using the residual-balancing strategy proposed in \cite{goldstein2013adaptive}, which promotes balanced primal and dual progress and robust numerical behaviour. The drift proximal step size \(\tau_E\) is kept fixed in the rotation-dominant and mixed benchmarks, and is mildly reduced in the rotation-free case to suppress spurious drift and recover the classical Benamou--Brenier behaviour. The dual variable \(\phi\) is updated using an \(H^1(I\times\Omega)\) proximal operator, which provides the grid-robust preconditioning properties established in \cite{jacobs2019solving}. The numerical scheme is fully mass conservative: the total mass of each density \(\rho_t\) is preserved at every time slice up to machine precision, without explicit renormalization cf Lemma \ref{lem:constraint-compatibility}--\ref{lem:discrete-mass-conservation}. The implementation follows the alternating scheme analyzed in Sections \ref{sec:abstractconv}--\ref{sec:conv}; the practical differences are discussed in Section \ref{subsec:generalremarkreproduc}.

For each outer iterate \(E^n\), the corresponding frozen inner saddle problem is approximated by running the PDHG solver up to a prescribed inner iteration cap. The resulting inner output is then used to update the skew-symmetric field. The complete alternating scheme is terminated when either the outer iterate-difference criterion, defined as the sum of the \(L^2(I\times\Omega)\)-differences for \((\rho,m)\) and the \(H^1(I\times\Omega)\)-difference for \(\phi\) between successive outer iterations, falls below \(10^{-3}\), or a prescribed outer iteration cap is reached. This practical outer stopping criterion is distinct from the inner fixed-point residual used in the convergence analysis.

For the MBB formulation, the skew-symmetric field \(E_t\) is updated via an
exact proximal step and projected onto \(\mathfrak{so}(2)\) at each outer
iteration. For visualization purposes only, the associated orthogonal path is
reconstructed a posteriori through the exponential time-stepping scheme
\[
    \Theta_{k+1}
    = \exp(h_tE_k)\Theta_k
    \qquad\text{with}\qquad 
    \Theta_0=I,
\]
which discretizes the matrix ODE
\[
    \partial_t\Theta_t=E_t\Theta_t
\]
previously studied in Lemma \ref{lem:rigid-flow} and preserves \(\Theta_k\in SO(2)\) at every time step. All experiments are
implemented in Python using \texttt{NumPy}.

\vspace{1em}
\subsection{Pure Rotational Transport}
\label{subsec:pure-rotation}

\paragraph{Test distribution.}
We first consider a benchmark designed to isolate the effect of a global rigid motion.
The source and target densities, shown in Fig.~\ref{fig:pure_rho0rhoT}, consist of the same
double-crescent distribution, with the target obtained by applying a rigid rotation of
approximately $45^\circ$ about the origin.
Both densities are smooth, compactly supported in $\Omega=(-0.5,0.5)^2$, and remain well
separated from the boundary, so that the exact transport between them corresponds to a
pure orthogonal motion without deformation.
This setting provides a controlled test to assess whether the MBB formulation can
correctly identify and encode a global rotation.

\begin{figure}[H]
    \centering
    \includegraphics[width=.85\linewidth]{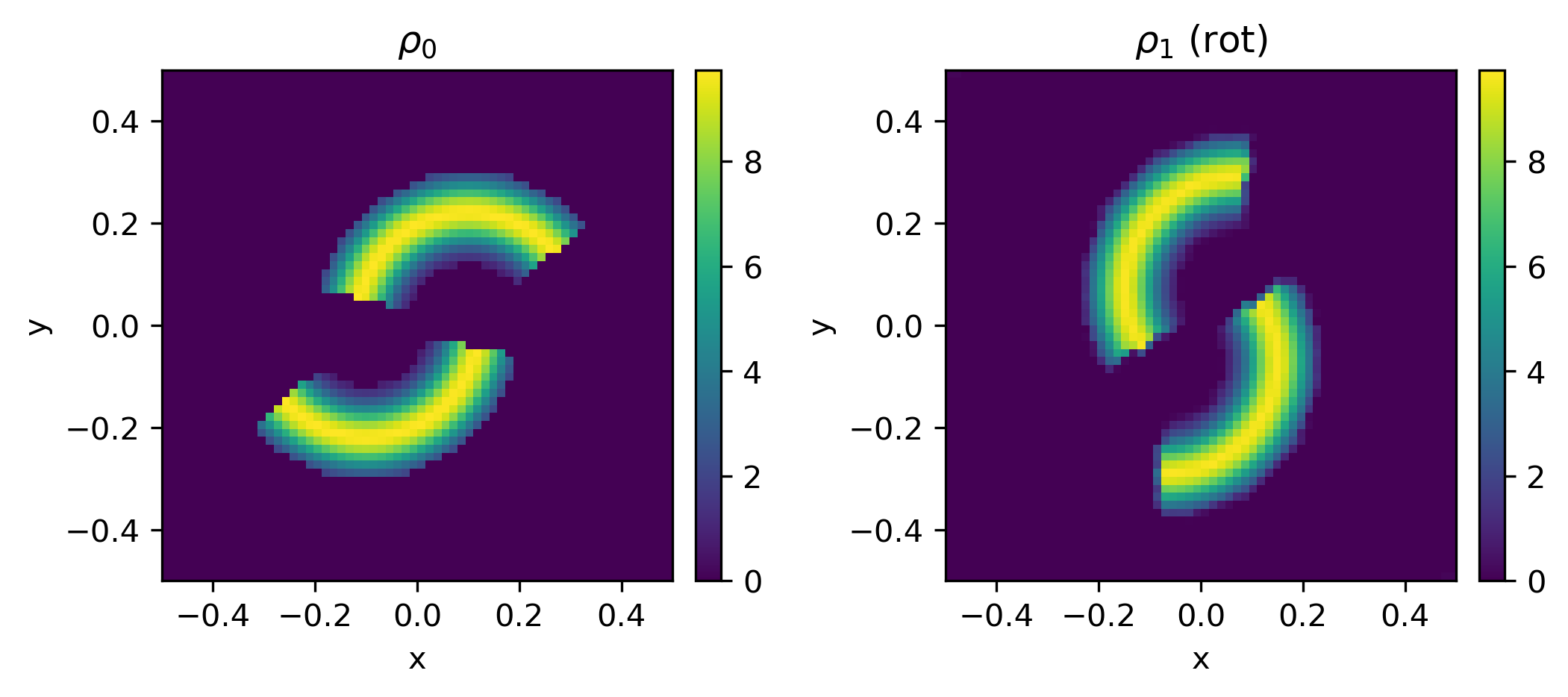}
    \caption{Pure rotation test: source $\rho_0$ and target $\rho_1$ differing by a rigid $45^\circ$ rotation.}
    \label{fig:pure_rho0rhoT}
\end{figure}

\paragraph{Evolution of densities.}
In Figure~\ref{fig:PR_compare_MBB_BB}  the top row displays the evolution of the interpolating density under
the MBB dynamics, while the bottom row shows the corresponding interpolation
obtained with the classical Benamou--Brenier (BB) formulation.
The MBB interpolation follows a coherent rotation of the entire density around the origin,
with negligible diffusion and noticeable preservation of the geometric structure.
In contrast, the BB interpolation exhibits visible shearing and mild blurring, as the
global rotation must be approximated through potential velocity fields rather than an
explicit orthogonal component.

\begin{figure}[H]
    \centering
    \includegraphics[width=1\linewidth]{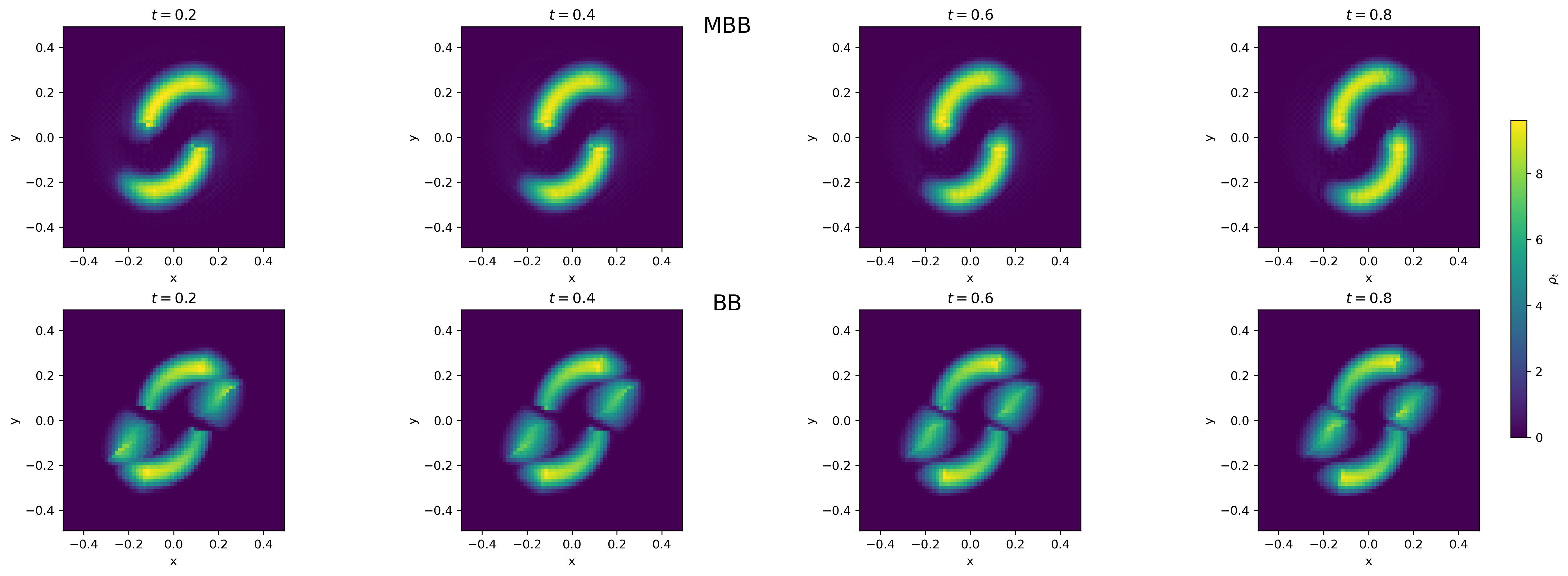}
\caption{Evolution of $\rho_t$ for the pure rotation experiment.}
    \label{fig:PR_compare_MBB_BB} 
\end{figure}

\paragraph{Convergence behaviour.}
Figure~\ref{fig:pure_residuals} reports the convergence of the fixed-point residual for both
formulations.
In both cases, the residual decreases rapidly during the first iterations and stabilizes
below the prescribed tolerance.
The comparable behaviour of the two curves suggests that, in this experiment, introducing the orthogonal variable does not adversely affect the numerical stability of the underlying primal--dual scheme.

\begin{figure}[H]
    \centering
    \begin{minipage}[t]{0.50\linewidth}
    \centering
    \includegraphics[width=\linewidth]{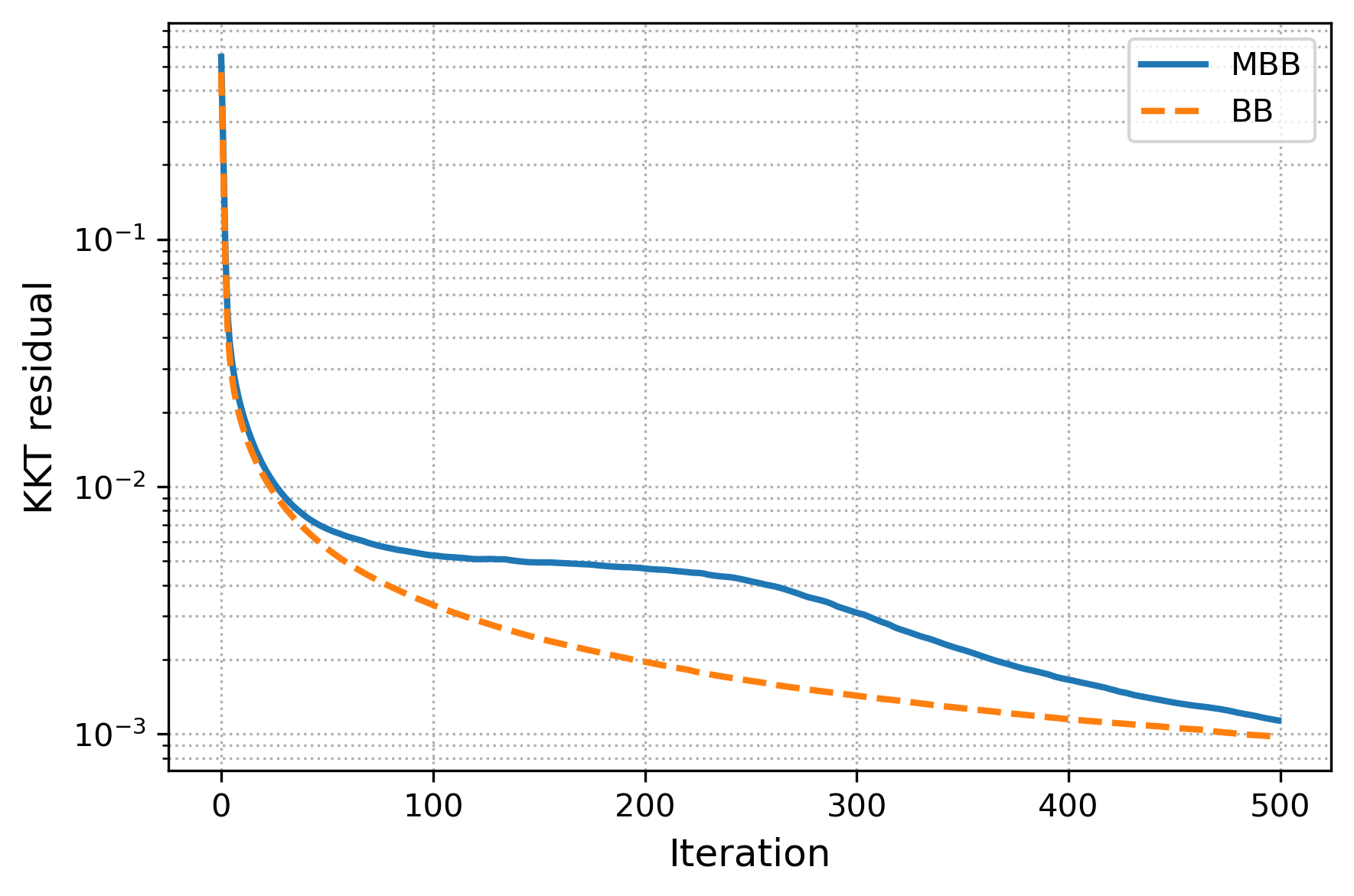}
    \caption{Convergence of residuals in the pure rotation experiment for both formulations.}
    \label{fig:pure_residuals}
    \end{minipage}
    \hfill
    \begin{minipage}[t]{0.48\linewidth}
    \centering
    \includegraphics[width=\linewidth]{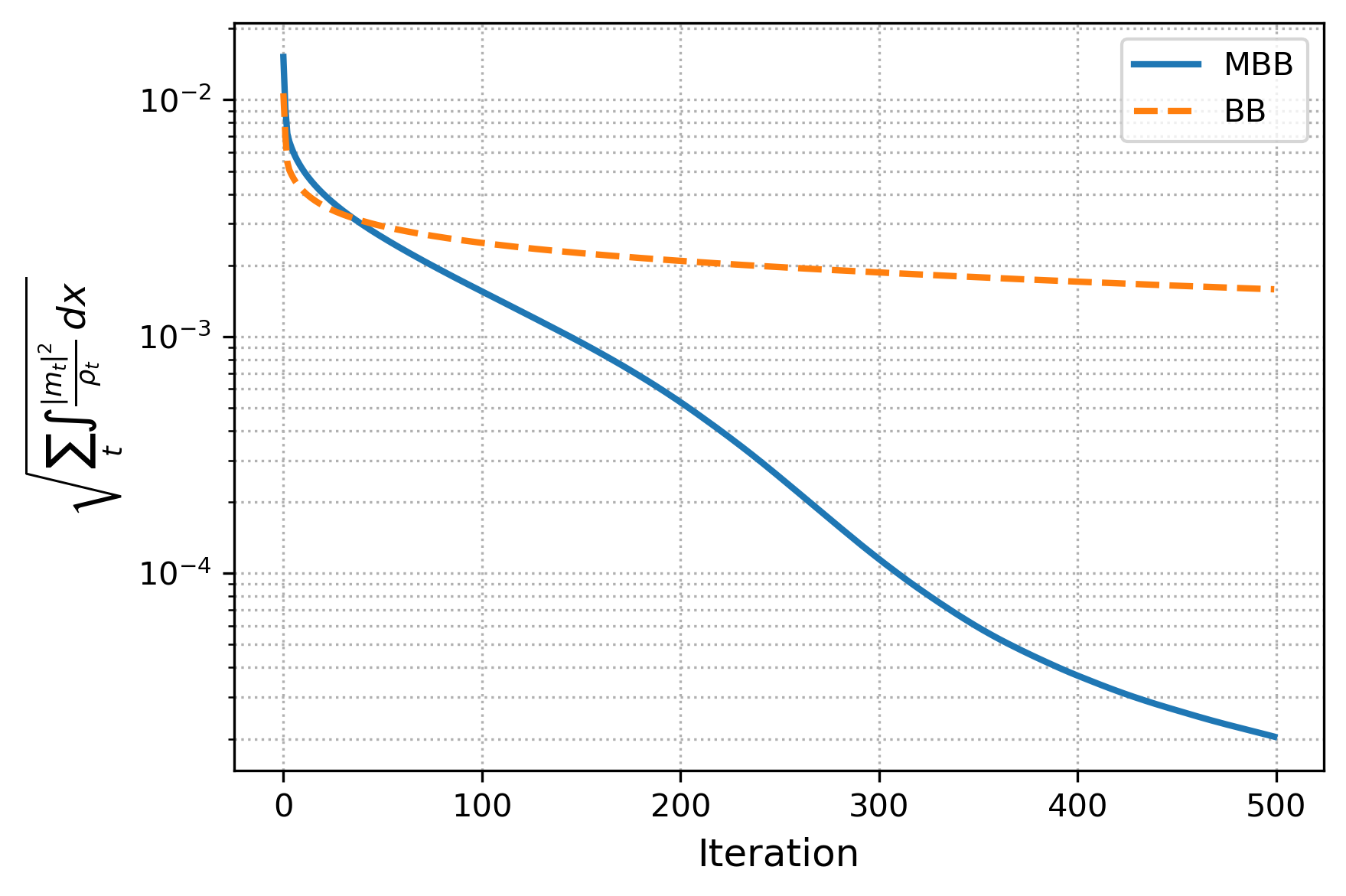}
    \caption{Total kinetic energy per iteration in the pure rotation case.}
    \label{fig:pure_momentum}
    \end{minipage}
\end{figure}

\paragraph{Momentum and orthogonal component.}
Figure~\ref{fig:pure_momentum} compares the total kinetic energy defined by $\int_0^1 \int_\Omega {|m_t|^2}/{\rho_t}\dx\dt$ for the MBB and BB formulations. In the MBB case, the kinetic energy decays rapidly and stabilizes at a value close to zero, indicating that the transport cost is almost entirely absorbed by the orthogonal component.
By contrast, the BB formulation maintains a significantly higher kinetic energy, as the
rotation must be realized through advective transport.

The evolution of the reconstructed orthogonal path $\Theta_t$ is shown in
Fig.~\ref{fig:pure_angle}.
The rotation angle increases smoothly and monotonically over time, reaching approximately
$33^\circ$ at $t=1$.
Although this value is below the nominal $45^\circ$ rotation used to generate the
target density, the qualitative behaviour is correctly captured.
The remaining discrepancy is absorbed by the transport field and may also reflect fixed finite inner solves,  temporal and spatial discretization and step-size effects affect the ability of the MBB formulation to identify and encode the global rigid motion.
\begin{figure}[H]
    \centering
     \includegraphics[width=.6\linewidth]{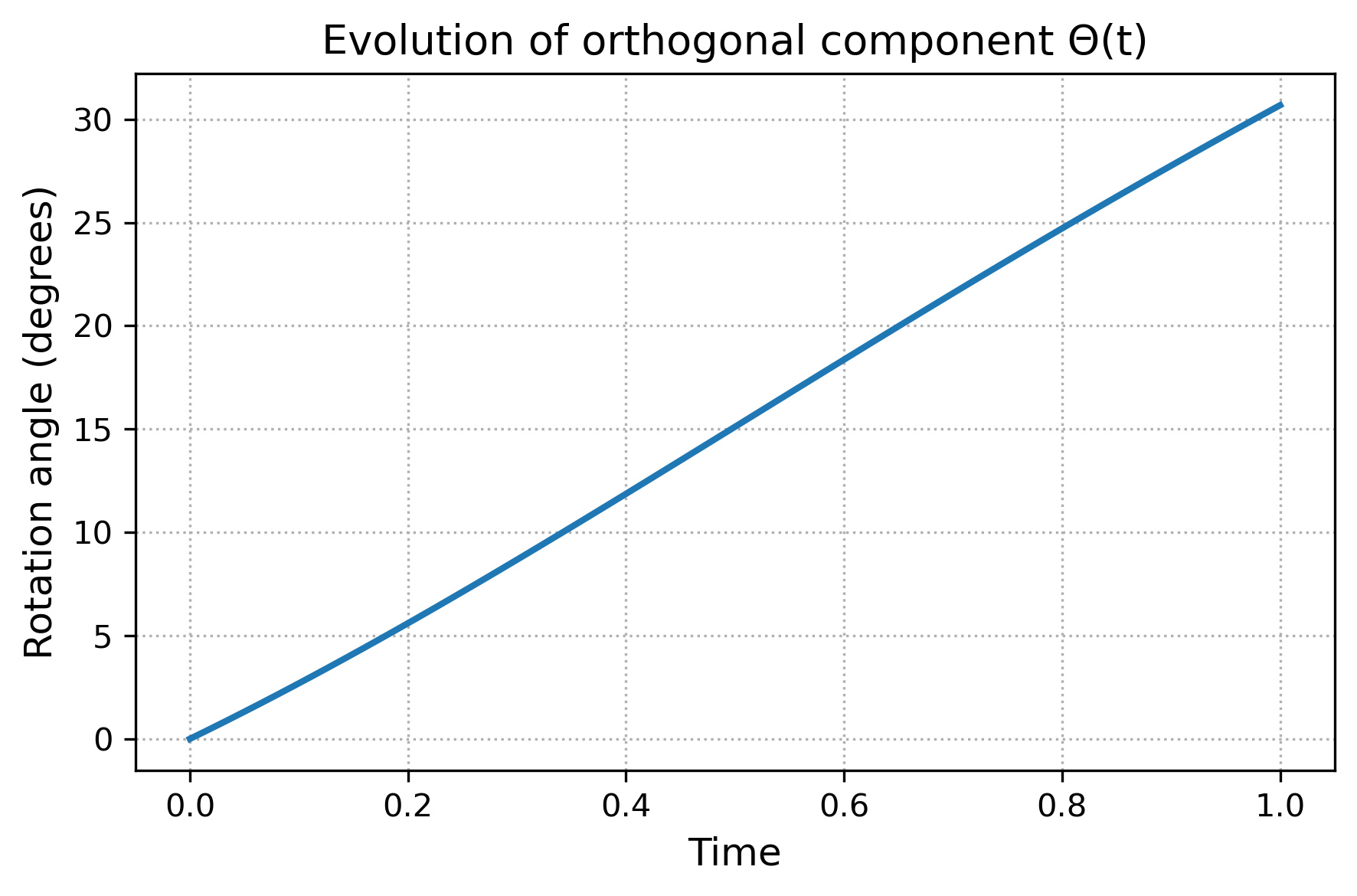}
    \caption{Evolution of the orthogonal component $\Theta_t$ (angle reconstructed from $E_t$).}
    \label{fig:pure_angle}
\end{figure}

\paragraph{Discussion.}
This experiment demonstrates that the MBB formulation correctly captures a purely rotational
motion while maintaining stable convergence and negligible transport energy.
Compared with the classical BB interpolation, which compensates for the rotation through numerical smearing, the MBB approach provides a clear separation between rigid motion
and mass redistribution at the discrete level.
This indicates that the additional orthogonal degree of freedom provides
a geometrically consistent mechanism for handling global isometries
within the dynamical optimal transport framework.

\vspace{1em}
\subsection{Effective Mass Transport}
\label{subsec:effective-transport}

\paragraph{Test distributions.}
We next consider a configuration designed to suppress any meaningful global orthogonal motion.
The initial density $\rho_0$ is again given by the double-crescent distribution,
while the target density $\rho_1$ is constructed by smoothly stretching the crescents
into two approximately horizontal elongated structures.
This transformation induces a substantial redistribution of mass but does not correspond
to any rigid rotation or reflection of the domain.
The resulting pair $(\rho_0,\rho_1)$ therefore provides a benchmark in which
the optimal transport is expected to be purely advective.

\begin{figure}
    \centering
    \includegraphics[width=.9\linewidth]{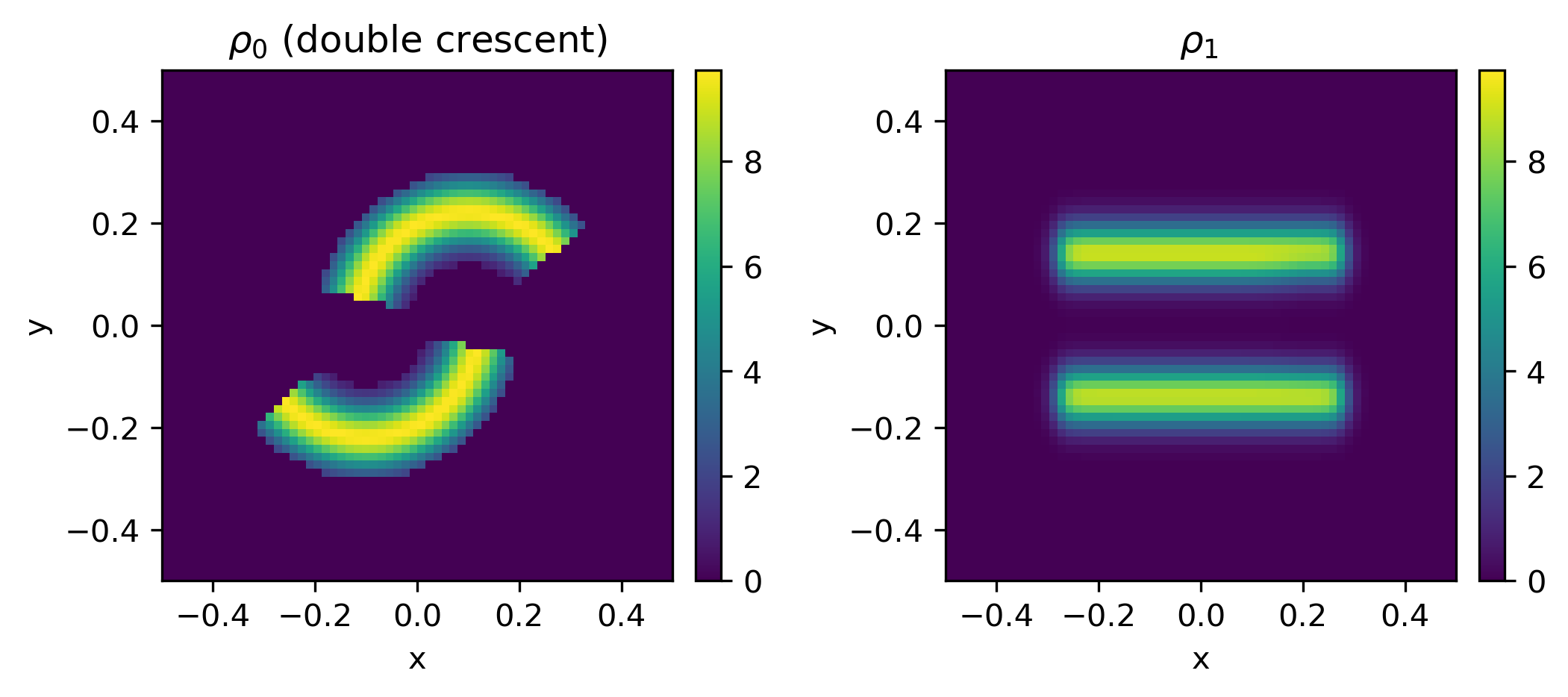}
    \caption{Effective transport test: source $\rho_0$ (double crescent) and target $\rho_1$
    obtained by horizontal stretching, yielding two elongated structures.}
    \label{fig:effective_rho0rhoT}
\end{figure}

\paragraph{Evolution of densities.}
In Figure~\ref{fig:ET_compare_MBB_BB} the top row displays the evolution of the interpolating density under the MBB dynamics, while the bottom row shows the corresponding interpolation
obtained with the classical Benamou--Brenier (BB) formulation.
In this setting, both formulations produce visually indistinguishable density evolutions:
the mass is transported and stretched smoothly into the target configuration without
exhibiting coherent rotation.
This indicates that, when no global orthogonal structure is present in the data,
the additional drift variable in the MBB formulation does not alter the transport mechanism.

\begin{figure}[H]
    \centering
    \includegraphics[width=1\linewidth]{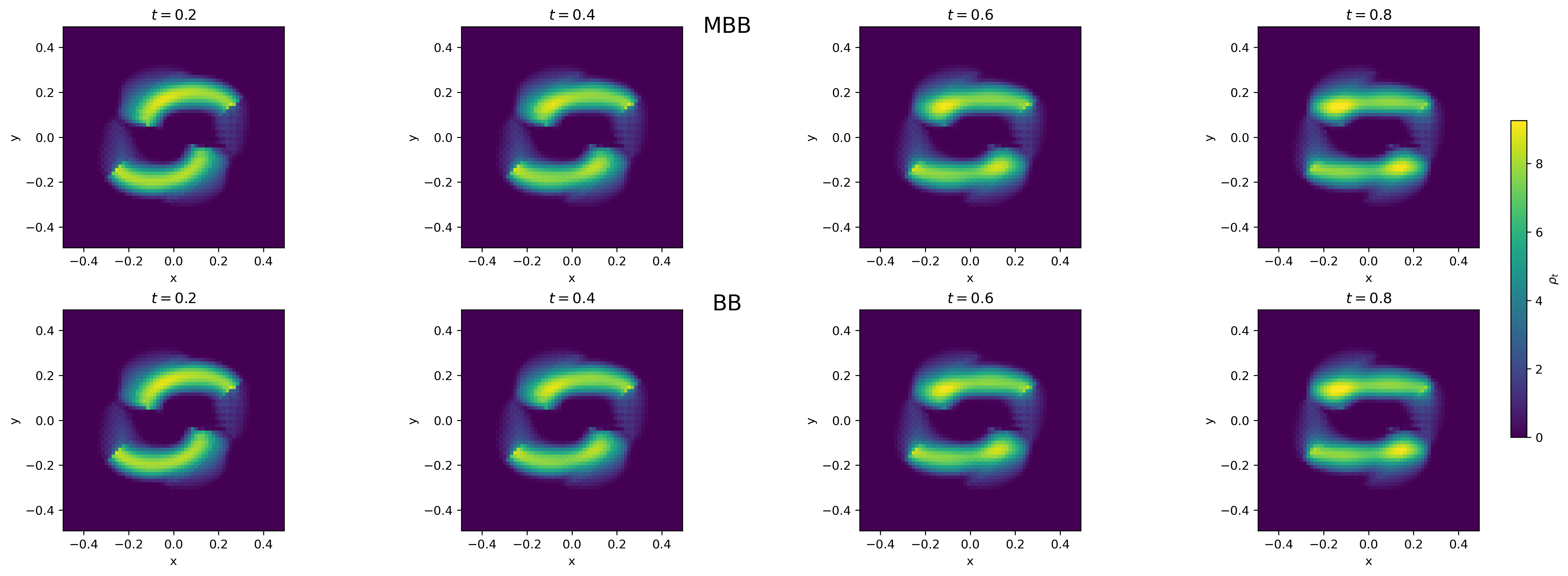}
 \caption{Evolution of $\rho_t$ for the Effective Transport experiment.}
    \label{fig:ET_compare_MBB_BB}  
\end{figure}

\paragraph{Convergence behaviour.}
The convergence histories and kinetic energy decay for both formulations
are reported in Figs.~\ref{fig:effective_residuals} and~\ref{fig:effective_momentum}.
The residual curves are nearly identical, and the total kinetic energies
evolve in close agreement throughout the iterations.
This behaviour indicates that, in the absence of an exploitable orthogonal component,
the MBB formulation reduces numerically to the classical Benamou--Brenier dynamics.

\begin{figure}[H]
    \centering
    \begin{minipage}[t]{0.50\linewidth}
        \centering
        \includegraphics[width=\linewidth]{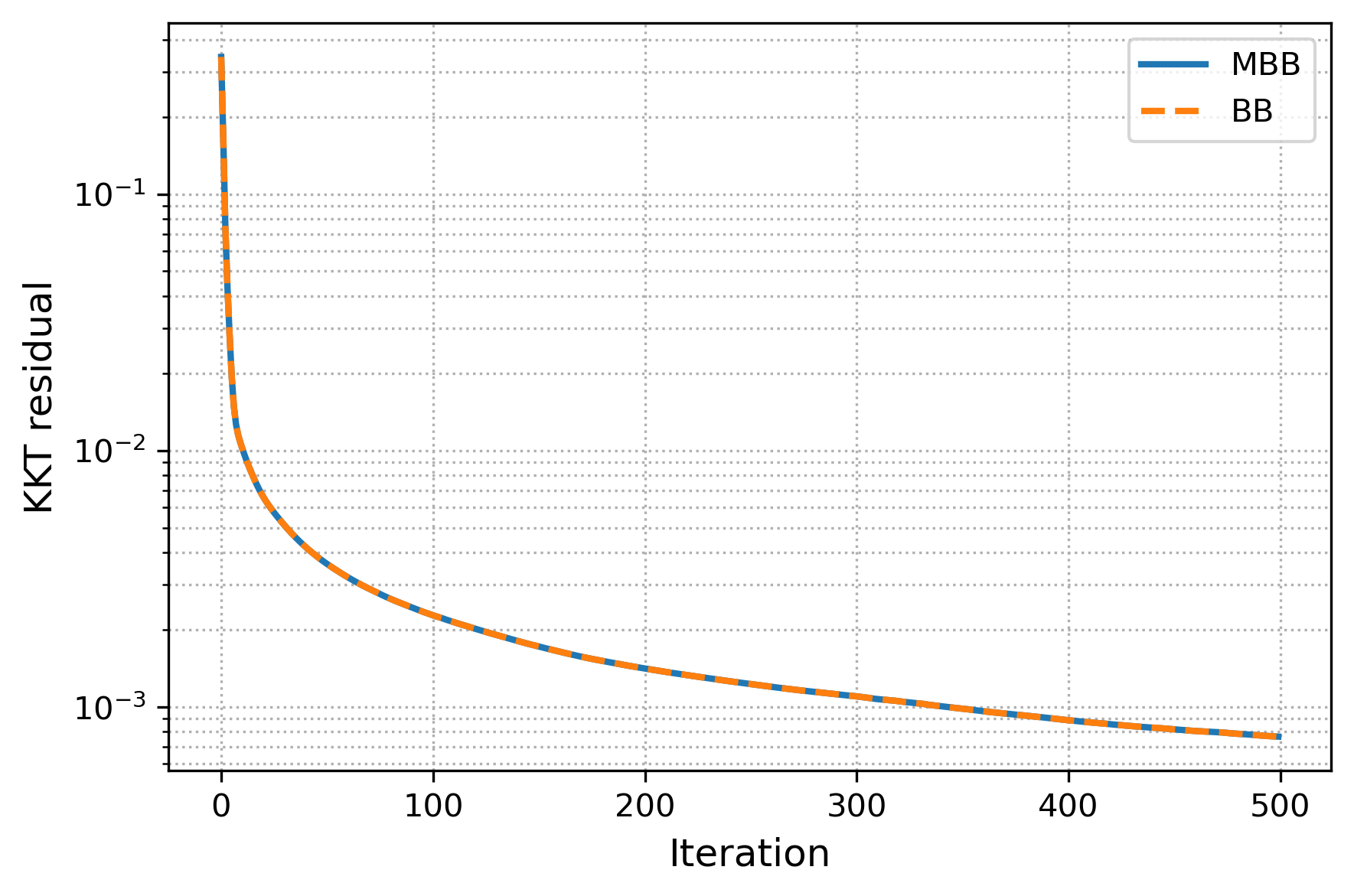}
    \caption{Convergence of residuals in the effective transport experiment for both formulations.}
    \label{fig:effective_residuals}
    \end{minipage}
    \hfill
    \begin{minipage}[t]{0.48\linewidth}
        \centering
        \includegraphics[width=\linewidth]{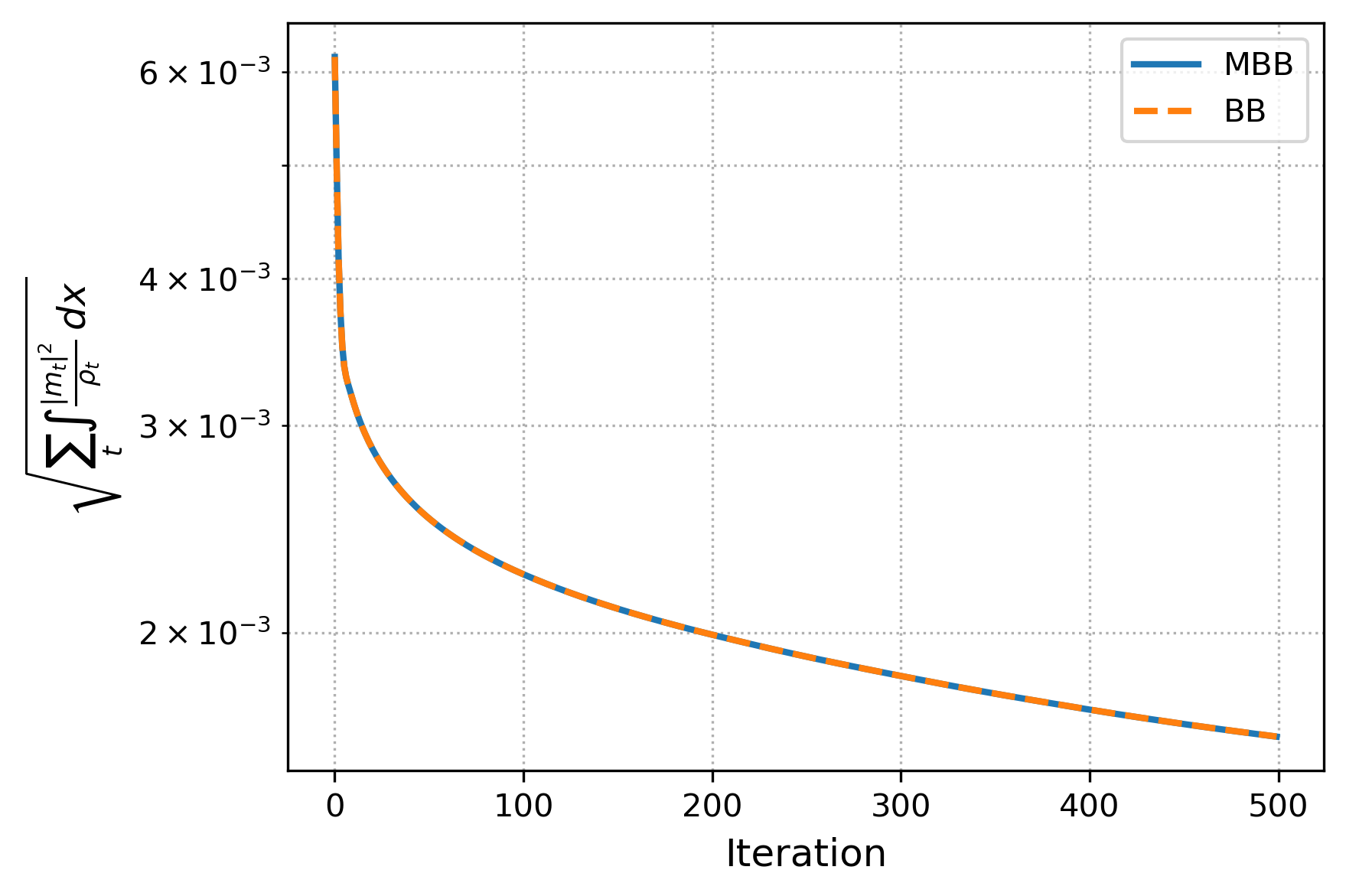}
    \caption{Total kinetic energy per iteration in the effective transport case.}
    \label{fig:effective_momentum}
    \end{minipage}
\end{figure}

\paragraph{Orthogonal component.}
Figure~\ref{fig:effective_angle} displays the evolution of the angle associated with
the orthogonal component $\Theta_t$.
The angle remains extremely small throughout the time interval,
confirming that no significant rotation is learned by the MBB scheme.
Minor fluctuations can be attributed to discretization and numerical tolerance,
and do not result in a visible effect on the density evolution.

\begin{figure}[H]
    \centering
    \includegraphics[width=.6\linewidth]{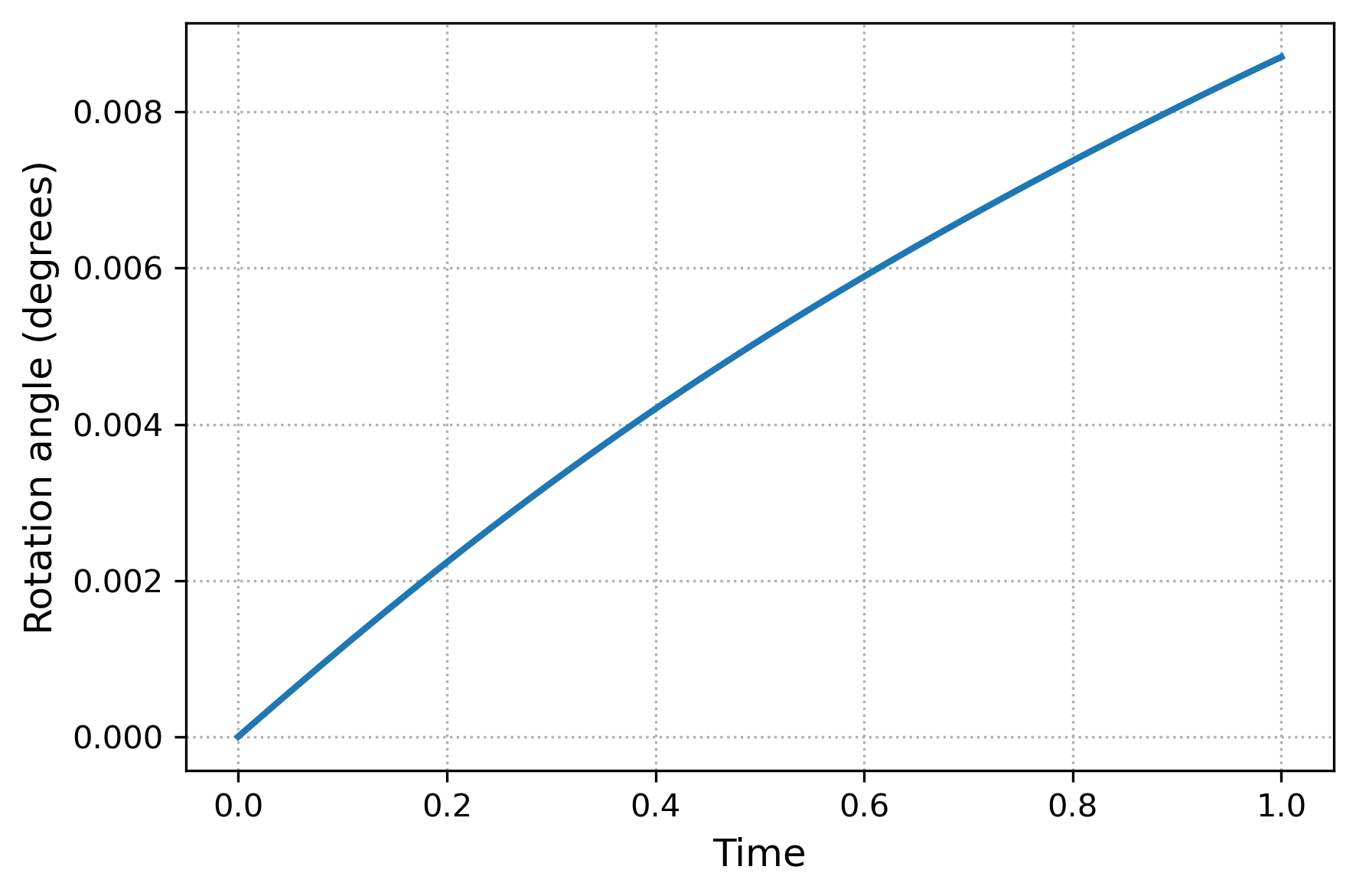}
    \caption{Evolution of the orthogonal component $\Theta_t$ (angle reconstructed from $E_t$) in the effective transport case.}
    \label{fig:effective_angle}
\end{figure}

\paragraph{Discussion.}
This experiment demonstrates that the Modified Benamou--Brenier formulation
closely reproduces the classical Benamou--Brenier dynamics when the optimal transport
does not involve a global orthogonal motion. This indicates that the MBB framework extends classical optimal transport
without altering its behaviour in purely advective regimes.

\subsection{Mixed Rotation–Deformation Transport}
\label{subsec:mixed}

\paragraph{Test distribution.}
The third experiment considers a configuration combining global rotation
with non-uniform deformation.
The source density is the double-crescent distribution introduced earlier,
and the target density is obtained by applying a rigid rotation of approximately $25^\circ$
together with a mild anisotropic stretching.
This setting is designed to test whether the MBB formulation can disentangle
rigid and advective components when both are simultaneously present.

\begin{figure}[H]
    \centering
    \includegraphics[width=.8\linewidth]{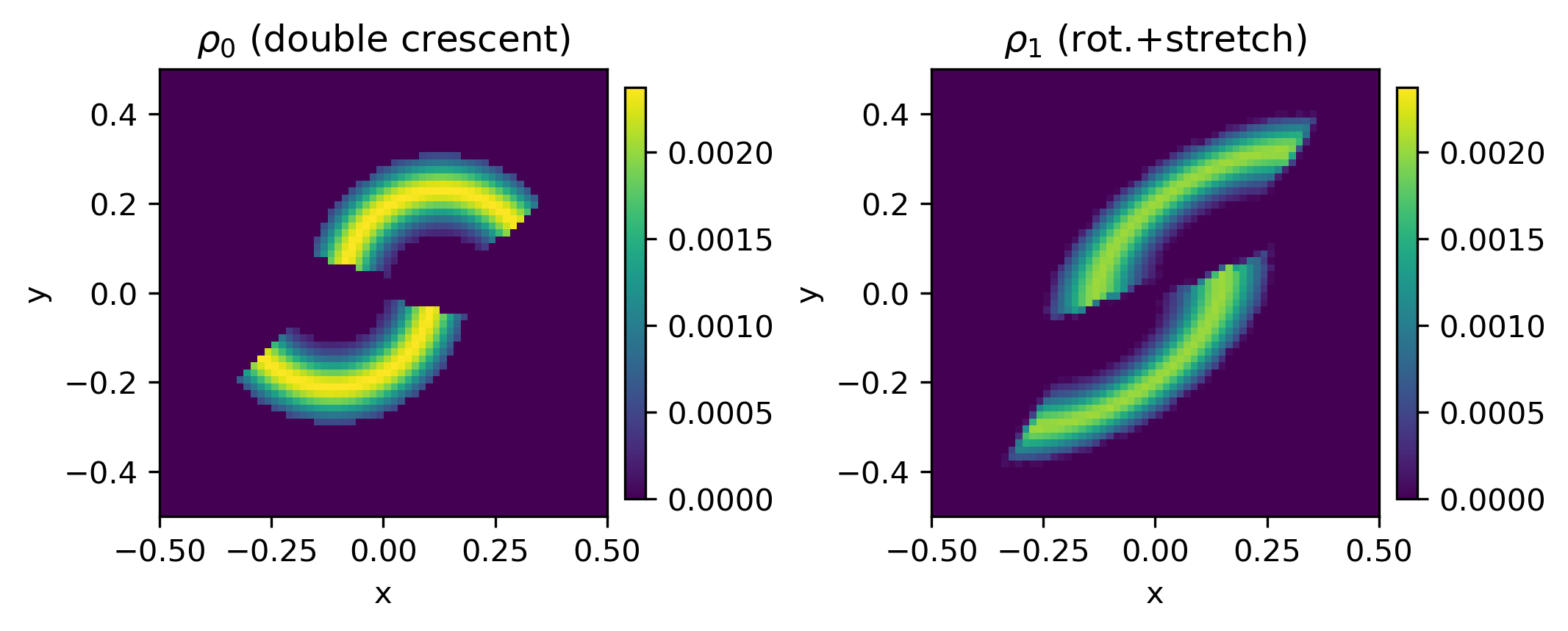}
    \caption{Mixed rotation–deformation test: source $\rho_0$ and target $\rho_T$.}
    \label{fig:mixed_rho0rhoT}
\end{figure}

\paragraph{Evolution of densities.} Figure~\ref{fig:MT_compare_rows} displays the evolution of the interpolating density $\rho_t$
under the MBB and BB formulations, respectively.
While both schemes transport mass toward the target configuration,
the MBB interpolation appears to preserve sharper structures and improved geometric coherence.
In contrast, the classical BB evolution shows mild diffusion,
particularly around intermediate time slices where rotational motion
must be approximated through potential flows.

\begin{figure}[H]
    \centering
    \includegraphics[width=1\linewidth]{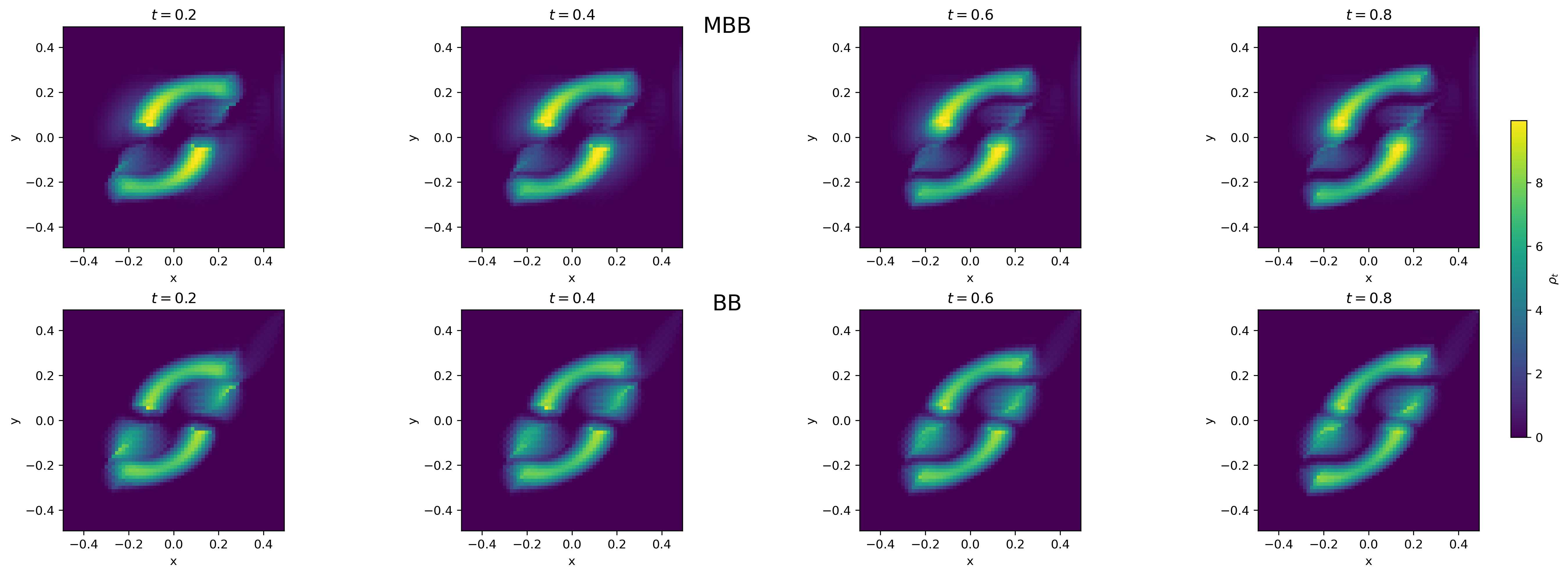}
    \caption{Evolution of $\rho_t$ for the mixed motion experiment.}
    \label{fig:MT_compare_rows}
\end{figure}

\paragraph{Convergence behaviour.}
The decay of the Fixed-point residuals is shown in Fig.~\ref{fig:mixed_residuals}.
Both formulations converge within the prescribed tolerance.
The MBB residual stabilizes at slightly lower values,
suggesting that the availability of an explicit orthogonal component
facilitates the minimization of the action functional
when rigid motion is present in the data.

\begin{figure}[H]
    \centering
    \begin{minipage}[t]{0.50\linewidth}
        \centering
        \includegraphics[width=\linewidth]{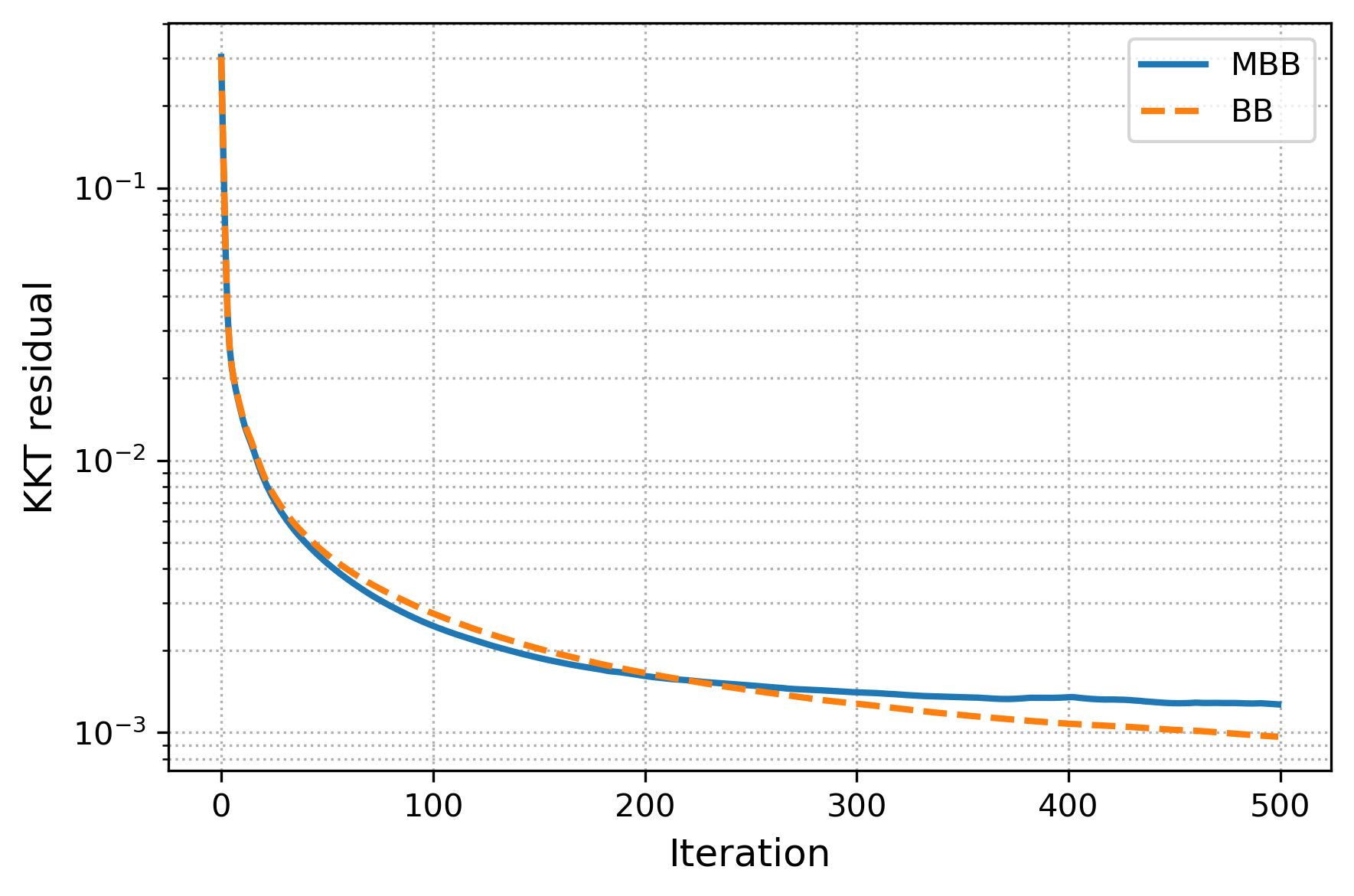}
    \caption{Convergence of KKT residuals in the mixed rotation–deformation experiment.}
    \label{fig:mixed_residuals}
    \end{minipage}
    \hfill
    \begin{minipage}[t]{0.48\linewidth}
        \centering
        \includegraphics[width=\linewidth]{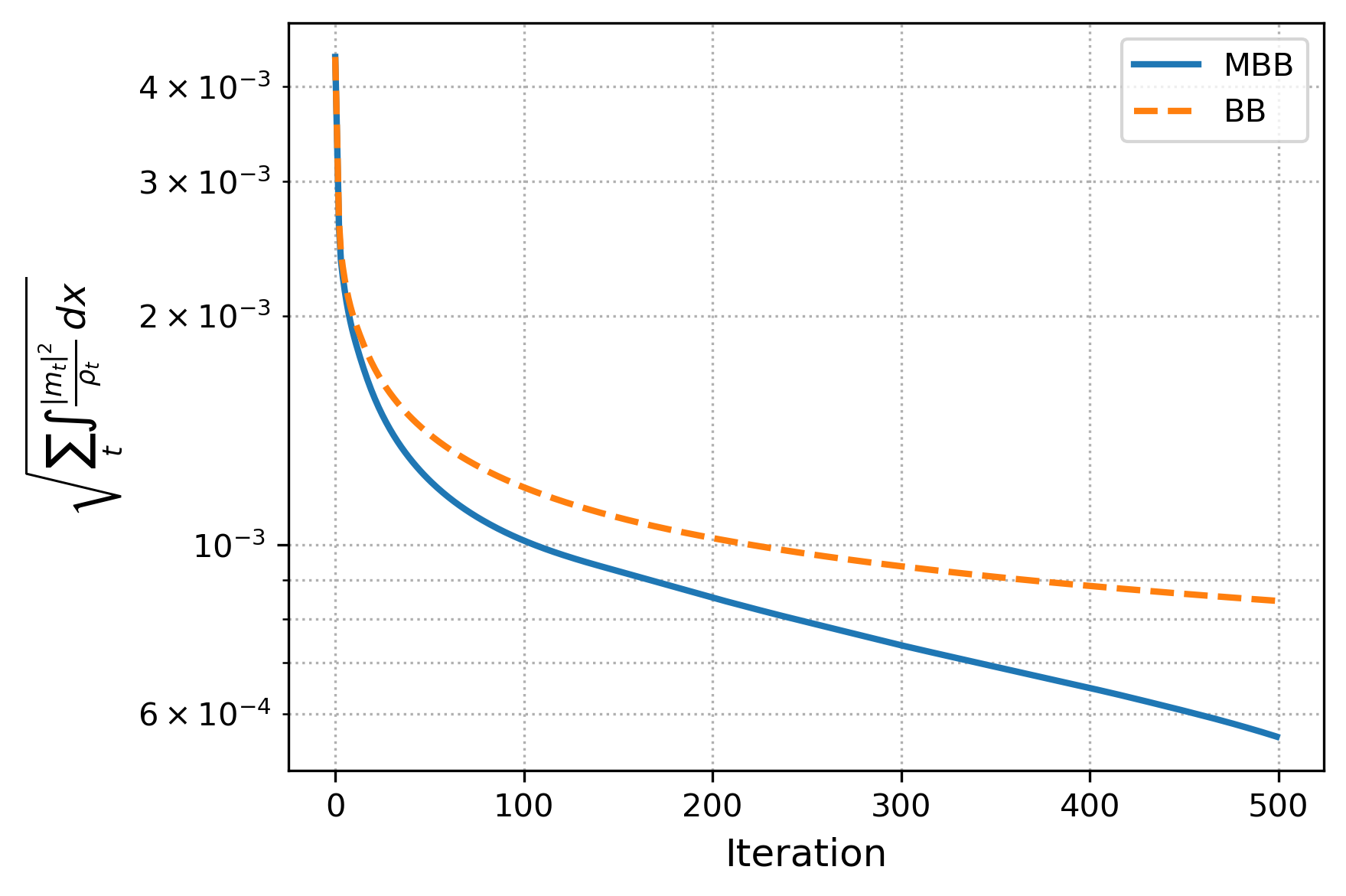}
    \caption{Total kinetic energy per iteration in the mixed motion case.}
    \label{fig:mixed_momentum}
    \end{minipage}
\end{figure}

\paragraph{Momentum and orthogonal component.}
Figure~\ref{fig:mixed_momentum} reports the total kinetic energy
associated with the advective momentum.
The MBB formulation consistently yields lower kinetic energy than BB, suggesting that part of the motion is absorbed by the orthogonal drift.
The corresponding orthogonal path $\Theta_t$,
reconstructed from the optimized skew-symmetric field $E_t$,
is shown in Fig.~\ref{fig:mixed_angle}.
The recovered angle increases smoothly in time and reaches values around $14^\circ$,
identifying a substantial rotational component, while the remaining motion is represented by the advective transport field.

\begin{figure}[H]
    \centering
     \includegraphics[width=.6\linewidth]{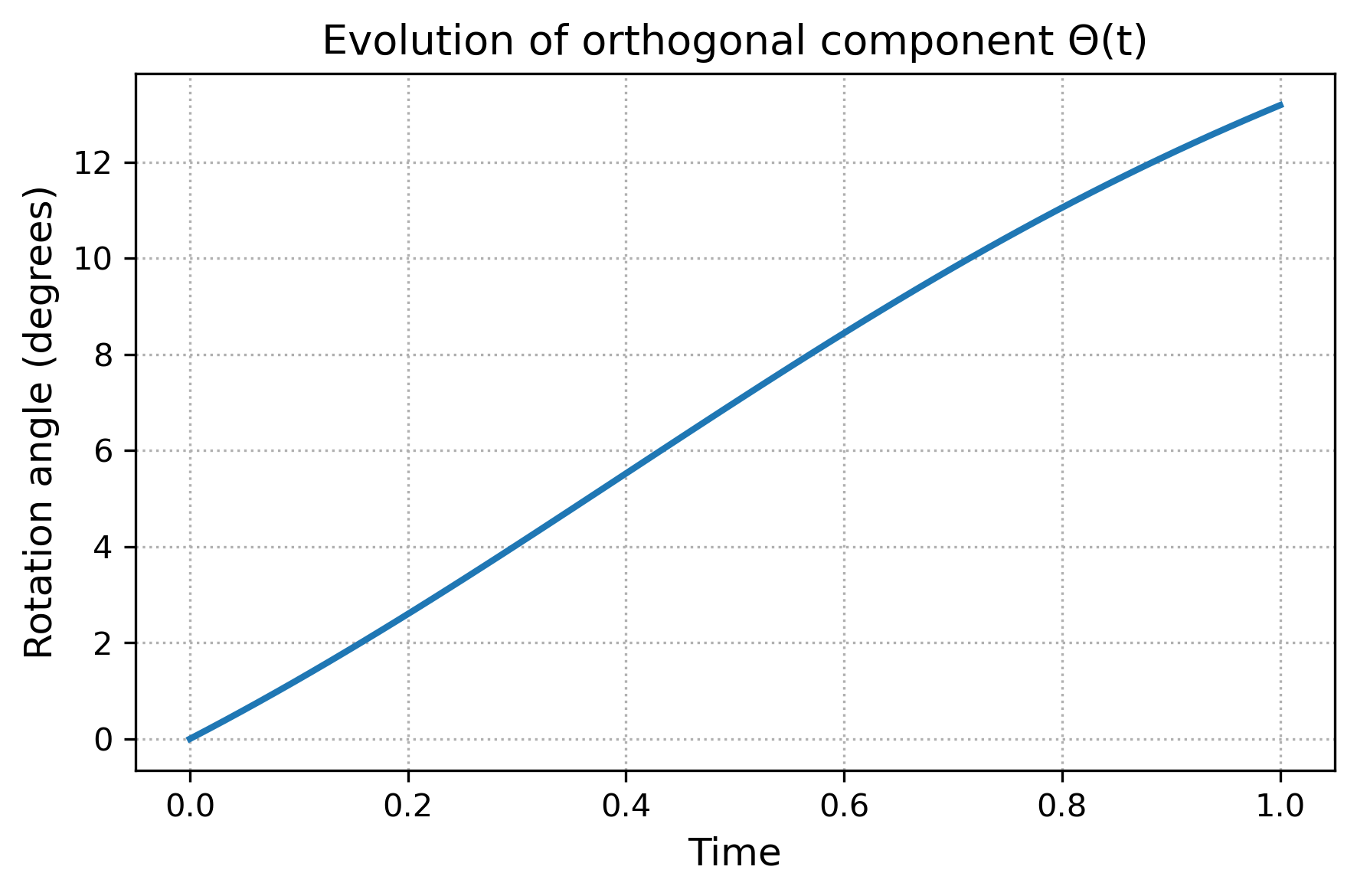}
    \caption{Evolution of the orthogonal component $\Theta_t$ (angle reconstructed from $E_t$).}
    \label{fig:mixed_angle}
\end{figure}

\paragraph{Discussion.}
In this mixed regime, the Modified Benamou–Brenier formulation
naturally decomposes the motion into a rigid rotational part
and a deformable transport component.
The additional skew-symmetric drift $E_t$ provides a compact and interpretable
representation of the global rotation,
while the advective velocity field accounts for the remaining deformation.
Compared with the classical BB formulation,
the MBB approach yields clearer density evolutions and reduced artificial diffusion,
without compromising convergence or mass conservation.
This experiment demonstrates the ability of MBB to extend classical optimal transport
to settings with intrinsic orthogonal structure,
while remaining consistent with BB in purely advective regimes.


\subsection{General remarks and reproducibility}
\label{subsec:generalremarkreproduc}
For each frozen outer field \(E^n\), the corresponding inner saddle problem is
approximated by performing a fixed number of PDHG iterations. Unless stated
otherwise, we use an inner iteration cap of five steps. The resulting inner
output is then used to update the skew-symmetric field, and all computations
reported in this section are run for a common outer horizon of \(500\)
iterations. This common iteration budget allows the MBB and BB formulations
to be compared over the same computational horizon.
During the computations, we monitor the outer iterate-difference
\[
    \|\rho^{n+1}-\rho^n\|_{L^2(I\times\Omega)}
    +
    \|m^{n+1}-m^n\|_{L^2(I\times\Omega)}
    +
    \|\phi^{n+1}-\phi^n\|_{H^1(I\times\Omega)}.
\]
This quantity is used as a diagnostic of stabilization rather than as the
stopping criterion for the experiments displayed here. Indeed, the MBB and BB
iterates need not reach a prescribed tolerance at the same outer iteration,
and the reported runs are therefore continued until the common outer cap.
The variables are initialized by setting
\[
    E^0=0,
    \qquad
    m^0=0,
\]
and taking \(\rho^0\) to be the linear interpolation between the prescribed
endpoint densities. The initial dual variable \(\phi^0\) is obtained from a
bootstrap solve of the dual update associated with this initial
density--momentum pair.

The primal--dual step sizes
\((\tau_\rho,\tau_m,\tau_\phi)\) are adapted numerically using an adaptive residual-balancing strategy following \cite{goldstein2013adaptive}. This adaptation strategy is used only in the implementation. In the local convergence analysis of
Sections~\ref{sec:abstractconv}--\ref{sec:conv}, the primal--dual step sizes are instead fixed and chosen to
satisfy the uniform stability condition Corollary \ref{cor:uniform-stability} and the outer step size \(\tau_E\) is chosen to satisfy Proposition \ref{prop:outer-inexact-descent} on the prescribed regularity
neighborhood. Numerically, the outer step size \(\tau_E\) is kept fixed within each
experiment.

In practice, the iterates typically enter a stable regime before the final
outer iteration. Once the primal--dual variables are close to equilibrium,
continued residual balancing and outer updates may produce small numerical
fluctuations without a visible improvement in the reconstructed transport.
Similar late-stage behaviour has been observed for adaptive primal--dual
methods when step sizes continue to evolve near equilibrium
\parencite[\S IV]{goldstein2013adaptive}.

The convergence analysis in Sections~\ref{sec:abstractconv}--\ref{sec:conv} uses a residual-based inner
acceptance rule with a summable tolerance sequence. The experiments implement
a practical finite-budget variant, based on a fixed inner iteration cap and a
fixed outer iteration horizon. Thus, the theoretical and computational
settings share the same alternating update structure but use different
controls on the inner accuracy.

\section{Discussion: Why only local convergence?}

The local convergence analysis developed in Sections \ref{sec:abstractconv}--\ref{sec:conv} assumes local regularity of the skew-symmetric field $E$. The following examples explain why such assumptions are unavoidable in general.

The examples are formulated in two dimensions on the ball \(\Omega=B_R(0)\subset\mathbb R^2\), although the same phenomena arise on \(\mathbb R^2\). Let $J\coloneqq
  \begin{pmatrix}
    0&-1\\
    1&0
  \end{pmatrix} $ and $ \mathfrak{so}(2)=\{\alpha J:\alpha\in\mathbb R\}$.
Since \(Jx\) is tangent to \(\partial B_R(0)\), the no-flux condition $(m+\rho E_t x)\cdot n=0$ is automatically satisfied for \(m=0\) and \(E_t=\alpha_tJ\).

\paragraph{1. Exact nonuniqueness for radial densities.}
Let \(\bar\rho(x)=\bar r(|x|)\) be a smooth radial probability density on \(\Omega\), and set $ \rho_0=\rho_1=\bar\rho$.
Consider the stationary path \(\rho_t=\bar\rho\) and choose \(m_t=0\). Then, for every \(\alpha\in L^2(0,1)\) and \(E_t=\alpha_tJ\), $K(\rho,m,E)=  \partial_t\rho+\nabla\cdot m+\nabla\cdot(\rho E_t x) =
  \nabla\cdot(\bar\rho\,\alpha_tJx)$.
Since \(\operatorname{tr}J=0\) and \(\nabla\bar\rho(x)\) is parallel to \(x\), while \(Jx\) is orthogonal to \(x\), we obtain
\begin{align*}
  \nabla\cdot(\bar\rho\,\alpha_tJx)
  =
  \alpha_t\nabla\bar\rho(x)\cdot Jx
  +
  \alpha_t\bar\rho(x)\operatorname{tr}J
  =0 .
\end{align*}
Thus \((\rho,m,E)\) is feasible for every \(\alpha\in L^2(0,1)\), while
\begin{align*}
  \mathcal J(\rho,m)
  =
  \frac12\int_0^1\int_\Omega\frac{|m|^2}{\rho}\dx\dt
  =0.
\end{align*}
Hence the minimizer is nonunique in the \(E\)-variable. In particular, the residual \(K(\rho,m,E)\) and the energy \(\mathcal J(\rho,m)\) do not control \(\|E\|_{L^2(0,1)}\). Therefore no estimate of the form
\begin{align*}
  \|E-E^\ast\|_{L^2(0,1)}
  \le
  C\bigl(\|K(\rho,m,E)\|+\text{optimality residuals}\bigr)
\end{align*}
can hold without an additional gauge condition or regularization. The corresponding linearized \(E\)-operator at \(\bar\rho\) is $A_{\bar\rho}\delta E  \coloneqq  \nabla\cdot(\bar\rho\,\delta E x)$.
The computation above gives
\begin{align*}
  \ker A_{\bar\rho}=\mathfrak{so}(2).
\end{align*}
Thus the linearized operator is not invertible in the \(E\)-direction at such configurations. In particular, no uniform local stability estimate with respect to \(E\) can hold.

\paragraph{2. A small perturbation destroys the exact kernel but creates ill-conditioning.}
The previous example is not merely a symmetric special case; it also indicates what happens near symmetric densities. Let \(\chi\in C^\infty_c(0,R)\) be nonzero and define, in polar coordinates \(x=(r,\theta)\),
\begin{align*}
  \rho_\varepsilon(r,\theta)
  \coloneqq
  \bar r(r)\bigl(1+\varepsilon\chi(r)\cos\theta\bigr),
  \qquad
  0<|\varepsilon|\ll1.
\end{align*}
For \(|\varepsilon|\) small enough, \(\rho_\varepsilon\) is positive and has the same total mass as \(\bar\rho\), since the angular average of \(\cos\theta\) vanishes. For \(E=\alpha J\), $A_{\rho_\varepsilon}(\alpha J) = \nabla\cdot(\rho_\varepsilon\alpha Jx) =  \alpha\nabla\rho_\varepsilon(x)\cdot Jx$. Since \(Jx=r e_\theta\) and
\[
\nabla\rho_\varepsilon = \partial_r\rho_\varepsilon e_r
+ r^{-1}\partial_\theta\rho_\varepsilon e_\theta,
\]
we have $\nabla\rho_\varepsilon(x)\cdot Jx =  \partial_\theta\rho_\varepsilon(r,\theta) = -\varepsilon\bar r(r)\chi(r)\sin\theta $.
Consequently,
\begin{align*}
  A_{\rho_\varepsilon}(\alpha J)
  =
  -\alpha\varepsilon\bar r(r)\chi(r)\sin\theta .
\end{align*}
Thus \(A_{\rho_\varepsilon}\) is injective on \(\mathfrak{so}(2)\) for every \(\varepsilon\neq0\), but its smallest singular value is of order \(|\varepsilon|\). Equivalently, the stability estimate has the form
\begin{align*}
  |\alpha|
  \le
  \frac{C}{|\varepsilon|}
  \|A_{\rho_\varepsilon}(\alpha J)\|.
\end{align*}
The constant blows up as \(\varepsilon\to0\).

This example shows that even when the kernel is trivial, the recovery of \(E\) can be ill-conditioned near symmetric configurations. Hence one should not expect stability estimates for \(E\) with constants that are uniform near such configurations.

\paragraph{3. Closed rotational loops and nonuniqueness of the time-dependent generator.}
We now point out a source of nonuniqueness arising from the time-dependent rotational gauge. Even if the instantaneous kernel is trivial for a nonradial density, the time-dependent field \(E_t\) is not determined by the endpoints and the action alone. Let \((\rho,m,E)\) be any feasible triple, and let \(R_t\in SO(2)\) be a smooth closed rotation loop satisfying $ R_0=R_1=I$. Define
\begin{align*}
  \widetilde\rho_t=(R_t)_\#\rho_t,\qquad
  \widetilde m_t(y)=R_t m_t(R_t^\top y),\qquad
  \widetilde E_t=\dot R_tR_t^\top+R_tE_tR_t^\top .
\end{align*}
Then \(\widetilde\rho_0=\rho_0\) and \(\widetilde\rho_1=\rho_1\). Moreover, the transformed triple satisfies
\begin{align*}
  \partial_t\widetilde\rho+\nabla\cdot\widetilde m
  +\nabla\cdot(\widetilde\rho\,\widetilde E_t y)
  =0.
\end{align*}
The action is unchanged:
\begin{align*}
  \frac12\int_0^1\int_\Omega
  \frac{|\widetilde m_t(y)|^2}{\widetilde\rho_t(y)}
  \dy\dt
  =
  \frac12\int_0^1\int_\Omega
  \frac{|m_t(x)|^2}{\rho_t(x)}
  \dx\dt .
\end{align*}
Thus a closed rotational loop changes the intermediate density path and the generator \(E_t\), but it does not change the endpoints or the value of the action. In particular, take any nonradial density \(\rho_\varepsilon\) and set $\rho_0=\rho_1=\rho_\varepsilon$. The stationary path \(\rho_t=\rho_\varepsilon\), \(m_t=0\), \(E_t=0\), has zero action. However, for any closed rotation loop \(R_t\), the path
\begin{align*}
  \widetilde\rho_t=(R_t)_\#\rho_\varepsilon,\qquad
  \widetilde m_t=0,\qquad
  \widetilde E_t=\dot R_tR_t^\top
\end{align*}
also has zero action and the same endpoints. Hence even a nonradial density, for which the instantaneous kernel may be trivial, does not determine a unique time-dependent field \(E_t\).

\paragraph{Conclusion.} The first example exhibits exact nonuniqueness through the existence of a nontrivial kernel of the linearized \(E\)-operator. The second shows that this kernel may disappear under arbitrarily small perturbations, while the associated stability constants become arbitrarily large. The third shows that, even when the instantaneous kernel is trivial, the time-dependent generator still has a gauge-type nonuniqueness.

Consequently, one cannot expect global identification or a uniform stability estimate with respect to a fixed representative \(E^\star\) in the unregularized formulation. These examples therefore motivate the use of a local saddle branch around a regular reference solution, as in Sections~\ref{sec:abstractconv}--\ref{sec:conv}. They do not exclude convergence to a set of equivalent stationary solutions.

\printbibliography

@article{shen2026penalty,
  title={Penalty-Based First-Order Methods for Bilevel Optimization with Minimax and Constrained Lower-Level Problems},
  author={Shen, Yiyang and He, Yutian and Wang, Weiran and Lin, Qihang},
  journal={arXiv preprint arXiv:2605.08006},
  year={2026}
}

@article{solla2026optimistic,
  title={Optimistic Bilevel Optimization with Composite Lower-Level Problem},
  author={Solla, Mattia and Royset, Johannes O},
  journal={arXiv preprint arXiv:2602.05417},
  year={2026}
}

@inproceedings{yao2025overcoming,
  title={Overcoming lower-level constraints in bilevel optimization: A novel approach with regularized gap functions},
  author={Yao, Wei and Yin, Haian and Zeng, Shangzhi and Zhang, Jin},
  booktitle={International Conference on Learning Representations},
  volume={2025},
  pages={55516--55549},
  year={2025}
}

@article{jiang2024primal,
  title={A Primal-Dual-Assisted Penalty Approach to Bilevel Optimization with Coupled Constraints},
  author={Jiang, Liuyuan and Xiao, Quan and Tenorio, Victor and Real-Rojas, Fernando and Marques, Antonio G and Chen, Tianyi},
  journal={Advances in Neural Information Processing Systems},
  volume={37},
  pages={95026--95066},
  year={2024}
}

@book{rockafellar1997convex,
  title={Convex analysis},
  author={Rockafellar, R Tyrrell},
  volume={28},
  year={1997},
  publisher={Princeton University press}
}

@book{nocedal2006numerical,
  title={Numerical optimization},
  author={Nocedal, Jorge and Wright, Stephen J},
  year={2006},
  publisher={Springer}
}

@book{dontchev2009implicit,
  title={Implicit functions and solution mappings},
  author={Dontchev, Asen L and Rockafellar, R Tyrrell},
  volume={543},
  year={2009},
  publisher={Springer}
}

@book{villani2009optimal,
  title={Optimal transport: old and new},
  author={Villani, C{\'e}dric},
  volume={338},
  year={2009},
  publisher={Springer}
}

@article{benamou2000computational,
  title={A computational fluid mechanics solution to the Monge-Kantorovich mass transfer problem},
  author={Benamou, Jean-David and Brenier, Yann},
  journal={Numerische Mathematik},
  volume={84},
  number={3},
  pages={375--393},
  year={2000},
  publisher={Springer-Verlag Berlin/Heidelberg}
}

@article{dowson1982frechet,
  title={The Fr{\'e}chet distance between multivariate normal distributions},
  author={Dowson, DC and Landau, BV666017},
  journal={Journal of multivariate analysis},
  volume={12},
  number={3},
  pages={450--455},
  year={1982},
  publisher={Elsevier}
}

@inproceedings{alvarez2019towards,
  title={Towards optimal transport with global invariances},
  author={Alvarez-Melis, David and Jegelka, Stefanie and Jaakkola, Tommi S},
  booktitle={The 22nd International Conference on Artificial Intelligence and Statistics},
  pages={1870--1879},
  year={2019},
  organization={PMLR}
}

@article{peyre2019computational,
  title={Computational optimal transport: With applications to data science},
  author={Peyr{\'e}, Gabriel and Cuturi, Marco and others},
  journal={Foundations and Trends{\textregistered} in Machine Learning},
  volume={11},
  number={5-6},
  pages={355--607},
  year={2019},
  publisher={Now Publishers, Inc.}
}

@book{horn1994topics,
  title={Topics in matrix analysis},
  author={Horn, Roger A and Johnson, Charles R},
  year={1994},
  publisher={Cambridge University press}
}

@article{kloeckner2010geometric,
  title={A geometric study of {W}asserstein spaces: Euclidean spaces},
  author={Kloeckner, Beno{\^\i}t},
  journal={Annali della Scuola Normale Superiore di Pisa-Classe di Scienze},
  volume={9},
  number={2},
  pages={297--323},
  year={2010}
}

@book{burago2001course,
    AUTHOR = {Burago, Dmitri and Burago, Yuri and Ivanov, Sergei},
     TITLE = {A course in metric geometry},
    SERIES = {Graduate Studies in Mathematics},
    VOLUME = {33},
 PUBLISHER = {American Mathematical Society, Providence, RI},
      YEAR = {2001},
     PAGES = {xiv+415},
      ISBN = {0-8218-2129-6},
   MRCLASS = {53C23},
  MRNUMBER = {1835418},
MRREVIEWER = {Mario\ Bonk},
       DOI = {10.1090/gsm/033},
       URL = {https://doi.org/10.1090/gsm/033},
}

@book{do1992riemannian,
  title={Riemannian geometry},
  author={Do Carmo, Manfredo Perdigao and Flaherty Francis, J},
  volume={2},
  year={1992},
  publisher={Springer}
}

@incollection{gallot2004differential,
  title={Differential manifolds},
  author={Gallot, Sylvestre and Hulin, Dominique and Lafontaine, Jacques},
  booktitle={Riemannian Geometry},
  pages={1--49},
  year={2004},
  publisher={Springer}
}

@inproceedings{grave2019unsupervised,
  title={Unsupervised alignment of embeddings with {W}asserstein procrustes},
  author={Grave, Edouard and Joulin, Armand and Berthet, Quentin},
  booktitle={The 22nd International Conference on Artificial Intelligence and Statistics},
  pages={1880--1890},
  year={2019},
  organization={PMLR}
}

@article{burger2025covariance,
  title={Covariance-modulated optimal transport and gradient flows},
  author={Burger, Martin and Erbar, Matthias and Hoffmann, Franca and Matthes, Daniel and Schlichting, Andr{\'e}},
  journal={Archive for Rational Mechanics and Analysis},
  volume={249},
  number={1},
  pages={},
  year={2025},
  publisher={Springer}
}

@book{ambrosio2005gradient,
    AUTHOR = {Ambrosio, Luigi and Gigli, Nicola and Savar\'{e}, Giuseppe},
     TITLE = {Gradient flows in metric spaces and in the space of
              probability measures},
    SERIES = {Lectures in Mathematics ETH Z\"{u}rich},
   EDITION = {Second},
 PUBLISHER = {Birkh\"{a}user Verlag, Basel},
      YEAR = {2008},
     PAGES = {x+334},
      ISBN = {978-3-7643-8721-1},
   MRCLASS = {49-02 (28A33 35K55 35K90 49Q20 60B05)},
  MRNUMBER = {2401600},
MRREVIEWER = {Pietro\ Celada},
}

@inproceedings{adamo2025depth,
  title={An in depth look at the Procrustes-{W}asserstein distance: properties and barycenters},
  author={Adamo, Davide and Corneli, Marco and Vuillien, Manon and Vila, Emmanuelle},
  booktitle={International Conference on Machine Learning},
  pages={444--459},
  year={2025},
  organization={PMLR}
}

@article{chambolle2011first,
  title={A first-order primal-dual algorithm for convex problems with applications to imaging},
  author={Chambolle, Antonin and Pock, Thomas},
  journal={Journal of mathematical imaging and vision},
  volume={40},
  number={1},
  pages={120--145},
  year={2011},
  publisher={Springer}
}

@article{courty2016optimal,
  title={Optimal transport for domain adaptation},
  author={Courty, Nicolas and Flamary, R{\'e}mi and Tuia, Devis and Rakotomamonjy, Alain},
  journal={IEEE transactions on pattern analysis and machine intelligence},
  volume={39},
  number={9},
  pages={1853--1865},
  year={2016},
  publisher={IEEE}
}

@inproceedings{arjovsky2017wasserstein,
  title={{W}asserstein generative adversarial networks},
  author={Arjovsky, Martin and Chintala, Soumith and Bottou, L{\'e}on},
  booktitle={International conference on machine learning},
  pages={214--223},
  year={2017},
  organization={PMLR}
}

@article{chambolle2016ergodic,
  title={On the ergodic convergence rates of a first-order primal--dual algorithm},
  author={Chambolle, Antonin and Pock, Thomas},
  journal={Mathematical Programming},
  volume={159},
  number={1},
  pages={253--287},
  year={2016},
  publisher={Springer}
}

@article{elamvazhuthi2023dynamical,
  title={Dynamical optimal transport of nonlinear control-affine systems},
  author={Elamvazhuthi, Karthik and Liu, Siting and Li, Wuchen and Osher, Stanley},
  journal={Journal of Computational Dynamics},
  volume={10},
  number={4},
  pages={425--449},
  year={2023},
  publisher={Journal of Computational Dynamics}
}

@article{cuturi2013sinkhorn,
  title={Sinkhorn distances: Lightspeed computation of optimal transport},
  author={Cuturi, Marco},
  journal={Advances in neural information processing systems},
  volume={26},
  year={2013}
}

@article{otto2001geometry,
  title={The geometry of dissipative evolution equations: The porous medium equation},
  author={Otto, F},
  journal={Communications in Partial Differential Equations},
  volume={26},
  number={1-2},
  pages={101--174},
  year={2001},
  publisher={Marcel Dekker Inc.}
}

@article{chen2016optimal,
  title={Optimal transport over a linear dynamical system},
  author={Chen, Yongxin and Georgiou, Tryphon T and Pavon, Michele},
  journal={IEEE Transactions on Automatic Control},
  volume={62},
  number={5},
  pages={2137--2152},
  year={2016},
  publisher={IEEE}
}

@book{agrachev2013control,
    AUTHOR = {Agrachev, Andrei A. and Sachkov, Yuri L.},
     TITLE = {Control theory from the geometric viewpoint},
    SERIES = {Encyclopaedia of Mathematical Sciences},
    VOLUME = {87},
      NOTE = {Control Theory and Optimization, II},
 PUBLISHER = {Springer-Verlag, Berlin},
      YEAR = {2004},
     PAGES = {xiv+412},
      ISBN = {3-540-21019-9},
   MRCLASS = {93-02 (49-02 49K15 93B27 93C15)},
  MRNUMBER = {2062547},
MRREVIEWER = {Kevin\ A.\ Grasse},
       DOI = {10.1007/978-3-662-06404-7},
       URL = {https://doi.org/10.1007/978-3-662-06404-7},
}

@article{agrachev2009optimal,
  title={Optimal transportation under nonholonomic constraints},
  author={Agrachev, Andrei and Lee, Paul},
  journal={Transactions of the American Mathematical Society},
  volume={361},
  number={11},
  pages={6019--6047},
  year={2009}
}

@article{jacobs2019solving,
  title={Solving large-scale optimization problems with a convergence rate independent of grid size},
  author={Jacobs, Matt and L{\'e}ger, Flavien and Li, Wuchen and Osher, Stanley},
  journal={SIAM Journal on Numerical Analysis},
  volume={57},
  number={3},
  pages={1100--1123},
  year={2019},
  publisher={SIAM}
}

@article{goldstein2013adaptive,
  title={Adaptive primal-dual hybrid gradient methods for saddle-point problems},
  author={Goldstein, Tom and Li, Min and Yuan, Xiaoming and Esser, Ernie and Baraniuk, Richard},
  journal={arXiv preprint arXiv:1305.0546},
  year={2013}
}

@book{coddington1955theory,
  title={Theory of ordinary differential equations},
  author={Coddington, Earl A and Levinson, Norman},
  year={1955},
  publisher={McGraw-Hill New York}
}

@inproceedings{pock2011diagonal,
  title={Diagonal preconditioning for first order primal-dual algorithms in convex optimization},
  author={Pock, Thomas and Chambolle, Antonin},
  booktitle={2011 International Conference on Computer Vision},
  pages={1762--1769},
  year={2011},
  organization={IEEE}
}

@book{nesterov2018lectures,
  title={Lectures on convex optimization},
  author={Nesterov, Yurii and others},
  volume={137},
  year={2018},
  publisher={Springer}
}

@book{bonnans2013perturbation,
  title={Perturbation analysis of optimization problems},
  author={Bonnans, J Fr{\'e}d{\'e}ric and Shapiro, Alexander},
  year={2013},
  publisher={Springer Science \& Business Media}
}

@article{tseng2001convergence,
  title={Convergence of a block coordinate descent method for nondifferentiable minimization},
  author={Tseng, Paul},
  journal={Journal of optimization theory and applications},
  volume={109},
  number={3},
  pages={475--494},
  year={2001},
  publisher={Springer}
}

@article{razaviyayn2013unified,
  title={A unified convergence analysis of block successive minimization methods for nonsmooth optimization},
  author={Razaviyayn, Meisam and Hong, Mingyi and Luo, Zhi-Quan},
  journal={SIAM Journal on Optimization},
  volume={23},
  number={2},
  pages={1126--1153},
  year={2013},
  publisher={SIAM}
}

@article{brandt1977multi,
  title={Multi-level adaptive solutions to boundary-value problems},
  author={Brandt, Achi},
  journal={Mathematics of computation},
  volume={31},
  number={138},
  pages={333--390},
  year={1977}
}

@article{nash2000multigrid,
  title={A multigrid approach to discretized optimization problems},
  author={Nash, Stephen G},
  journal={Optimization Methods and Software},
  volume={14},
  number={1-2},
  pages={99--116},
  year={2000},
  publisher={Taylor \& Francis}
}

@article{bertsekas2000gradient,
  title={Gradient convergence in gradient methods with errors},
  author={Bertsekas, Dimitri P and Tsitsiklis, John N},
  journal={SIAM Journal on Optimization},
  volume={10},
  number={3},
  pages={627--642},
  year={2000},
  publisher={SIAM}
}

@article{toukam2025procrustes,
  title={Procrustes {W}asserstein metric: A modified benamou-brenier approach with applications to latent gaussian distributions},
  author={Toukam, Kevine Meugang},
  journal={arXiv preprint arXiv:2503.16580},
  year={2025}
}

@article{eymard2000finite,
  title={Finite volume methods},
  author={Eymard, Robert and Gallou{\"e}t, Thierry and Herbin, Rapha{\`e}le},
  journal={Handbook of numerical analysis},
  volume={7},
  pages={713--1018},
  year={2000},
  publisher={Elsevier}
}

@book{leveque2002finite,
  title={Finite volume methods for hyperbolic problems},
  author={LeVeque, Randall J},
  volume={31},
  year={2002},
  publisher={Cambridge university press}
}

@book{ekeland1999convex,
  title={Convex analysis and variational problems},
  author={Ekeland, Ivar and Temam, Roger},
  year={1999},
  publisher={SIAM}
}
\end{document}